\documentclass[10pt]{article}
\usepackage{graphicx} 
\usepackage{natbib}
\usepackage{amsmath, amsthm, amssymb}
 \usepackage{fullpage}
\usepackage{dsfont}
\usepackage{color} 
\usepackage{graphicx}
\usepackage{color}
\usepackage{enumitem}

\newcommand{\bx}{\boldsymbol{x}}

\newcommand{\bnu}{\boldsymbol{\nu}}
\newcommand{\bX}{\boldsymbol{X}}
\newcommand{\bomega}{\boldsymbol{\omega}}
\newcommand{\kernel}{\mathcal{K}}
\newcommand{\smooth}{\mathcal{N}^{(h)}}

\DeclareMathOperator*{\argmin}{arg\,min}

\newcommand{\loss}{\mathcal{L}}
\newcommand{\lossh}{\mathcal{L}^{(h)}}

\newcommand{\lossnh}{\mathcal{L}^{(h,n)}}
\newcommand{\profilenh}{\widetilde{\mathcal{L}}^{(h,n)}}
\newcommand{\profileh}{\widetilde{\mathcal{L}}^{(h)}}
\newcommand{\btheta}{\boldsymbol{\theta}}
\newcommand{\bpsi}{\boldsymbol{\psi}}
\newcommand{\bSigma}{\boldsymbol{\Sigma}}
\newcommand{\bup}{\boldsymbol{\upsilon}^{(n)}}
\newcommand{\buptop}{\boldsymbol{\upsilon}^{(n)\top}}
\newcommand{\bUp}{\boldsymbol{\Upsilon}^{(n)}}
\newcommand{\bUptop}{\boldsymbol{\Upsilon}^{(n)}}
\newcommand{\bpi}{\boldsymbol{\pi}}

\newcommand{\bt}{\boldsymbol{t}}
\newcommand{\bmap}{\boldsymbol{\texttt{l}}}
\newcommand{\hpi}{\widehat{\bpi}^{(h,n)}}

\newcommand{\hpsi}{\widehat{\bpsi}^{(h,n)}}
\newcommand{\hpsipi}{\widehat{\bpsi}^{(h,n,\bpi)}}

\newcommand{\tpsipi}{ {\bpsi}^{(h,\bpi)}}
\newcommand{\tpsipis}{ {\bpsi}^{(h,\bpi^\star)}}

\newcommand{\score}{s}
\newcommand{\bscore}{\boldsymbol{\score}}
\newcommand{\effscore}{\tilde{\ell}}
\newcommand{\beffscore}{\boldsymbol{\effscore}}
\newcommand{\befffisher}{\boldsymbol{\Sigma}}
\newcommand{\beffscores}{\beffscore_{\bpi^\star,\bpsi^\star}}
\newcommand{\befffishers}{\befffisher_{\bpi^\star,\bpsi^\star}}

\newcommand{\befffishersinv}{\befffishers^{  -1}}
\newcommand{\noisy}{A}
\newcommand{\Gn}{\mathbb{G}_n}
\newcommand{\Pn}{\mathbb{P}_n}
\newtheorem{theorem}{Theorem} 
\newtheorem{assumptions}{Assumption}
\newtheorem{proposition}{Proposition}
\newtheorem{lemma}{Lemma}
\newtheorem{remark}{Remark}

\newcommand{\Eg}{\mathbb{E}_{g^\star}}
\newcommand{\KL}{\operatorname{KL}}
\newcommand{\bLambda}{\boldsymbol{\Lambda}}
\newcommand{\bG}{\boldsymbol{G}}
\title{Rates of Convergence of Maximum Smoothed Log-Likelihood  Estimators for Semi-Parametric Multivariate Mixtures}

 \usepackage{authblk} 
\author[A]{Marie Du Roy de Chaumaray}
\author[B]{Michael Levine}
\author[C]{Matthieu Marbac}
\affil[A]{Univ. Rennes, CNRS, IRMAR-UMR 6625, F-35000 Rennes, France}
\affil[B]{Department of Statistics, Purdue University, 150 N. University
St., West Lafayette, IN 47907, USA}
\affil[C]{Université Bretagne Sud, UMR CNRS 6205, LMBA, F-56000 Vannes, France.}

\begin{document}
\maketitle

\begin{abstract}
Theoretical guarantees are established for a standard estimator in a semi-parametric finite mixture model, where each component density is modeled as a product of univariate densities under a conditional independence assumption. The focus is on the estimator that maximizes a smoothed log-likelihood function, which can be efficiently computed using a majorization-minimization algorithm. This smoothed likelihood applies a nonlinear regularization operator defined as the exponential of a kernel convolution on the logarithm of each component density. Consistency of the estimators is demonstrated by leveraging classical M-estimation frameworks under mild regularity conditions. Subsequently, convergence rates for both finite- and infinite-dimensional parameters are derived by exploiting structural properties of the smoothed likelihood, the behavior of the iterative optimization algorithm, and a thorough study of the profile smoothed likelihood. This work provides the first rigorous theoretical guarantees for this estimation approach, bridging the gap between practical algorithms and statistical theory in semi-parametric mixture modeling.  \\
\textbf{Keywords}: Empirical process; 
Finite mixture model; 
Majorization-minimization algorithm; 
Rate of convergence; 
Semi-parametric mixture
\end{abstract}

\section{Introduction}
Finite mixture models are commonly used to perform clustering since they model heterogeneity in populations in a rather natural way \citep{McL00,Fruhwirth2019handbook}. 
In this framework, a standard definition of a cluster corresponds to the subset of individuals generated by the same mixture component (see \citet{hennig2010methods} and \citet{baudry2010combining} for several extensions, and \citet{hennig2015true} for a discussion on cluster definitions).
A finite mixture model is characterized by three main components: the number of mixture components, the mixing proportions, and the component-specific distributions.
The initial developments in this area focused on parametric mixture models, which posit a specific parametric form for the component distributions. Among them, the Gaussian mixture model \citep{Ban93}—in which each component is assumed to follow a Gaussian distribution—is widely regarded as the canonical example.
To address the bias that may arise from misspecified parametric assumptions, semi-parametric mixture models were subsequently introduced, relaxing the parametric constraints on the component distributions.
Among these semi-parametric approaches, two classes of models are particularly prominent (see \citep{ChauveauSurveys2015} for a comprehensive review).
The first, tailored to univariate data, assumes that the components are symmetric and belong to a common location family \citep{bordes2006semiparametric,hunter2007inference,butucea2014semiparametric}.
The second, applicable to multivariate data, assumes that the component distributions can be represented as products of univariate densities \citep{hall2003nonparametric}.

In this paper, we consider the semi-parametric mixture model that makes no assumptions on the component distribution except that it is defined as a product of univariate densities.  
Specifically, we assume that the observed data are a random vector $\bX_i = (X_{1}, \ldots, X_{J})^\top \in \mathcal{X}$ following a $K$-component semi-parametric mixture distribution with the density
\begin{equation}\label{eq:model}
 g_{\bpi,\bpsi}(\bx)=\sum_{k=1}^K \pi_k \psi_k(\bx),    
\end{equation}
where the density of component $k$ is defined as a product of $J$ univariate densities such that
\begin{equation}\label{eq:model2}
\psi_k(\bx) =\prod_{j=1}^J \psi_{k,j}(x_{j}),
\end{equation}
where $\bpi = (\pi_1, \dots, \pi_K)^\top$ denotes the vector of mixing proportions, and $\bpsi$ denotes the collection of univariate densities $\psi_{k,j}$, which constitute infinite-dimensional parameters. This model relies on a conditional independence assumption across variables given the latent component, which significantly simplifies estimation by reducing the complexity of the component distributions. This structural constraint often leads to improved empirical performance, as it limits the number of parameters to be estimated \citep{hand2001idiot}.
A classical setting in which this conditional independence assumption is justified is the repeated measures framework with random effects, where the subject-specific random effect is replaced by a component-specific latent effect. Hence, the model defined by \eqref{eq:model}–\eqref{eq:model2} has been widely applied in various domains, including behavioral sciences \citep{CloggBook1995}, econometrics \citep{HuJofEco2013}, and sociology \citep{HagenaarsBook2002}.

Several studies have addressed the issue of identifiability for the model defined by \eqref{eq:model}–\eqref{eq:model2}. \cite{kasahara2014non} show that the number of components $K$ is identifiable under the condition that, for at least two distinct indices $j$, the set of functions $\{\psi_{1,j}, \ldots, \psi_{K,j}\}$ is linearly independent. This in turn requires that $J \geq 2$. However, such conditions do not guarantee identifiability of the model parameters themselves—namely, the finite-dimensional parameters $\bpi$ and the infinite-dimensional component densities $\bpsi$—which calls for stronger assumptions. The first identifiability results for the parameters of the model \eqref{eq:model}–\eqref{eq:model2} were established by \cite{hall2003nonparametric} in the case of two-component mixtures (i.e., $K=2$). More generally, \cite{allmanAOS09} proved that the parameters are identifiable when the sets $\{\psi_{1,j}, \ldots, \psi_{K,j}\}$ are linearly independent for at least three distinct values of $j$, which implies that $J \geq 3$. Following the standard approach in the literature on such models, we adopt these identifiability assumptions throughout the paper (see Assumptions~\ref{ass:controlvariancetrue}).

Various theoretical results concerning the model  \eqref{eq:model}–\eqref{eq:model2} have been established by considering discretization of the data. In this context, sufficient conditions for the identifiability of the model parameters can be derived as consequences of the identifiability of latent class models for categorical data, as shown in \cite{allmanAOS09}. \cite{hettmansperger2000almost} proved the asymptotic normality of the maximum likelihood estimator of the mixing proportions when the original data are transformed into binary variables (see also \cite{cruz2004semiparametric}). \cite{kasahara2014non} introduced an estimator for the number of components based on discretized data. However, this estimator is only consistent for a lower bound of the true number of components. To address this limitation, \cite{kwon2021estimation} extended the approach by incorporating an integral operator, thereby obtaining a consistent estimator of the true number of components. More recently, \cite{du2024full} proposed a likelihood-based method using a discretization scheme in which the number of bins increases with the sample size. Their approach yields a consistent estimator of the number of components and additionally allows for variable selection. These discretization-based methods can be interpreted as projection techniques onto function spaces spanned by indicator functions. In a broader projection-based framework—but under the simplifying assumption that the univariate densities within each component are identical (i.e., $\psi_{k,1} = \ldots = \psi_{k,J}$)—\cite{bonhomme2016non} constructed a two-step estimator for the infinite-dimensional parameters of model \eqref{eq:model}–\eqref{eq:model2}. 
Still within the projection framework, but without imposing any assumptions beyond those required by \cite{allmanAOS09} for identifiability, \cite{bonhomme2016estimating} proposed an estimator based on multilinear decompositions of multiway arrays that is  both consistent and asymptotically normal. Despite its generality and mathematical elegance, this approach does not aim to estimate the component densities directly. Instead, it focuses on recovering the latent structure through low-rank tensor decompositions, which limits its usefulness in settings where inference on the component distributions themselves is required. In addition, the approach operates within a high-dimensional algebraic framework involving large multiway arrays, which can lead to substantial computational challenges.  

As an alternative to projection-based methods, likelihood-based approaches and their extensions can also be considered. In this context, one of the earliest strategies for estimating both the finite- and infinite-dimensional parameters was to employ an EM-like algorithm \cite{benaglia2009like}. While this algorithm is straightforward to implement, it lacks theoretical guarantees and does not satisfy the ascent property typically expected of EM procedures. To address these limitations, \cite{levine2011maximum} proposed a majorization–minimization (MM) algorithm (see \cite{hunter2004tutorial,lange2016mm}) that maximizes a smoothed version of the log-likelihood. The smoothed log-likelihood function corresponds to the standard log-likelihood evaluated at a smoothed version of each component density. When applied to a given density, the smoothing operator is defined as the exponential of the convolution between a kernel with bandwidth $h$ and the logarithm of that density. This algorithm enjoys a desirable descent property, is easy to implement, and is available through the \texttt{R} package \texttt{mixtools} \citep{mixtools}. Building upon this framework, \cite{zhu2016theoretical} reformulated the objective function in terms of a penalized, smoothed Kullback–Leibler divergence. They established a refined monotonicity property for the algorithm and proved the existence of a solution to the associated optimization problem. However, despite these algorithmic developments, no theoretical guarantees are currently available regarding the statistical properties  of the estimator produced by this approach.

In this paper, we provide theoretical guarantees for the estimator that maximizes the smoothed log-likelihood. We begin by establishing the consistency of both the finite- and infinite-dimensional parameter estimators (see Theorem~\ref{thm:consistence}), using standard arguments from M-estimation theory. We then focus on deriving convergence rates for these estimators. To this end, we first characterize the convergence rate of the infinite-dimensional estimators in terms of the sample size, the bandwidth parameter used for smoothing, and the convergence rate of the finite-dimensional estimators (see Theorem~\ref{thm:rateL1}). We then derive a bound on the convergence rate of the finite-dimensional parameters themselves (see Theorem~\ref{thm:norma}), thereby obtaining an overall control of the convergence rates for all estimators. The proof of Theorem~\ref{thm:rateL1} leverages structural properties of the objective function, notably its convexity when the finite-dimensional parameters are held fixed. It also relies on key algorithmic properties, in particular an inequality that links the value of the objective function at two successive iterations to the $L_1$-distance between the infinite-dimensional estimates obtained at those iterations (see Lemma~\ref{lem:inequality}, which can be viewed as an extension of \cite[Corollary 3.3]{zhu2016theoretical}). Theorem~\ref{thm:norma} is established through an analysis of the semi-parametric profile smoothed likelihood, where the infinite-dimensional parameters are treated as nuisance parameters and profiled out. In our setting, it turns out that the presence of these nuisance parameters degrades the standard convergence rate of the finite-dimensional estimators. To capture this phenomenon, we extend the quadratic expansion of the profile smoothed likelihood developed by \cite{murphy2000profile}, showing explicitly how the smoothing inherent in our objective function affects the asymptotic behavior (see Proposition~\ref{thm:convparam}).

The rest of the paper is organized as follows. 
Section~\ref{sec:framework} introduces the multivariate mixtures of products of univariate densities. 
Section~\ref{sec:computation} presents the estimation framework   using the smoothed log-likelihood. 
discusses computational aspects and establishes the consistency of the estimator. 
Section~\ref{sec:prop} gives properties on the mapping functions defined by the estimation algorithm. 
Section~\ref{sec:nprate} presents the theoretical convergence rates for the estimators of the component density based on the bandwidth, the sample size and the convergence rate of the estimator of the proportions.
Section~\ref{sec:rate} presents the theoretical convergence rates for the  estimator of the proportions. and thus the convergence rate for both the finite-dimensional parameters and the nonparametric component densities. 
Section~\ref{sec:simulation} illustrates the finite-sample performance of the proposed estimator through numerical simulations. 
Finally, Section~\ref{sec:conclusion} concludes with a discussion and potential directions for future work.

\section{Mixture model of products of univariate densities}\label{sec:framework}
Let $\bX=(X_{1},\ldots,X_{J})^\top$ be a random variable defined on  the space $\mathcal{X}=\mathcal{X}_1\times \ldots\times\mathcal{X}_J$ where each $\mathcal{X}_j$, $1\le j\le J$  is a compact.   We consider $\mathcal{G}_K$, the family of mixture models defined by
$$
\mathcal{G}_K=\{g_{\bpi,\bpsi}: \bpi\in \mathcal{S}_K,\, \bpsi\in\Psi_K(\mathcal{X})\},
$$
where $g_{\bpi,\bpsi}$ is the density of a $K$ component mixture model defined by \eqref{eq:model}-\eqref{eq:model2}, where  $\bpi=(\pi_1,\ldots,\pi_K)^\top$ is the finite-dimensional parameter composed of the vector of proportions defined on  the simplex 
$$\mathcal{S}_K=\left\{\bpi=(\pi_1,\ldots,\pi_K)\in\mathbb{R}^K,\, 0\leq \pi_k, \sum_{k=1}^K \pi_k=1\right\},$$ 
and where  $\bpsi=(\psi_{1,1}, \ldots,\psi_{K,1},\psi_{1,2},\ldots,\psi_{K,J})$ groups the infinite-dimensional parameters defined on $\Psi_K(\mathcal{X})$ with
$$
\Psi_K(\mathcal{X}) = [\Psi(\mathcal{X}_1)\times\ldots\times\Psi(\mathcal{X}_J)]^K.
$$
Let $L_2(\mathcal{X}_j)$ be a set of square integrable univariate density functions defined on $\mathcal{X}_j$
 In the following, we assume that $\mathcal{X}_j$ is compact and that the space of the univariate density functions of each component is defined as
$$
\Psi(\mathcal{X}_j) = \{ \psi_{k,j}\in L_2(\mathcal{X}_j),\, 0<\psi\leq C_1,   \|\ln \psi\|_{L_2}\leq C_2, \|(\ln \psi)^{''}\|_{L^{\infty}}\leq C_3\}.
$$ 
Here, we assume that $\mathcal{X}_j$ is compact  in order to avoid some additional technical arguments in the proof. However, at the end of the article, we explain how the results can be extended to the case where $\mathcal{X}_j$ is the real line. In addition, the arguments used in the proofs still hold if $\psi$ is equal to zero on a set of null Lebesgue measure. 

Any relabeling of the mixture components yields the same observed distribution, so the model parameters are only identifiable up to label switching.
To avoid these issues, we consider that the vector of proportions $\bpi$ belongs to the restriction of the simplex   $\mathcal{S}_K^r$ such that its elements are in non-decreasing order leading that
$$
\mathcal{S}_K^r=\left\{\bpi\in\mathcal{S}_K,  \pi_k\leq \pi_{k+1}\right\}.
$$
The set of all the parameters is defined as $$\Theta_K = \mathcal{S}_K^r \times \Psi_K(\mathcal{X}).$$

We assume that observations arise independently from a mixture model defined by \eqref{eq:model}-\eqref{eq:model2} with parameters $(\bpi^\star,\bpsi^\star)$ that belong to the parameter space $\Theta_K$ and we denote the true density $$g^\star:=g_{(\bpi^\star,\bpsi^\star)}.$$
We aim to give theoretical guarantees on an estimator of $(\bpi^\star,\bpsi^\star)$ that belongs to $\Theta_K $ and that is computed from a $n$-sample composed of $n$ independent copies of $\bX$ denoted by $\bX_1,\ldots,\bX_n$. 
To ensure the identifiability of the parameters $(\bpi^\star,\bpsi^\star)$, we assume that $g^\star$ satisfies the following assumptions. Indeed, as a direct consequence of Theorem~8 in \cite{allmanAOS09}, the following Assumptions~\ref{ass:controlvariancetrue} ensure that the parameters $(\bpi^\star,\bpsi^\star)$ are strictly identifiable up to label swapping. 
\begin{assumptions}\label{ass:controlvariancetrue}
\begin{enumerate}
\item Each proportion $\pi_k^\star$ is strictly positive. \label{ass:prop}
\item There exists at least three values of $j\in\{1,\ldots,J\}$ such that the set of functions $\{\psi_{1,j}^\star,\ldots,\psi_{K,j}^\star\}$ is composed of linearly independent functions. \label{ass:indptlin}
\item All the proportions are different: that $\pi_k^\star\neq\pi_\ell^\star$ if $k\neq\ell$.
\end{enumerate}
\end{assumptions}
By Theorem 8 in \cite{allmanAOS09}, the Assumptions.\ref{ass:controlvariancetrue}.1 and Assumptions.\ref{ass:controlvariancetrue}.2 ensure identifiability of the parameters up to label switching. To address this issue, we impose both the simplex constraint on the proportions $\mathcal{S}_K^r$ and an ordering constraint that ensures the proportions are pairwise distinct.  These restrictions allow us to simplify the notation throughout the paper, while still covering models with equal mixture proportions. In such cases, the label switching problem could alternatively be handled by imposing an ordering on the distributions of one observed variable—whose component densities are linearly independent—at the cost of losing the product structure of the parameter space for the component densities. Another approach would be to refrain from imposing any ordering constraints and instead define distances between true parameters and their estimators by minimizing over all possible permutations of component labels. However, both alternatives lead to heavier notation. For the sake of clarity and conciseness, we therefore chose to impose ordering constraints on the proportions which is a usual approach \cite{hunter2007inference,butucea2014semiparametric}.

\section{Estimation by maximizing the smoothed log-likelihood}\label{sec:computation}
\subsection{Smoothing operator and loss functions}
The estimation of the parameters in the mixture model defined by \eqref{eq:model}–\eqref{eq:model2} cannot be directly performed through log-likelihood maximization, as the model involves an infinite-dimensional parameter $\bpsi$. This difficulty can be circumvented by introducing a smoothing operator based on a kernel function. Let $\kernel$ denote a kernel density on the real line. We define the product kernel $\kernel(\bx) = \prod_{j=1}^J \kernel(x_j)$ and its rescaled version $\kernel_h(\bx) = h^{-J} \prod_{j=1}^J \kernel(x_j/h) = \prod_{j=1}^J \kernel_h(x_j)$ for a given bandwidth $h > 0$. Throughout, we use bold notation in the argument to indicate a rescaled multivariate kernel $\kernel_h(\bx)$, and regular font to denote a rescaled univariate kernel $\kernel_h(x)$. The kernel is assumed to satisfy standard regularity conditions.
\begin{assumptions}\label{ass:controlvariancekernel}
\begin{enumerate}
    \item The kernel function $\kernel$ is a symmetric,   square-integrable, continuous  density function of order $2$ that admits a derivative $\kernel'$ that has a finite $L_2$-norm. In other words, $\int K(u)\,du=1$, $\int u\kernel(u)\,du=0$, $\int u^{2}\kernel(u)\,du\ne 0$ and $\int (\kernel'(u))^2du<\infty$.
    \item There exists $b_1(h)$ and $b_2(h)$ two positive reals such that $b_1(h)\leq \kernel_h(u-v)\leq b_2(h)$\label{ass:boundkernel}
    \item There exists $L_h>0$ such that $|\kernel_h(x) - \kernel_h(y) |\leq L_h|x-y|$ for any $x,y$. \label{ass:lipschitzkernel}
    \item The kernel  a Gaussian or sub-Gaussian kernel with constant $\kappa$.
\end{enumerate}
\end{assumptions}

For any $J$-variate density function $\rho$, we consider the nonlinear smoothing operator $\smooth$ defined as
$$
\smooth \rho (\bx) =     \exp \int_{\mathcal{X}}   \kernel_h(\bx - \boldsymbol{y}) \ln \rho(\boldsymbol{y}) d\boldsymbol{y} ,
$$
where $h>0$ is a positive bandwidth. 
Note that $\smooth$ is a multiplicative operator in the following sense: for any function $\psi_{k}$ we have $$\smooth\psi_{k}(\bx)=\prod_{j=1}^{J}\smooth_j\psi_{k,j}(x_j),$$ with
$$\smooth_j\psi_{k,j}(x_j):= \exp[ (\kernel_h \star \ln \psi_{k,j})(x_j)],$$
where $\star$ denotes the convolution product such that $$ (\kernel_h \star \ln \psi_{k,j})(x_j)=\int_{\mathcal{X}_j}\kernel_{h}(x_{j}-u)\ln\psi_{k,j}(u)du.$$
Due to Jensen's inequality, although $\smooth_j\psi_{k,j}(x_j)$ is a positive function, its integral $\int \smooth_j\psi_{k,j}(x_j)\,dx_j\leq1$. Thus, the result of such a smoothing is a ``subdensity", not a true density.
 Starting from parameter $(\bpi,\bpsi)$ and applying the nonlinear smoothing operator $\smooth$ with bandwidth $h$ on each component of the mixture $g_{\bpi,\bpsi}$, provides the subdensity $f^{(h)}_{\bpi,\bpsi}$ defined as
$$
f^{(h)}_{\bpi,\bpsi}(\bx)=\sum_{k=1}^K \pi_k   \smooth \psi_{k}(\bx).
$$

From the smoothing operator defined with any bandwidth $h>0$, as suggested by \cite{levine2011maximum}, we consider the following loss function 
\begin{equation}\label{eq:loss}
\loss^{(h)}(\bpi,\bpsi) = \int_{\mathcal{X}} g^\star(\bx) \ln \frac{g^\star(\bx)}{f^{(h)}_{\bpi,\bpsi}(\bx)} d\bx.
\end{equation}
This loss function can be interpreted as a sum of the (generalized) Kullback-Leibler divergence and an additional term:
$$
\loss^{(h)}(\bpi,\bpsi) =\KL(g^\star,f^{(h)}_{\bpi,\bpsi})+\int f^{(h)}_{\bpi,\bpsi}(\bx)\,d\bx-1,
$$
where, for any two non-negative function $a(\bx)$ and $b(\bx)$, the (generalized) Kullback-Leibler divergence is defined as
$$
\KL(a,b)=\int_\mathcal{X}\left[ a(\bx) \ln \frac{a(\bx)}{b(\bx)} + b(\bx) - a(\bx)\right] d\bx.
$$
In addition, we extend the definition of the loss function at $h=0$ by
\begin{equation}\label{eq:loss0}
\loss^{(0)}(\bpi,\bpsi)=\int_{\mathcal{X}} g^\star(\bx) \ln \frac{g^\star(\bx)}{g_{\bpi,\bpsi}(\bx)} d\bx.
\end{equation}
The following lemma establishes the order of the biases caused by the smoothing of the target density $g_{\bpi,\bpsi}$ and the loss function.
\begin{lemma}\label{lem:taylor}
Under Assumptions~\ref{ass:controlvariancekernel}, the properties of $\Theta_K$ ensures that
    $$
\sup_{(\bpi,\bpsi)\in\Theta_K} \|g_{\bpi,\bpsi} - f^{(h)}_{\bpi,\bpsi}\|_\infty=O(h^2)
$$ 
and
 $$
\sup_{(\bpi,\bpsi)\in\Theta_K} |\lossh(\bpi,\bpsi) - \loss^{(0)}(\bpi,\bpsi)|=O(h^2).
$$ 
\end{lemma}
 
Note that Lemma~\ref{lem:taylor} implies that $\lim_{h\to 0^+} \loss^{(h)}(\bpi,\bpsi)=\loss^{(0)}(\bpi,\bpsi)$. 
We define $(\bpi^{(h)},\bpsi^{(h)})$ as a the minimizer of $\loss^{(h)}(\bpi,\bpsi) $ with respect to its parameters.  
Note that, due to the smoothing, $(\bpi^{(h)},\bpsi^{(h)})$ is not equal to $(\bpi^\star,\bpsi^\star)$ in general. However, under Assumption~\ref{ass:controlvariancetrue} that ensures the parameter identifiability, we have $\lim_{h \to 0}(\bpi^{(h)},\bpsi^{(h)})=(\bpi^\star,\bpsi^\star)$. We consider $\bX_1,\ldots,\bX_n$ an observed sample composed $n$ independent observations drawn from $g^\star$. To perform the estimation of the parameters, we consider an empirical version of the loss function defined for any $h>0$ by
$$
\loss^{(h,n)}(\bpi,\bpsi) = \frac{1}{n} \sum_{i=1}^n \ln \frac{g^\star(\bX_i)}{f^{(h)}_{\bpi,\bpsi}(\bX_i)}.
$$
The parameter estimation is performed by minimizing $\loss^{(h,n)}(\bpi,\bpsi)$ with respect to $(\bpi,\bpsi)$ which is equivalent to maximizing the  smoothed log-likelihood (\emph{i.e.,} the log-likelihood function computed with the subdensity function $f^{(h)}_{\bpi,\bpsi}$). Denoting by $(\hpi,\hpsi)$ estimator that minimizes the empirical version of the loss function, we have
$$
(\hpi,\hpsi) = \argmin_{(\bpi,\bpsi) \in \Theta_K} \lossnh(\bpi,\bpsi).
$$

\subsection{Consistency of the estimator maximizing the smoothed log-likelihood}

The following lemma permits to control uniformly in $\bpsi$ the error term between $\lossnh(\bpi,\bpsi)$ and $\lossh(\bpi,\bpsi)$. 
\begin{lemma}\label{lem:deviation}
 Under Assumption~\ref{ass:controlvariancekernel}, the properties of $\Theta_K$ ensures that
 $$
\sup_{(\bpi,\bpsi) \in \Theta_K} | \lossnh(\bpi,\bpsi) -  \lossh(\bpi,\bpsi)|= O_\mathbb{P}(n^{-1/2}h^{-1/2}).
 $$
\end{lemma}

From Lemmas~\ref{lem:taylor} and~\ref{lem:deviation}, sufficient conditions on the bandwidth can be derived to ensure that $\lossnh$ converges in probability to $\loss^{(0)}$ uniformly over $\Theta_K$, under Assumptions~\ref{ass:controlvariancekernel}. Combining this result with the parameter identifiability ensured by Assumption~\ref{ass:controlvariancetrue} allows us to establish the following theorem, which states the consistency of the estimators $(\hpi,\hpsi)$.
\begin{theorem}\label{thm:consistence}
    Under Assumptions~\ref{ass:controlvariancetrue} and \ref{ass:controlvariancekernel}, as $h$ tends to zero as $n$ tends to infinity and $nh$ tends to infinity, then $(\hpi,\hpsi)$ converges in probability to $(\bpi^\star,\bpsi^\star)$ leading that 
we have
$$
\|(\hpi,\hpsi) - (\bpi^\star,\bpsi^\star)\|_\infty=o_\mathbb{P}(1).
$$ 
\end{theorem}

 Next, we establish a convergence rate of the estimator  in three steps. First, profiling is introduced, as is standard for semi-parametric problems involving likelihood-based estimation. However, note that here the profiling is performed on the smoothed version of the loss functions. Second, Theorem~\ref{thm:rateL1} shows that the accuracy of the infinite-dimensional estimators depends on the bandwidth, the sample size, and the accuracy of the finite-dimensional estimators. Third, Theorem~\ref{thm:norma} demonstrates that efficient inference can be conducted for the finite-dimensional parameter.

\section{Mapping functions for the parameter estimation algorithm and profiling the loss functions}\label{sec:prop}
\subsection{Mapping functions for the parameter estimation algorithm}
 The minimizations of $\lossh$ and $\lossnh$ do not admit closed-form solutions. A standard approach is to use a Majorization-Minimization (MM) algorithm to minimize \eqref{eq:loss} with respect to the parameters $(\bpi,\bpsi)$ (see \cite{lange2016mm} for a general review of MM algorithms, and \cite{levine2011maximum} for their application to mixture models). Starting from an initial value of the parameters, the algorithm alternates between a Majorization step and a Minimization step.  Among the algorithm's properties, \cite{levine2011maximum} established its monotonicity, while \cite{zhu2016theoretical} proved the existence of solutions to the optimization problems associated with the minimization of both $\lossh$ and $\lossnh$. To compute $(\bpi^{(h)},\bpsi^{(h)})$, the minimizers of $\lossh$, the MM algorithm is initialized at some starting point $(\bpi^{[0]},\bpsi^{[0]})$ and iteratively updated until convergence. The two steps that compose each iteration of the algorithm can be combined into a single mapping from $\Theta_K$ to $\Theta_K$. Specifically, iteration $r$ of the algorithm produces an updated parameter $(\bpi^{[r]},\bpsi^{[r]})$ from the previous iterate $(\bpi^{[r-1]},\bpsi^{[r-1]})$ by 
$$
\pi^{[r]}_k= P^{(h)}_{k}[\bpi^{[r-1]},\bpsi^{[r-1]}]
$$
and
$$
\psi_{k,j}^{[r]} = M^{(h)}_{k,j}[\bpsi^{[r-1]};\bpi^{[r-1]},\bpi^{[r]}],
$$
with
$$
P^{(h)}_{k}[\bpi,\bpsi]=\int_{\mathcal{X}} g^\star(\bx) \omega_{\bpi,\bpsi,k}^{(h)}(\bx)d\bx,
$$
and for any $u\in\mathcal{X}_j$
$$
M^{(h)}_{k,j}[\bpsi;\bpi,\tilde\bpi](u) = \frac{1}{   \tilde\pi_k} \int_{\mathcal{X}} g^\star(\bx) \omega^{(h)}_{\bpi,\bpsi,k}(\bx)\frac{1}{h}\kernel\left(\frac{x_{j} - u}{h} \right) d\bx,
$$
where $\omega_{\bpi,\bpsi,k}^{(h)}(\bx)$ corresponds to a smoothed version of the posterior probabilities of classification that observation $\bx$ arise from component $k$ given the parameters $(\bpi,\bpsi)$ and the bandwidth $h$ that are defined by
\begin{equation}\label{eq:weights}
\omega_{\bpi,\bpsi,k}^{(h)}(\bx) = \frac{\pi_k  \smooth \psi_{k}(\bx)}{f^{(h)}_{\bpi,\bpsi}(\bx)}.
\end{equation}
In the following, we denote by $M^{(h)}[\bpsi;\bpi,\tilde\bpi]$ the collection of $M^{(h)}_{k,j}[\bpsi;\bpi,\tilde\bpi]$ for $k\in\{1,\ldots,K\}$ and $j\in\{1,\ldots,J\}$. Note that the algorithm only converges to local optima of the objective function. Hence, different starting points need to be considered.

 To compute $(\hpi,\hpsi)$, the minimizers of $\lossnh$, a similar MM algorithm to the one used for minimizing $\lossh$ is employed, where all quantities are replaced by their empirical counterparts. The algorithm starts from an initial value $(\bpi^{[0]},\bpsi^{[0]})$ and iterates until convergence. At each iteration $r$, the parameters $(\bpi^{[r]},\bpsi^{[r]})$ are updated from $(\bpi^{[r-1]},\bpsi^{[r-1]})$ in the same manner as in the optimization of $\lossh$, with the functions $P_{k}^{(h)}$ and $M^{(h)}_{k,j}$ replaced by their empirical counterparts 
$$
P_{k}^{(h,n)}[\bpi,\bpsi]=\frac{1}{n} \sum_{i=1}^n \omega_{\bpi,\bpsi,K}^{(h)}(\bX_i),
$$
and
$$
M_{k,j}^{(h,n)}[\bpsi;\bpi,\tilde\bpi](u) = \frac{1}{ n \tilde\pi_k} \sum_{i=1}^n\omega_{\bpi,\bpsi,K}^{(h)}(\bX_i)\frac{1}{h}\kernel\left(\frac{X_{i,j} - u}{h} \right).
$$
In the following, we denote by $M^{(h,n)}[\bpsi;\bpi,\tilde\bpi]$ the collection of $M^{(h,n)}_{k,j}[\bpsi;\bpi,\tilde\bpi]$ for $k\in\{1,\ldots,K\}$ and $j\in\{1,\ldots,J\}$. 

\subsection{Profiling the loss function}

Let $\profileh$ be the profiled version of $\lossh$ defined by
$$
\profileh(\bpi) = \lossh(\bpi,\tpsipi ),
$$
where $\tpsipi$ is the infinite-dimensional parameter that minimizes $\lossh$ with respect to $\bpsi$ for a fixed value of $\bpi$: 
\begin{equation} \label{eq:tpsipi}
\tpsipi = \argmin_{\bpsi\in\Psi_K(\mathcal{X})} \lossh(\bpi,\bpsi).
\end{equation}
Hence, by definition of $(\bpi^{(h)},\bpsi^{(h)})$, we have $\bpsi^{(h)} = \bpsi^{(h,\bpi^{(h)})}$ and $\bpi^{(h)}=\argmin_{\bpi\in\mathcal{S}_K} \profileh(\bpi)$.   Similarly, we defined $\profilenh$ as the profiled version of $\lossnh$ leading that
$$
\profilenh(\bpi) = \lossnh(\bpi,\hpsipi),
$$
where $\hpsipi$ is the infinite-dimensional parameter that minimizes $\lossnh$ with respect to $\bpsi$ for a fixed value of $\bpi$, leading that
\begin{equation}\label{eq:hpsipi}
\hpsipi = \argmin_{\bpsi\in\Psi_K(\mathcal{X})} \lossnh(\bpi,\bpsi).
\end{equation}
Hence, by definition of $(\hpi,\hpsi)$, we have $\hpsi = \widehat{\bpsi}^{(h,n,\hpi)}$ and $\hpi=\argmin_{\bpi\in\mathcal{S}_K}\profilenh(\bpi)$.  
Note that the computation of $\tpsipi$ and $\hpsipi$ can be computed via the MM algorithms described in the previous section, where the finite-dimensional parameters are not updated (\emph{i.e.,} $\pi_k^{[r]}=\pi_k$ for any iteration $r$). Hence, $\tpsipi$ and $\hpsipi$ are obtained by MM algorithms defined at iteration $r$ by
\begin{equation}\label{eq:Mh}
\bpsi^{[r]}=M^{(h)}_{\bpi}[\bpsi^{[r-1]}],
\end{equation}
and 
\begin{equation}\label{eq:Mhn}
\bpsi^{[r]}=M^{(h,n,\bpi)}[\bpsi^{[r-1]}],
\end{equation}
respectively, where $M^{(h)}_{\bpi}[\bpsi]:=M^{(h)}[\bpsi;\bpi,\bpi]$ and $M^{(h,n,\bpi)}[\bpsi]:=M^{(h,n)}[\bpsi;\bpi,\bpi]$.

The use of the operator $\argmin$, rather than $\inf$, in the definition of the profiling of $\lossh$ and $\lossnh$ is justified by the following lemma. Moreover, this lemma establishes that the infinite-dimensional parameters $\tpsipi$ and $\hpsipi$ are the unique fixed points of the MM algorithms defined respectively by \eqref{eq:Mh} and \eqref{eq:Mhn}, which optimize $\lossh$ and $\lossnh$ with the finite-dimensional parameters held fixed.
\begin{lemma}\label{lem:unicity}
Under Assumptions~\ref{ass:controlvariancetrue} and \ref{ass:controlvariancekernel}, for any $\bpi$ in the interior of $\mathcal{S}_K^r$
\begin{enumerate}
    \item the minimizer of $\lossnh(\bpi,\bpsi)$ with respect to $\bpsi\in\Psi_K(\mathcal{X})$ is unique and is the single fixed point of $M^{(h,n,\bpi)}[\bpsi]$ leading that
$M^{(h,n,\bpi)}[\bpsi]=\bpsi \Longleftrightarrow \bpsi=\hpsipi$.
\item the minimizer of $\lossh(\bpi,\bpsi)$ with respect to $\bpsi\in\Psi_K(\mathcal{X})$ is unique and is the single fixed point of $M^{(h,\bpi)}[\bpsi]$ leading that
$
M^{(h,\bpi)}[\bpsi]=\bpsi \Longleftrightarrow \bpsi=\tpsipi
$.
\end{enumerate}
\end{lemma}
The following remark highlights that, as a consequence of Lemma~\ref{lem:unicity}, the MM algorithm defined by \eqref{eq:Mhn}, which updates only the infinite-dimensional parameters while keeping the finite-dimensional parameters fixed at $\bpi$, converges to the minimizer $\hpsipi$ for any initial value of the infinite-dimensional parameters.\begin{remark}\label{rq:unicity}
As a consequence of Lemma~\ref{lem:unicity}, we have for any $\bpi$ in the interior of $\mathcal{S}_K^r$
$$
\forall \bpsi\in\Psi_K(\mathcal{X}),\, \lim_{p \to \infty} M^{(h,n,\bpi)\{p\}}[\bpsi]=\hpsipi,
$$
where $M^{(h,n,\bpi)\{p\}}[\bpsi]=M^{(h,n,\bpi)}[M^{(h,n,\bpi)\{p-1\}}[\bpsi]]$ denotes $p$ compositions of function $M^{(h,n,\bpi)}[\bpsi]$.
\end{remark}
Note that a similar result can be established for $\tpsipi$, but it will not be used in the proof that establishes the rate of convergence of the estimator.

\section{Controlling the convergence of the infinite-dimensional estimates}\label{sec:nprate}
The objective is to derive a convergence rate for the infinite-dimensional estimators that depends solely on the sample size, the bandwidth, and the convergence rate of the finite-dimensional estimators. To this end, we begin with Lemma~\ref{lem:prel}, which shows that, when the proportions are fixed, the norm of the difference between the infinite-dimensional parameters at two successive iterations of the MM algorithm can be upper bounded by the difference in the loss function $\lossnh$ evaluated at these points. As a consequence of Remark~\ref{rq:unicity}, the norm of the difference between the initial value of the infinite-dimensional parameter and its estimator minimizing $\lossnh$ with fixed proportions can be controlled by the corresponding difference in the loss function. Moreover, since Remark~\ref{rq:unicity} ensures that any element of $\Psi_K(\mathcal{X})$ can be used as an initial value $\bpsi^{[0]}$, taking the true infinite-dimensional parameter $\bpsi^\star$ as a starting point yields, via Lemma~\ref{lem:inequality}, a bound on the norm of the difference between $\bpsi^\star$ and $\hpsipi$ in terms of the loss function evaluated at these points with fixed proportions. Finally, combining the uniform control of the difference between the empirical and theoretical versions of the loss function provided by Lemma~\ref{lem:deviation}, with a bound on the difference between $\lossh(\bpi,\bpsi)$ and $\lossh(\bpi^\star,\bpsi)$ that depends on the norm of the difference between $\bpi$ and $\bpi^\star$, Theorem~\ref{thm:rateL1} establishes a bound on the difference between $\bpsi^\star$ and $\hpsipi$ as a function of the sample size, the bandwidth, and the norm of the difference between $\bpi$ and $\bpi^\star$.

Using the definition of the mapping functions that are implied by the algorithm, Lemma~\ref{lem:prel} shows that, when the proportions are fixed, the norm of the difference between $\bpsi$ and the infinite dimensional parameters $M^{(h,n,\bpi)}[\bpsi]$ defined by the mapping of the MM algorithm can be upper bounded by the difference in the loss function $\lossnh$ evaluated at these points. Note that this results can be seen as an extension of \cite[Corollary 3.1 and Corollary 3.3]{zhu2016theoretical}.
\begin{lemma}\label{lem:prel}
Under Assumptions~\ref{ass:controlvariancetrue} and \ref{ass:controlvariancekernel}, we have for any $\bpi\in\mathcal{S}_K$
$$  \lossnh(\bpi,\bpsi) - \lossnh(\bpi,M^{(h,n,\bpi)}[\bpsi])  \geq \frac{1}{4} \sum_{k=1}^K  \pi_k\sum_{j=1}^J \| \psi_{k,j} - M_{k,j}^{(h,n,\bpi)}[\bpsi]\|^2_1,$$
   where $ M_{k,j}^{(h,n,\bpi)}[\bpsi]$ is the element $(k,j)$ of $M^{(h,n,\bpi)}[\bpsi]$ that correspond to the update of $\psi_{k,j}$ provided by one iteration of the MM algorithm with fixed proportions leading that $ M_{k,j}^{(h,n,\bpi)}[\bpsi]=M_{k,j}^{(h,n)}[\bpsi;\bpi,\bpi]$.
\end{lemma}

Since the MM algorithm with fixed proportions converges to $\hpsipi$ from any starting value of the infinite-dimensional parameters (see Remark~\ref{rq:unicity}), this holds in particular when starting from $\bpsi^\star$. Exploiting this property together with Lemma~\ref{lem:prel}, the following lemma provides an upper bound on the sum of the squared $L_1$ norms of the differences between $\psi^\star_{k,j}$ and $\widehat{\psi}^{(h,n,\bpi)}_{k,j}$ in terms of the difference between the empirical loss function evaluated at $\bpsi^\star$ and at $\hpsipi$, with $\bpi$ fixed.

\begin{lemma}\label{lem:inequality}
Under Assumptions~\ref{ass:controlvariancetrue} and \ref{ass:controlvariancekernel}, we have for any $\bpi\in\mathcal{S}_K$
   $$ \lossnh(\bpi,\bpsi^\star) - \lossnh(\bpi,\hpsipi) \geq \frac{1}{4} \sum_{k=1}^K \sum_{j=1}^J \|\psi_{k,j}^\star - \widehat{\psi}^{(h,n,\bpi)}_{k,j}\|^2_1.$$
\end{lemma}

We are now in a position to derive the convergence rate of any estimator of the univariate component densities in the mixture model, under the assumption that the finite-dimensional parameter is fixed to some value $\bpi$ (not necessarily equal to the true value $\bpi^\star$).
\begin{theorem}\label{thm:rateL1}
Let $\mathcal{B}(\bpi^\star)$ be the ball centered in $\bpi^\star$ with radius equal to $\min \pi^\star_k/2$.
 Under Assumptions~\ref{ass:controlvariancetrue}, \ref{ass:controlvariancekernel}, we have   
$$
\forall \bpi\in\mathcal{B}(\bpi^\star),\; \sum_{k=1}^K \sum_{j=1}^J \| \psi^\star_{k,j} - \widehat{\psi}^{(h,n,\bpi)}_{k,j}\|^2_1 = O_\mathbb{P}( n^{-1/2}h^{-1/2}+h^2+\|\bpi - \bpi^\star\|_1).
$$
\end{theorem}

\section{Controlling the convergence of the finite-dimensional estimates}\label{sec:rate}
\subsection{Three score functions with smoothing}
To study the asymptotic behavior of $\hpi$, we need to introduced three score functions obtained after smoothing: the naive score function with smoothing, the nuisance  score function with smoothing and the efficient  score function with smoothing (see \cite{kosorok2008introduction} for a general introduction of these three score functions). Noting that $\bpi\in\mathcal{S}_K^r$, there is a linear constraints between the elements of the vector, therefore all the partial derivative as considered only with respect to $\pi_k$ with $k=1,\ldots,K-1$ and where $\pi_K=1 - \sum_{k=1}^{K-1}\pi_k$
Let $\bscore_{\bpi,\bpsi}^{(h)}(\bx)=(\score_{\bpi,\bpsi,1}^{(h)}(\bx),\ldots,\score_{\bpi,\bpsi,K-1}^{(h)}(\bx))^\top\in\mathbb{R}^{K-1}$ be the \emph{naive score function with smoothing} associated to the smoothed log-likelihood such that  $\score_{\bpi,\bpsi,k}^{(h)}(\bx)$ is defined as the partial derivative of $\ln f^{(h)}_{\bpi,\bpsi}(\bx)$ with respect to $\pi_k$ leading
$$
\score_{\bpi,\bpsi,k}^{(h)} = \frac{\partial}{\partial \pi_k} \ln f^{(h)}_{\bpi,\bpsi}.
$$
Hence, the naive score function  with smoothing is the $(K-1)$-dimensional vector where the element $k$ is defined by
\begin{equation}\label{eq:naive}
\score_{\bpi,\bpsi,k}^{(h)}(\bx)  = \frac{\smooth \psi_k(\bx) - \smooth \psi_K(\bx)}{f^{(h)}_{\bpi,\bpsi}(\bx)}.    
\end{equation}
The naive score function reflects the direction in which the smoothed log-likelihood increases the most when only the finite-dimensional parameter is perturbed, without accounting for the variability introduced by the infinite dimensional parameter that is unknown and thus needs to be estimated. As a result, it fails to capture the full uncertainty of the estimation problem and is not sufficient for stating the rate of convergence of the estimators of the proportions. Therefore, we need to introduce the nuisance score function with smoothing that captures the sensitivity of the smoothed log-likelihood with respect to infinitesimal perturbations of the infinite-dimensional parameter, while keeping the finite-dimensional parameter fixed. It quantifies how the smoothed log-likelihood reacts to small variations in the infinite-dimensional parameters. In a sense, it characterizes the influence of $\bpsi$ on the estimation procedure. The \emph{nuisance score function with smoothing} at $(\bpi,\bpsi)$ in direction $\bar\bpsi - \bpsi$, denoted by $\noisy_{\bpi,\bpsi}^{(h)}[\bar\bpsi-\bpsi]$, is defined as the Gateaux derivative of the smoothed log-likelihood  at $(\pi,\bpsi)$ in direction $\bar\bpsi-\bpsi$ leading that
$$
\noisy_{\bpi,\bpsi}^{(h)}[\bar\bpsi-\bpsi] =\frac{\partial}{\partial t} \ln f^{(h)}_{\bpi,\bpsi + t(\bar\bpsi - \bpsi)} \Big|_{t=0}.
$$
The Gateaux derivative of $\smooth_j \psi_{k,j}$ in direction $\bar\psi_{k,j} - \psi_{k,j}$, denoted by $\partial \smooth_j\psi_{k,j}[\bar\psi_{k,j} - \psi_{k,j}]$ is defined as
$$\partial \smooth_j\psi_{k,j}[\bar\psi_{k,j} - \psi_{k,j}] = \frac{\partial}{\partial t}\ \smooth_j[\psi_{k,j} + t(\bar\psi_{k,j}-\psi_{k,j})](x_j)  \Big|_{t=0}.$$ Therefore, we have
\begin{equation*}\label{eq:Gat}
\partial \smooth_j\psi_{k,j}[\bar\psi_{k,j} - \psi_{k,j}](x_j) = \left( \left[\kernel_h \star \frac{\bar\psi_{k,j}-\psi_{k,j}}{\psi_{k,j}}\right](x_j) \right) \smooth_j\psi_{k,j}(x_j).
\end{equation*}
Hence, using the chain rule and the product rule, the nuisance score function with smoothing is defined by
$$
\noisy_{\bpi,\bpsi}^{(h)}[\bar\bpsi-\bpsi](\bx) =\sum_{k=1}^K  \omega_{\bpi,\bpsi,k}^{(h)}(\bx)  \zeta_{\bpsi,\bar\bpsi,k}^{(h)}(\bx),
$$
 where 
\[
\omega_{\bpi,\bpsi,k}^{(h)}(\bx)=\frac{\pi_{k}\smooth\psi_{k}(\bx)}{\sum_{k=1}^{K}\pi_{k}\smooth\psi_{k}(\bx)}
\] 
is a smoothed version of the posterior probability of  observation $\bx$ arising from the component $k$ given the parameters $(\bpi,\bpsi)$ and the bandwidth $h$ that are defined by \eqref{eq:weights}. At the same time, 
\begin{equation}\label{eq:zeta}
\zeta_{\bpsi,\bar\bpsi,k}^{(h)}(\bx)=\sum_{j=1}^J \left[\kernel_h \star \frac{\bar\psi_{k,j}-\psi_{k,j}}{\psi_{k,j}}\right](x_j) . 
\end{equation}
We can now define the \emph{tangent cone with smoothing} that characterizes the possible directions in which the infinite-dimensional parameter can vary infinitesimally, under the model constraints. Hence, it is defined as
$$
\mathcal{T}^{(h)}=\left\{ \bx \mapsto \noisy_{\bpi^\star,\bpsi^\star}^{(h)}[\bpsi - \bpsi^\star] ; \bpsi\in\Psi_K(\mathcal{X})\right\}.
$$
The \emph{efficient score function with smoothing} $\beffscore_{\bpi^\star,\bpsi^\star}^{(h)}=(\effscore_{\bpi^\star,\bpsi^\star,1}^{(h)},\ldots,\effscore_{\bpi^\star,\bpsi^\star,K}^{(h)})^\top\in\mathbb{R}^{K-1}$  corresponds to the component of the naive score function with smoothing evaluated at the true parameters that is orthogonal to all variations of the infinite-dimensional parameter, as characterized by the tangent cone $\mathcal{T}^{(h)}$. Hence, it represents the part of the score that carries pure information about $\bpi$, uncontaminated by the influence of $\bpsi$.  It is defined as the projection of each coordinate of the naive score, in the sense $L_2(g^\star)$, on the tangent cone $\mathcal{T}^{(h)}$.  Hence, we have for any $k=1,\ldots,K-1$
\begin{equation}\label{eq:effscore}
\effscore_{\bpi^\star,\bpsi^\star,k}^{(h)} =\score_{\bpi^\star,\bpsi^\star,k}^{(h)}  - \noisy_{\bpi^\star,\bpsi^\star}^{(h)}[\check\bpsi^{(h)} - \bpsi^\star],
\end{equation}
where $\check\bpsi^{(h)}\in\Psi_K(\mathcal{X})$  satisfies for any $\bpsi\in\Psi_K(\mathcal{X})$ and any $k=1,\ldots,K$
\begin{equation}\label{eq:condproj}
 \Eg\left[\left(\score_{\bpi^\star,\bpsi^\star,k}^{(h)}(\bX_1) - \noisy_{\bpi^\star,\bpsi^\star}^{(h)}[\check\bpsi^{(h)}  - \bpsi^\star](\bX_1)\right) \noisy_{\bpi^\star,\bpsi^\star}^{(h)}[\bpsi - \bpsi^\star](\bX_1) \right]=0.
\end{equation}
 
In particular,  the \emph{asymptotic efficient score function} (\emph{i.e.,} efficient score function with smoothing when the smoothing vanished)  is defined by
$$
\beffscore_{\bpi^\star,\bpsi^\star} = \lim_{h\to 0}\beffscore_{\bpi^\star,\bpsi^\star}^{(h)}.
$$
Note that, by definition, we have
$$
\Eg [\beffscore_{\bpi^\star,\bpsi^\star}(\bX)] = \boldsymbol{0}_K.
$$
Similarly,   the \emph{asymptotic efficient Fisher information matrix} (\emph{i.e.,} when the smoothing vanished) and 
\begin{equation}\label{eq:fisher}
\befffishers= \Eg[\beffscore_{\bpi^\star,\bpsi^\star}(\bX)\beffscore_{\bpi^\star,\bpsi^\star}^\top(\bX)].
\end{equation}
To establish the rate of convergence for our estimators, the following result is important.
\begin{lemma} \label{lem:invert}
Under Assumptions~\ref{ass:controlvariancetrue}-\ref{ass:controlvariancekernel}, the asymptotic efficient Fisher information matrix $\befffishers$ is invertible
\end{lemma}

\subsection{Rate of convergence of the finite-dimensional estimates}
We can now establish that the estimator of the proportions, $\hpi$, converges in probability to $\bpi^\star$ at a rate $n^{-r}$, where $r > 1/4$ depends on the bandwidth, as specified in Assumption~\ref{ass:band1}.
Indeed, Assumption~\ref{ass:band1} ensures that $$n^{-1/2}h^{-1/2}+h^2=o(n^{-r}).$$
\begin{assumptions} \label{ass:band1}
The bandwidth $h$ satisfies that $h n^{r/2} \to 0$ and $h n^{1 - 2r} \to \infty$ as $n \to \infty$ for some $r$ such that $r>1/4$.
\end{assumptions}

\begin{remark}
Since $\hat\bpi^{h,n}$ is a consistent estimator of $\bpi^\star$, it belongs to $\mathcal{B}(\bpi^\star)$ with high probability.
As a direct consequence of Theorem~\ref{thm:rateL1}, and under Assumptions~\ref{ass:controlvariancetrue}, \ref{ass:controlvariancekernel}, and \ref{ass:band1}, we obtain
$$
 \sum_{k=1}^K \sum_{j=1}^J \| \psi^\star_{k,j} - \widehat{\psi}_{k,j}^{(h,n)}\|^2_1 = O_\mathbb{P}( \|\hpi - \bpi^\star\|_1)+o_\mathbb{P}( n^{-r}).
$$
\end{remark}

To establish the asymptotic distribution of the maximum likelihood estimator, the standard proof relies on the quadratic expansion of the likelihood. However, here we have to work with the smoothed profile log-likelihood. \cite[Theorem 1]{murphy2000profile} gives sufficient conditions to state that semi-parametric profile likelihoods, where the nuisance parameter has been profiled out, behave like ordinary likelihoods in that they have a quadratic expansion. This result cannot be used directly in our context since we consider smoothed version of the likelihoods. Therefore, we start by giving a proposition that extends the results of \cite[Theorem 1]{murphy2000profile} to smoothed likelihoods. In addition, in our situation the "no-bias" condition introduced by  \cite{murphy2000profile} is no longer satisfied because the rate of convergence of the infinite-dimensional estimator established by Theorem~\ref{thm:rateL1} is too slow. Hence, we need to adapt  \cite[Theorem 1]{murphy2000profile} to the situation where the "no-bias" condition is not satisfied but that a "small-bias" condition is satisfied, with careful attention paid to the effects introduced by the smoothing. Note that the "small-bias" condition would no lead to the efficiency and implies that the estimation of the infinite-dimensional parameter slows down the rate of convergence of the finite-dimensional parameters.

 \begin{proposition}\label{thm:convparam}
 Let $\bt$ having the same dimension that $\bpi$. For each parameter $(\bpi,\bpsi)$, there exists a map, which we denote by $\bt\mapsto \bpsi_{\bt}(\bpi,\bpsi)$, from a fixed neighborhood of $\bpi$ into the parameter set for $\bpsi$ such that the map $\bt \mapsto \bmap^{(h)}(\bt,\bpi,\bpsi)(\bx)$ is defined by 
\begin{equation}\label{eq:defmap}
\bmap^{(h)}(\bt,\bpi,\bpsi) = \ln f^{(h)}_{\bt,\bpsi_{\bt}(\bpi,\bpsi)}.
\end{equation}
 Hence,  $\bmap^{(h)}(\bt,\bpi,\bpsi)(\bx) $ corresponds to the smoothed version of the log-likelihood of the mixture model with parameters $(\bt,\bpsi_{\bt}(\bpi,\bpsi))$ evaluated at $\bx$. Suppose that the following conditions are satisfied for some real $r$ with $1/4<r\leq 1/2$ and for  a neighborhood $V$ of $(\bpi^\star,\bpi^\star,\bpsi)$ 
\begin{enumerate}[label=C-\arabic*]
\item \label{C0}
Suppose that $\bt \mapsto \bpsi_{\bt}(\bpi,\bpsi)$ where $\bpsi_{\bt}(\bpi,\bpsi)$ is a matrix of functions with $K$ rows and $J$ column such that its element of row $\ell$ and column $j$ is the real function defined on $\mathcal{X}_j$ and denoted by $\psi_{\bt,\ell,j}(\bpi,\bpsi)$. Suppose that for any $(\ell,j)$,  all its first and second order partial derivatives of $\psi_{\bt,\ell,j}(\bpi,\bpsi)$ are continuous functions in the neighborhood of $V$, and that there exist square integrable functions of $\bx$ that upper-bound $\sup_{(\bt,\bpi,\bpsi)\in V} |\bpsi_{\bt}(\bpi,\bpsi)|$ and $\sup_{(\bt,\bpi,\bpsi)\in V} \left| \frac{\frac{\partial}{\partial t_k}   \psi_{\bt,\ell,j}(\bpi,\bpsi)}{ \psi_{\bt,\ell,j}(\bpi,\bpsi)}\right|$  and an integrable function of $\bx$ that upper-bounds $\sup_{(\bt,\bpi,\bpsi)\in V} \left| \frac{\frac{\partial^2}{\partial t_{k'} \partial t_k}   \psi_{\bt,\ell,j}(\bpi,\bpsi)}{ \psi^2_{\bt,\ell,j}(\bpi,\bpsi)}\right|$. 
\item \label{C1} The map $\bt \mapsto \bmap^{(h)} (\bt,\bpi,\bpsi)(\bx)$ is twice continuously differentiable with respect to $\bt$ for all $\bx$ and all $h$ and its first two derivatives are denoted by $\dot{\bmap}^{(h)}(\bt,\bpi,\bpsi)$ and $\ddot{\bmap}^{(h)}(\bt,\bpi,\bpsi)$. Furthermore, 
\begin{enumerate}
    \item the class of functions $\mathcal{D}_{n,r}=\{n^{r-1/2}\dot\bmap^{(h)}(\bt,\bpi,\bpsi): (\bt,\bpi,\bpsi)\in V\}$ 
    is $g^\star$-Donsker with square-integrable envelope function, meaning that for any $k$, $\Gn n^{r-1/2} \dot\bmap^{(h)}_k(\bt,\bpi,\bpsi)$
    converges in distribution to a centered Gaussian process, where we have  $
    \Gn s=\frac{1}{\sqrt{n}} \sum_{i=1}^n  \left( s(\bX_i) - \Eg[s(\bX_i)]\right) $ and that there exists $\dot\nu_k\in L_2(g^\star)$ such that for any $(\bt,\bpi,\bpsi)\in V$, we have 
    $     |n^{r-1/2}\dot\bmap_k^{(h)}(\bt,\bpi,\bpsi)|\leq \dot\nu_k$.
    \item the class of functions  $\{\ddot\bmap(\bt,\bpi,\bpsi): (t,\bpi,\bpsi)\in V\}$ is $g^\star$-Glivenko-Cantelli and is bounded in $L_1(g^\star)$ meaning that 
       $$ \sup_{(\bt,\bpi,\bpsi)\in V} \left|\Pn \ddot\bmap^{(h)}_{\ell,k}(\bt,\bpi,\bpsi) - \Eg[\ddot\bmap_{k,\ell}^{(h)}(\bt,\bpi,\bpsi)(\bX_1)]\right|=o_\mathbb{P}(1), $$
       where $\Pn s= n^{-1} \sum_{i=1}^n s(\bX_i)$,        and there exists $\ddot\nu_{k,\ell}\in L_1(g^\star)$ such that for any $(\bt,\bpi,\bpsi)\in V$, we have   $|\dot\bmap_{k,\ell}^{(h)}(\bt,\bpi,\bpsi)|\leq \ddot\nu_{k,\ell}$.
\end{enumerate} 
     \item \label{C2} The submodel with parameters $(\bt,\bpsi_t(\bpi,\bpsi))$ should pass through $(\bpi,\bpsi)$ at $\bt=\bpi$:
    \begin{equation*}  
        \bpsi_{\bpi}(\bpi,\bpsi)=\bpsi,\, \forall (\bpi,\bpsi).
    \end{equation*}
    \item \label{C3} The score function with smoothing for the parameter $\bt$  of the model with likelihood $\bmap(\bt,\bpi^\star,\bpsi^\star)$ evaluated at $\bt=\bpi^\star$ tends to the efficient score function for $\bpi$ as $h$ tends to zero leading that
    \begin{equation*}  
       \lim_{h\to 0} \dot\bmap^{(h)}(\bpi^\star,\bpi^\star,\bpsi^\star) = \beffscores,
    \end{equation*}
    \item \label{C4} For any random sequences $\tilde\bpi^{(n)}$ that converges in probability to $\bpi^\star$, we have
    \begin{equation*}
        \widehat{\bpsi}^{(h,n,\tilde\bpi^{(n)})} \xrightarrow{p} \bpsi^\star,
    \end{equation*}
    for some metric and an extension of "small-bias condition" is satisfied  meaning that for any $k$
    \begin{equation*}
        \Eg   [     \dot\bmap^{(h)}_k(\bpi^\star,\tilde\bpi^{(n)}, \widehat{\bpsi}^{(h,n,\tilde\bpi^{(n)})})(\bX_1)] = o_\mathbb{P}(\|\tilde\bpi^{(n)} - \bpi^\star\| +   n^{-r}  ).
    \end{equation*}
\end{enumerate}
Then, for any random sequence $\tilde\bpi^{(n)}$ that converges in probability to $ \bpi^\star$, 
\begin{multline*}
\profilenh(\bpi^\star)=\profilenh(\tilde\bpi^{(n)}) + (\tilde\bpi^{(n)} -  \bpi^\star)^\top \Pn \beffscores  - \frac{1}{2}(\tilde\bpi^{(n)} -  \bpi^\star)^\top \befffishers(\tilde\bpi^{(n)} -  \bpi^\star) \\+ o_\mathbb{P}([ \|\tilde\bpi^{(n)} -  \bpi^\star\| +  n^{-r}  ]^2).
\end{multline*}
 \end{proposition}

We are now able to state the stochastic order of the estimator the finite-dimensional parameters. 
\begin{theorem}\label{thm:norma}
    Under Assumptions~\ref{ass:controlvariancetrue}, \ref{ass:controlvariancekernel} and \ref{ass:band1}, the estimator of the proportions $\hpi$ converges at the rate $n^{-r}$ such that
    $$
    \|\hpi - \bpi^\star\|_1=O_\mathbb{P}(n^{-r}).
    $$
\end{theorem}
To prove this theorem, we begin by verifying that the assumptions of Proposition~\ref{thm:convparam} are satisfied, allowing us to derive a quadratic expansion of the smoothed profile log-likelihood. To this end, we rely on the regularity of the parameter spaces, an appropriate choice of the bandwidth, the consistency of the parameter estimators as established in Theorem~\ref{thm:convparam}, and the control over the accuracy of the infinite-dimensional estimators, which depends on the bandwidth, the sample size, and the accuracy of the finite-dimensional estimators, as stated in Theorem~\ref{thm:rateL1}. 
As a direct consequence of Theorems~\ref{thm:norma} and \ref{thm:norma}, under Assumptions~\ref{ass:controlvariancetrue}, \ref{ass:controlvariancekernel} and \ref{ass:band1}, if $h=Cn^{-1/5}$, for some constant $C$, we have for any $\varepsilon>0$
$$
 \|\hpi - \bpi^\star\|_1=O_\mathbb{P}(n^{-2/5 - \varepsilon}).
$$
and
$$
\sum_{k=1}^K \sum_{j=1}^J \| \psi^\star_{k,j} - \widehat{\psi}_{k,j}^{(h,n)}\|^2_1=O_\mathbb{P}(n^{-2/5 - \varepsilon}).
$$

\subsection{Extension to the variables defined on the real line}
$$
\widetilde \Psi(\mathbb{R}) = \{ \psi_{k,j}\in L_2(\mathbb{R}),\, 0<\psi\leq C_1,   \|\ln \psi\|_{L_2(g^\star)}\leq C_2, \|(\ln \psi)^{''}\|_{L^{\infty}}\leq C_3\}.
$$ 
We denote by $\widetilde \Psi(\mathbb{R}^J)$ the space obtain as a product of $J$ spaces $\widetilde \Psi(\mathbb{R})$. We consider the set of parameters
$$
\widetilde\Theta_K=\mathcal{S}_K^r \times \widetilde \Psi(\mathbb{R}).
$$
To ensure that the asymptotic Fisher information matrix is still invertible in the case of densities define on real line, some additional assumptions needs to be done. These assumptions are stated by Assumptions~\ref{ass:poly}. For example, this assumption is satisfied for marginal densities with tails decaying at the same polynomial order in the same dimension. This result cannot be extended, however, to many other marginal densities.   
\begin{assumptions}\label{ass:poly}
    For any $(k,k')$ and and $j$, $\psi_{k,j}^\star/\psi_{k',j}$ is bounded away from zero and infinity.
\end{assumptions}
\begin{theorem}\label{thm:normabis}
If $\mathcal{X}_j=\mathbb{R}$ and considering the parameter space $\widetilde{\Theta}_K$, under Assumptions~\ref{ass:controlvariancetrue}, \ref{ass:controlvariancekernel}, \ref{ass:band1} and \ref{ass:poly}, we have
    $$
    \|\hpi - \bpi^\star\|_1=O_\mathbb{P}(n^{-r})
    $$
    and
    $$
    \sum_{k=1}^K \sum_{j=1}^J \| \psi^\star_{k,j} - \widehat{\psi}^{(h,n,\bpi)}_{k,j}\|^2_1 = O_\mathbb{P}( n^{-r}).
    $$
\end{theorem}

\section{Simulation} \label{sec:simulation}

In this section, we illustrate the finite-sample performance of the proposed smoothed likelihood estimator on a simple benchmark mixture model.  
Our main objective is to assess the empirical behavior of both the estimated mixing proportions and the component densities, and to verify whether the convergence rates suggested by the theory are observed in practice across different underlying distributions.
 
We consider a two-component mixture model in dimension $d = 3$,  
$g^\star(\bx) = \tfrac{1}{3} f_1(\bx) + \tfrac{2}{3} f_2(\bx)$,
where the components $f_1$ and $f_2$ are \emph{product densities} with identical marginals up to a location shift of order $1/\sqrt{d}$.  
Specifically, for each component $u \in \{1,2\}$ and each coordinate $j$,
\[
X_{ij}^{(u)} \sim F_0(\,\cdot\, + (-1)^u / \sqrt{d}),
\]
where $F_0$ is either the standard Gaussian, Student-$t_3$, or Laplace distribution.  
This setting ensures partial overlap between the components, thus providing a realistic and moderately challenging mixture identification problem. For each choice of the baseline law $F_0$ and each sample size $n \in \{ 200,400,800,1600,3200\}$,  
we generated $1000$ independent samples.  
The smoothed likelihood estimator was computed using the \texttt{npEM} algorithm from the \texttt{mixtools} package in R \citep{mixtools}, with a bandwidth set to
$
h = \mathrm{sd}(X)\, n^{-1/5},
$
in accordance with the theoretical prescription of the model.  
 
Two aspects were evaluated:   we recorded the absolute deviation of the proportions of the first component $|\hat{\pi}_1 - 1/3|$ and  for each component and marginal, we computed the $L^1$ distance between the estimated univariate density and the true one. The results are summarized in terms of scaled errors, \emph{i.e.} $n^{2/5 - \varepsilon}\, \mathbb{E}[|\hat{\pi}_1 - 1/3|]$ and $n^{2/5 - 0.001}\, \mathbb{E}\!\left[ \|\hat{f}_{u,h} - f_u\|_1^2 \right]$, with $\varepsilon=0.001$. Table~\ref{tab:prop_error} reports the scaled errors on the estimated mixing proportions, while Table~\ref{tab:density_error} reports the corresponding scaled $L^1$ errors for the component densities.

\begin{table}[ht]
\centering
\begin{tabular}{rrrrrrrr}
  \hline
& 200 & 400 & 800 & 1600 & 3200 \\ 
  \hline
Gaussian & 0.62 & 0.60 & 0.55 & 0.51 & 0.49 \\ 
  Student & 1.03 & 1.00 & 0.92 & 0.74 & 0.73 \\ 
  Laplace & 0.40 & 0.35 & 0.34 & 0.34 & 0.32 \\ 
   \hline
\end{tabular}
\caption{Scaled errors on estimated mixing proportions: $n^{2/5 - \varepsilon}\, |\hat{\pi}_1 - 1/3|$.}
\label{tab:prop_error}
\end{table}

\begin{table}[ht]
\centering
\begin{tabular}{rrrrrrrr}
  \hline
 &  200 & 400 & 800 & 1600 & 3200 \\ 
  \hline
Gaussian  & 0.54 & 0.28 & 0.17 & 0.13 & 0.10 \\ 
  Student & 1.46 & 1.03 & 0.64 & 0.39 & 0.36 \\ 
  Laplace & 0.47 & 0.39 & 0.34 & 0.29 & 0.24 \\ 
   \hline
\end{tabular}
\caption{Scaled $L^1$ errors for component densities: $n^{2/5 - \varepsilon}\, \mathbb{E}[\|\hat{f}_{u,h} - f_u\|_1^2]$.}
\label{tab:density_error}
\end{table}

For all distributions, the scaled errors decrease as $n$ increases, showing that both the mixing proportion and density estimates improve with larger sample sizes.  
Gaussian and Laplace mixtures reach very small errors for the largest $n$, illustrating stable estimation.  
Student-$t_3$ mixtures converge more slowly due to heavy tails, which increase variability in the kernel density estimates.  Overall, these results validate the theoretical findings derived in Sections 4--5:  
the smoothed likelihood estimator achieves the expected rate of convergence for both the finite-dimensional parameters and the nonparametric component densities.  
They also illustrate the practical influence of the underlying distribution, with heavy-tailed components requiring larger sample sizes for stable estimation.

\section{Conclusion}\label{sec:conclusion}

In this paper, we studied the problem of parameter estimation in semi-parametric finite mixture models where each component density is represented as a product of univariate densities.
Unlike existing approaches based on data discretization or tensor decompositions, our analysis focused on the estimator obtained by maximizing a smoothed version of the log-likelihood function, in which each component density is replaced by the exponential of the convolution between a kernel and its logarithm.

We established the consistency of both the finite- and infinite-dimensional estimators under standard identifiability and regularity assumptions, as the sample size increases and the bandwidth decreases at an appropriate rate.
Furthermore, by exploiting the convexity properties of the smoothed likelihood and a key inequality linking successive iterations of the MM algorithm (see Lemma~\ref{lem:prel}), we derived convergence rates that explicitly characterize the impact of the smoothing parameter on the estimation accuracy.
The subsequent analysis of the profile smoothed likelihood provided additional insight into how the presence of nuisance infinite-dimensional parameters modifies the asymptotic behavior of the estimators for the mixing proportions, and in particular how smoothing affects their convergence rate.

The rates obtained are not claimed to be optimal. Improving them while preserving the spirit of the approach would likely require sharper lower bounds on the Kullback–Leibler divergence than those provided by Pinsker’s inequality.
Such refinements would probably come at the cost of stronger regularity or separation assumptions on the component densities, ensuring better local identifiability of the mixture structure.

Overall, our theoretical results provide the first formal guarantees for the smoothed likelihood approach introduced by \cite{levine2011maximum}, thereby offering a principled justification for its practical use in semi-parametric mixture models.
Beyond their methodological implications, these results open the way to several extensions.
Future research directions include establishing the asymptotic normality of the finite-dimensional estimators, developing data-driven bandwidth selection rules, and extending the analysis to models incorporating covariates or dependence structures within components.
Another promising avenue is the study of the algorithmic convergence properties of the MM procedure and its possible acceleration through stochastic or proximal variants.

\section*{Acknowledgements}
Michael Levine's research has been partially funded by the NSF-DMS grant \# 2311103. 
\bibliographystyle{abbrvnat}
\bibliography{biblioConsistency}

@article{gassiat2018efficient,
  title={Efficient semiparametric estimation and model selection for multidimensional mixtures},
  author={Gassiat, Elisabeth and Rousseau, Judith and Vernet, Elodie},
journal={Electronic Journal of Statistics},
volume={12},
pages={703--740},
  year={2018}
}

@article{anandkumar2014tensor,
  title={Tensor decompositions for learning latent variable models.},
  author={Anandkumar, Animashree and Ge, Rong and Hsu, Daniel J and Kakade, Sham M and Telgarsky, Matus and others},
  journal={J. Mach. Learn. Res.},
  volume={15},
  number={1},
  pages={2773--2832},
  year={2014}
}

@Book{McL00,
  author =	 {McLachlan, G. and Peel, D.},
  title = 	 {Finite mixutre models},
  publisher = 	 {Wiley Series in Probability and Statistics: Applied Probability and Statistics},
  year = {2000},
  address =  {Wiley-Interscience, New York},
}

@article{hettmansperger2000almost,
  title={Almost nonparametric inference for repeated measures in mixture models},
  author={Hettmansperger, TP and Thomas, Hoben},
  journal={Journal of the Royal Statistical Society: Series B (Statistical Methodology)},
  volume={62},
  number={4},
  pages={811--825},
  year={2000},
  publisher={Wiley Online Library}
}

@article{cruz2004semiparametric,
  title={Semiparametric mixture models and repeated measures: the multinomial cut point model},
  author={Cruz-Medina, IR and Hettmansperger, TP and Thomas, H},
  journal={Journal of the Royal Statistical Society Series C: Applied Statistics},
  volume={53},
  number={3},
  pages={463--474},
  year={2004},
  publisher={Oxford University Press}
}

@article{bonhomme2016non,
  title={Non-parametric estimation of finite mixtures from repeated measurements},
  author={Bonhomme, St{\'e}phane and Jochmans, Koen and Robin, Jean-Marc},
  journal={Journal of the Royal Statistical Society Series B: Statistical Methodology},
  volume={78},
  number={1},
  pages={211--229},
  year={2016},
  publisher={Oxford University Press}
}

@article{du2024full,
  title={Full-model estimation for non-parametric multivariate finite mixture models},
  author={Du Roy de Chaumaray, Marie and Marbac, Matthieu},
  journal={Journal of the Royal Statistical Society Series B: Statistical Methodology},
  volume={86},
  number={4},
  pages={896--921},
  year={2024},
  publisher={Oxford University Press US}
}

@article{kwon2021estimation,
  title={Estimation of the number of components of nonparametric multivariate finite mixture models},
  author={Kwon, Caleb and Mbakop, Eric},
  journal={The Annals of Statistics},
  volume={49},
  number={4},
  pages={2178--2205},
  year={2021},
  publisher={Institute of Mathematical Statistics}
}

@article{bonhomme2016estimating,
  title={Estimating multivariate latent-structure models},
  author={Bonhomme, St{\'e}phane and Jochmans, Koen and Robin, Jean-Marc},
  year={2016},
  journal={The Annals of Statistics},
  volume={44},
  number={2},
  pages={540--563},
  publisher={Institute of Mathematical Statistics}
}

@article{kasahara2014non,
  title={Non-parametric identification and estimation of the number of components in multivariate mixtures},
  author={Kasahara, H. and Shimotsu, K.},
  journal={Journal of the Royal Statistical Society: Series B},
  volume={76},
  number={1},
  pages={97--111},
  year={2014},
  publisher={Wiley Online Library}
}

@article{hunter2007inference,
  title={Inference for mixtures of symmetric distributions},
  author={Hunter, David R and Wang, Shaoli and Hettmansperger, Thomas P},
  journal={The Annals of Statistics},
  volume={35},
  number={1},
  pages={224--251},
  year={2007},
  publisher={Institute of Mathematical Statistics}
}

@incollection{CloggBook1995,
  title={Latent class models},
  author={Clogg, C. C.},
  booktitle={Handbook of statistical modeling for the social and behavioral sciences},
  pages={311--359},
  year={1995},
  publisher={Springer}
}

@article{HuJofEco2013,
  title={Identification of first-price auctions with non-separable unobserved heterogeneity},
  author={Hu, Y. and McAdams, D. and Shum, M.},
  journal={Journal of Econometrics},
  volume={174},
  number={2},
  pages={186--193},
  year={2013},
  publisher={Elsevier}
}

@book{HagenaarsBook2002,
  title={Applied latent class analysis},
  author={Hagenaars, J. A. and McCutcheon, A. L.},
  year={2002},
  publisher={Cambridge University Press}
}

@article{hand2001idiot,
  title={Idiot's Bayes—not so stupid after all?},
  author={Hand, David J and Yu, Keming},
  journal={International statistical review},
  volume={69},
  number={3},
  pages={385--398},
  year={2001},
  publisher={Wiley Online Library}
}

@article{bordes2006semiparametric,
  title={Semiparametric estimation of a two-component mixture model},
  author={Bordes, Laurent and Mottelet, St{\'e}phane and Vandekerkhove, Pierre},
  journal={The Annals of Statistics},
  volume={34},
  number={3},
  pages={1204--1232},
  year={2006},
  publisher={Institute of Mathematical Statistics}
}

@article{hall2003nonparametric,
  title={Nonparametric estimation of component distributions in a multivariate mixture},
  author={Hall, Peter and Zhou, Xiao-Hua},
  journal={The annals of statistics},
  volume={31},
  number={1},
  pages={201--224},
  year={2003},
  publisher={Institute of Mathematical Statistics}
}

@article{zhu2016theoretical,
  title={Theoretical grounding for estimation in conditional independence multivariate finite mixture models},
  author={Zhu, Xiaotian and Hunter, David R},
  journal={Journal of Nonparametric Statistics},
  volume={28},
  number={4},
  pages={683--701},
  year={2016},
  publisher={Taylor \& Francis}
}

@article{benaglia2009like,
    title = {An EM-like algorithm for semi-and nonparametric estimation in multivariate mixtures},
    author = {Benaglia, T. and Chauveau, D. and Hunter, D. R.},
    journal = {Journal of Computational and Graphical Statistics},
    volume = {18},
    pages = {505--526},
    year = {2009},
    publisher = {Taylor \& Francis}
}

@article{hunter2004tutorial,
  title={A tutorial on MM algorithms},
  author={Hunter, D. R. and Lange, K.},
  journal={The American Statistician},
  volume={58},
  number={1},
  pages={30--37},
  year={2004},
  publisher={Taylor \& Francis}
}

@article{butucea2014semiparametric,
  title={Semiparametric mixtures of symmetric distributions},
  author={Butucea, Cristina and Vandekerkhove, Pierre},
  journal={Scandinavian Journal of Statistics},
  volume={41},
  number={1},
  pages={227--239},
  year={2014},
  publisher={Wiley Online Library}
}

@article{ChauveauSurveys2015,
  title={Semi-parametric estimation for conditional independence multivariate finite mixture models},
  author="Chauveau, D. and Hunter, D. R. and Levine, M.",
  journal={Statistics Surveys},
  volume={9},
  pages={1--31},
  year={2015},
  publisher={The author, under a Creative Commons Attribution License}
}

@ARTICLE{Ban93,
  author = {Banfield, J.D. and Raftery, A.E.},
  title = {{Model-based Gaussian and non-Gaussian clustering}},
  journal = {Biometrics},
  year = {1993},
  pages = {803--821},
  publisher = {JSTOR}
}

@book{Fruhwirth2019handbook,
  title={Handbook of mixture analysis},
  author={Fruhwirth-Schnatter, Sylvia and Celeux, Gilles and Robert, Christian P},
  year={2019},
  publisher={CRC press}
}

@article{van1994bracketing,
  title={Bracketing smooth functions},
  author={van der Vaart, Aad},
  journal={Stochastic Processes and their Applications},
  volume={52},
  number={1},
  pages={93--105},
  year={1994},
  publisher={Elsevier}
}

@article{hansen2008uniform,
  title={Uniform convergence rates for kernel estimation with dependent data},
  author={Hansen, Bruce E},
  journal={Econometric Theory},
  volume={24},
  number={3},
  pages={726--748},
  year={2008},
  publisher={Cambridge University Press}
}

@incollection{van1996weak,
  title={Weak convergence},
  author={Van Der Vaart, Aad W and Wellner, Jon A},
  booktitle={Weak convergence and empirical processes: with applications to statistics},
  pages={16--28},
  year={1996},
  publisher={Springer}
}

@article{baudry2010combining,
  title={Combining mixture components for clustering},
  author={Baudry, Jean-Patrick and Raftery, Adrian E and Celeux, Gilles and Lo, Kenneth and Gottardo, Raphael},
  journal={Journal of computational and graphical statistics},
  volume={19},
  number={2},
  pages={332--353},
  year={2010},
  publisher={Taylor \& Francis}
}

@article{hennig2015true,
  title={What are the true clusters?},
  author={Hennig, Christian},
  journal={Pattern Recognition Letters},
  volume={64},
  pages={53--62},
  year={2015},
  publisher={Elsevier}
}

@article{hennig2010methods,
  title={Methods for merging Gaussian mixture components},
  author={Hennig, Christian},
  journal={Advances in data analysis and classification},
  volume={4},
  number={1},
  pages={3--34},
  year={2010},
  publisher={Springer}
}

@book{lange2016mm,
  title={MM optimization algorithms},
  author={Lange, Kenneth},
  year={2016},
  publisher={SIAM}
}

@book{kosorok2008introduction,
  title={Introduction to empirical processes and semiparametric inference},
  author={Kosorok, Michael R},
  volume={61},
  year={2008},
  publisher={Springer}
}

@article{murphy2000profile,
  title={On profile likelihood},
  author="Murphy, Susan A and Van der Vaart, Aad W",
  journal={Journal of the American Statistical Association},
  volume={95},
  number={450},
  pages={449--465},
  year={2000},
  publisher={Taylor \& Francis}
}

@Article{mixtools,
    title = {{mixtools}: An {R} Package for Analyzing Finite Mixture Models},
    author = {Benaglia, T. and Chauveau, D. and Hunter, D. R. and Young, D.},
    journal = {Journal of Statistical Software},
    year = {2009},
    volume = {32},
    number = {6},
    pages = {1--29}
  }

@book{van2000asymptotic,
  title={Asymptotic statistics},
  author={Van der Vaart, Aad W},
  volume={3},
  year={2000},
  publisher={Cambridge university press}
}

@book{eggermont2001maximum,
  title={Maximum penalized likelihood estimation},
  author={Eggermont, Paulus Petrus Bernardus and LaRiccia, Vincent N and LaRiccia, VN},
  volume={1},
  year={2001},
  publisher={Springer}
}

@book{geer2000empirical,
  title={Empirical Processes in M-estimation},
  author={van der Geer, Sara A},
  volume={6},
  year={2000},
  publisher={Cambridge university press}
}

@article{levine2011maximum,
  title={Maximum smoothed likelihood for multivariate mixtures},
  author={Levine, Michael and Hunter, David R and Chauveau, Didier},
  journal={Biometrika},
  volume={98},
  number={2},
  pages={403--416},
  year={2011},
  publisher={Oxford University Press}
}

@article{allmanAOS09,
author = {Elizabeth S. Allman and Catherine Matias and John A. Rhodes},
title = {{Identifiability of parameters in latent structure models with many observed variables}},
volume = {37},
journal = {The Annals of Statistics},
number = {6A},
publisher = {Institute of Mathematical Statistics},
pages = {3099 -- 3132},
year = {2009},
doi = {10.1214/09-AOS689},
URL = {https://doi.org/10.1214/09-AOS689}
}
\begin{appendix}

\section{Consistency}

\begin{proof}[Proof of Lemma~\ref{lem:taylor}]
   A Taylor expansion of order $2$ of the logarithm implies that
$$\ln \psi_{k,j}(u+vh) = \ln \psi_{k,j}(u) + vh [\ln \psi_{k,j}]'(u) + (vh)^2/2 [\ln \psi_{k,j}]''(u+\alpha_u vh),$$
with $|\alpha_u|\leq 1$. Hence, for any $\psi_{k,j} \in \Psi(\mathcal{X}_j)$,
$$
(\kernel_h \star \ln \psi_{k,j})(u) = \ln \psi_{k,j} (u) + h^{2}\iota_{k,j}(u; \psi_{k,j}),
$$
where 
$$
\iota_{k,j}(u; \psi_{k,j})=\frac{1}{2}\int v^{2}\kernel(v)\left( [\ln\psi_{k,j}]''(u+\alpha_u hv  )\right)dv,
$$
for some $0\leq \alpha_u\leq 1$. Since  $\|[\ln\psi_{k,j}]''\|_\infty\leq C_3$  by definition of $\Psi(\mathcal{X}_j)$  and since $\int v^2\kernel(v) dv$ is finite as we consider a second order kernel (see Assumption~\ref{ass:boundkernel}), then it exists a finite constant $C$, such that
$$
\sup_{\psi_{k,j}\in\Psi(\mathcal{X}_j)} \sup_{u\in\mathcal{X}_j} |\iota_{k,j}(u; \psi_{k,j}) | \leq C.
$$
Hence, we have
\begin{equation} \label{eq:dev}
\smooth_j \psi_{k,j}(u) = \psi_{k,j}(u) \exp[h^2 \iota_{k,j}(u; \psi_{k,j})].
\end{equation}
Hence, using a Taylor expansion of the exponential, we have 
\begin{equation*}
\smooth_j \psi_{k,j}(u)= \psi_{k,j}(u)  + h^{2} \psi_{k,j}(u) \iota_{k,j}(u; \psi_{k,j}) \exp(\beta_u h^2 \iota_{k,j}(u; \psi_{k,j})),
\end{equation*}
for some $0\leq \beta_u \leq 1$.
Hence, since $\psi_{k,j}$ and  $ \iota_{k,j}(.; \psi_{k,j}) $ are bounded uniformly on $\psi_{k,j}$ we have
\color{black}
$$
\sup_{\psi_{k,j} \in\Psi(\mathcal{X}_j)} \|\psi_{k,j} - \smooth_j\psi_{k,j}\|_\infty = O(h^2).
$$
Combining \eqref{eq:dev} and the definition of $f^{(h)}_{\bpi,\bpsi}$ leads to
$$
f^{(h)}_{\bpi,\bpsi}(\bx)   = g_{\bpi,\bpsi}(\bx) \left(1+\sum_{k=1}^K \omega_{\bpi,\bpsi,k}^{(0)}(\bx) \tau^{(h)}_{\bpi,\bpsi,k}(\bx) \right).
$$
with
$$\tau^{(h)}_{\bpi,\bpsi,k}(\bx)=\exp\left[h^2 \sum_{j=1}^J \iota_{k,j}(x_j; \psi_{k,j})\right]-1$$
and $\omega_{\bpi,\bpsi,k}^{(0)}(\bx) $ corresponds to the posterior probabilities of classification obtained without smoothing that satisfies  $0\leq \omega_{\bpi,\bpsi,k}^{(0)}(\bx)\leq 1$ and $\sum_{k=1}^K \omega_{\bpi,\bpsi,k}^{(0)}(\bx)  =1$ and that are defined by
$$
\omega_{\bpi,\bpsi,k}^{(0)}(\bx) = \frac{\pi_k  \prod_{j=1}^J \psi_{k,j}(\bx)}{g_{\bpi,\bpsi}(\bx)}.
$$
By a Taylor expansion of the exponential, we have 
\begin{equation}\label{eq:tau}
\sup_{(\bpi,\bpsi)\in\Theta_K}\max_{k=1,\ldots,K}\|\tau^{(h)}_{\bpi,\bpsi,k} \|_\infty=O(h^2).
\end{equation}
Hence,  noting that $\sup_{(\bpi,\bpsi)\in\Theta_K}\|g_{\bpi,\bpsi}(\bx)\|_\infty \leq C_1^J$, and
$$
f^{(h)}_{\bpi,\bpsi}(\bx) -g_{\bpi,\bpsi}(\bx)  = g_{\bpi,\bpsi}(\bx) \sum_{k=1}^K \omega_{\bpi,\bpsi,k}^{(0)}(\bx)\tau^{(h)}_{\bpi,\bpsi,k}(\bx),
$$
then we have
$$\sup_{(\bpi,\bpsi)\in\Theta_K} \|g_{\bpi,\bpsi} - f^{(h)}_{\bpi,\bpsi}\|_\infty=O(h^2).$$
Using the definition of the loss function for any positive $h$ (see \eqref{eq:loss}) and for $h=0$ (see \eqref{eq:loss0}), we have
$$
\lossh(\bpi,\bpsi) = \loss^{(0)}(\bpi,\bpsi) -  \int_{\mathcal{X}} g^\star(\bx) \ln  \left(1+\sum_{k=1}^K \omega_{\bpi,\bpsi,k}^{(0)}(\bx) \tau^{(h)}_{\bpi,\bpsi,k}(\bx) \right) d\bx.
$$
Therefore, we have
$$
|\lossh(\bpi,\bpsi) - \loss^{(0)}(\bpi,\bpsi) |\leq \int_{\mathcal{X}} g^\star(\bx)\left| \ln  \left(1+\sum_{k=1}^K \omega_{\bpi,\bpsi,k}^{(0)}(\bx) \tau^{(h)}_{\bpi,\bpsi,k}(\bx) \right) \right|d\bx.
$$
Using \eqref{eq:tau}, there exists $h_0>0$ such that for any $h\leq h_0$ we have,
$$\sup_{(\bpi,\bpsi)\in\Theta_K}\max_{k=1,\ldots,K}\left|\sum_{k=1}^K \omega_{\bpi,\bpsi,k}^{(0)}(\bx) \tau_{\bpi,\bpsi,k}^{(h)}(\bx)\right|\leq 1/2,$$
then using the inequality $|\ln (1+u)|\leq 2|u|$ that holds when $u\in[-1/2,1/2]$, if $h<h_0$ we have
$$
|\lossh(\bpi,\bpsi) - \loss^{(0)}(\bpi,\bpsi) | \leq 2 \int_{\mathcal{X}} g^\star(\bx) \left|\sum_{k=1}^K \omega_{\bpi,\bpsi,k}^{(0)}(\bx) \tau_{\bpi,\bpsi,k}^{(h)}(\bx)\right| d\bx.
$$
Therefore, noting that \eqref{eq:tau} combined with the properties of $\omega_{\bpi,\bpsi,k}^{(0)}(\bx)$ implies that 
$$
\sup_{(\bpi,\bpsi)\in\Theta_K} \sup_{\bx\in\mathcal{X}}\left|\sum_{k=1}^K \omega_{\bpi,\bpsi,k}^{(0)}(\bx) \tau_{\bpi,\bpsi,k}^{(h)}(\bx)\right|=O(h^2),
$$
then taking the supremum over $(\bpi,\bpsi)\in\Theta_K$ in both sides of the previous equation leads to
$$\sup_{(\bpi,\bpsi)\in\Theta_K}|\lossh(\bpi,\bpsi) - \loss^{(0)}(\bpi,\bpsi) | =O(h^2).$$
\end{proof}

  \begin{proof}[Proof of Lemma~\ref{lem:deviation}]
    To establish the result, we start by giving some properties of the functional space $\Gamma^{(h)}(\mathcal{X}_j)$ that is defined as the image of the smoothing operator $\smooth_j$  applied to the elements of $\Psi(\mathcal{X}_j)$. Hence, we can define this space as 
    $$
    \Gamma^{(h)}(\mathcal{X}_j) =\{\gamma^{(h)}=\smooth_j \psi_j; \psi_j\in\Psi(\mathcal{X}_j)\}.
    $$
     Since any element of $\Psi(\mathcal{X}_j)$ is strictly positive, then $\ln \psi_j$ with $\psi_j\in\Psi(\mathcal{X}_j)$ is a continuous function $\mathcal{X}_j$. In addition, the kernel is also a continuous function on $\mathcal{X}_j$.  Therefore, by composition of continuous functions, any element of $\Gamma^{(h)}(\mathcal{X}_j) $ is a continuous function on $\mathcal{X}_j$. Let $\gamma^{(h)}$ be a particular element of $\Gamma^{(h)}(\mathcal{X}_j)$, then there exists an element $\psi_j\in\Psi(\mathcal{X}_j)$ such that 
     $$
     \gamma^{(h)} = \exp\left(\kernel_h \star \ln \psi_j\right).
     $$
     Hence, using the fact that the exponential is a non-decreasing function, we have
     $$
     \sup_{u\in\mathcal{X}_j} \gamma^{(h)}(u) = \exp\left[\sup_{u\in\mathcal{X}_j}\left(\kernel_h \star \ln \psi_j\right)(u)\right].
     $$
     Since $\psi_j$ is upper bounded by $C_1>0$ and that the kernel integrate to one over $\mathcal{X}_j$ then
     $$\sup_{u\in\mathcal{X}_j}\left(\kernel_h \star \ln \psi_j\right)(u) \leq \ln C_1,$$
     leading that $\sup_{u\in\mathcal{X}_j} \gamma^{(h)}(u)\leq C_1$. Noting that by construction $\gamma^{(h)}(u) >0$ implies that
     $$
     \| \gamma^{(h)}\|_\infty \leq C_1,
     $$
     leading that $\gamma^{(h)}$ is bounded uniformly on $h$ and on $\Psi(\mathcal{X}_j)$ such that
     $$
     \sup_{\gamma^{(h)} \in \Gamma^{(h)}(\mathcal{X}_j)}     \|\gamma^{(h)}\|_\infty \leq  C_1.
     $$
    We have using the Leibniz integral rule then a variable change,
    \begin{align*}
        \frac{\partial}{\partial u}\left(\left(\kernel_h \star \ln \psi_j\right)(u)\right)&=\frac{\partial}{\partial u}\int_{\mathcal{X}_j} \frac{1}{h}\kernel\left(\frac{u-w}{h} \right) \ln \psi_j(w) dw\\
        &=\int_{\mathcal{X}_j} \frac{\partial}{\partial u}\frac{1}{h}\kernel\left(\frac{u-w}{h} \right) \ln \psi_j(w) dw\\
        &=\frac{1}{h}\int_{\mathcal{X}_j} \kernel'\left(v\right) \ln \psi_j(u-vh) dw
    \end{align*}
       Hence, using the Cauchy-Schwarz inequality, we have
       $$
\int_{\mathcal{X}_j} \kernel'\left(v\right) \ln \psi_j(u-vh) dw \leq \|\kernel'\|_{L_2} \|\ln \psi\|_{L_2},
$$
  since the upper-bound in the previous inequality does not depend on $u$, we have
  $$
  \|\gamma^{(h)'}\|_\infty \leq \frac{1}{h}\|\gamma^{(h)}\|_\infty \|\kernel'\|_{L_2} \|\ln \psi\|_{L_2}.
  $$
  Hence, defining $\bar C_2=C_1  \|\kernel'\|_{L_2} C_2 $, we have $\bar C_2<\infty$ since $ \|\kernel'\|_{L_2}$ is finite by assumption and
     $$
      \sup_{\gamma^{(h)} \in \Gamma^{(h)}(\mathcal{X}_j)}     \|\gamma^{(h)'}\|_\infty =\bar C_2 h^{-1}.
     $$
     Since $\mathcal{X}_j$ is compact, we also have that the $L_2$ norms of any element of $\Gamma^{(h)}(\mathcal{X}_j)$ and its derivative are less than $\tilde C_1$ and $\tilde C_2 h^{-1/2}$ respectively, where $\tilde C_1$ is the product between  $\bar C_1$ and the length of $\mathcal{X}_j$ and where $\tilde C_2$ is the product between  $\bar C_2$ and the length of $\mathcal{X}_j$. Therefore, we define $\mathcal{W}^{1,2,r}\color{black}(\mathcal{X}_j)$ as the Sobolev class of order 1 with radius $r$ with respect to the norm $\|\cdot\|_{W_1}$ defined by
     \begin{equation} \label{eq:sobolev}
     \mathcal{W}^{1,2,r}(\mathcal{X}_j) =\left\{u: \mathcal{X}_j \mapsto \mathbb{R},\, \|u\|_{W_1}\leq r\right\},
     \end{equation} 
     where for any univariate function $u$ we define $\|u\|_{W_1}^2=\|u\|_{L_2}^2 + \|u'\|_{L_2}^2$,      we have
     $$
     \Gamma^{(h)}(\mathcal{X}_j) \subseteq     \mathcal{W}^{1,2,\tilde C_1 + \tilde C_2 h^{-1/2}}(\mathcal{X}_j).
     $$
     Let $N_{[]}(\varepsilon,\mathcal{G},\|.\|)$ be the smallest value of $N$ for which there exist pairs of function $\{[g_j^L,g_j^u]\}_{j=1}^N$ such that $\|g_j^u - g_j^L\| \leq \varepsilon$ for all $j=1,\ldots,N$ and such that for any $g\in\mathcal{G}$ there is a $j=j(g)\in\{1,\ldots,N\}$ such that $g_j^L\leq g \leq g_j^u$. Then $\mathcal H(\varepsilon,\mathcal{G},\|.\|)=\ln N_{[]}(\varepsilon,\mathcal{G},\|.\|)$ is the $\varepsilon$-entropy with bracketing of $\mathcal{G}$. Using the property of the Sobolev class, \cite[Theorem 2.7.1]{van1996weak} (see also \cite[Theorem 2.4]{geer2000empirical} or \cite[Example 19.10]{van2000asymptotic}) states that the $\varepsilon$-entropy with bracketing of a Sobolev class with radius $1$ is upper-bounded as follows 
     $$
    \mathcal H(\varepsilon,   \mathcal{W}^{1,2,1}(\mathcal{X}_j),\|.\|_\infty) \lesssim  \frac{1}{\varepsilon} ,
     $$
     where $a\lesssim b$ means that there exists a positive constant $C$ such that $a\leq Cb$. 
     For any radius $r>0$, $\mathcal{W}^{1,2,r}(\mathcal{X}_j)$ can be defined with a $r$ scaling factor of the elements  of $\mathcal{W}^{1,2,1}(\mathcal{X}_j)$ such that
     $$
 \mathcal{W}^{1,2,r}((\mathcal{X}_j)=\{ rw; w\in  \mathcal{W}^{1,2,1}((\mathcal{X}_j)\}.
     $$
     Hence, we have the following relation between the entropies with bracketing
     $$
      \mathcal H(\varepsilon,  \mathcal{W}^{1,2,r}(\mathcal{X}_j),\|.\|_\infty)= \mathcal H(\varepsilon/r,  \mathcal{W}^{1,2,1}(\mathcal{X}_j),\|.\|_\infty).
     $$
    Using the previous equation with $r=\tilde C_1 + \tilde C_2 h^{-1}$ and the upper bound stated for $    \mathcal H(\varepsilon,  \mathcal{W}^{1,2,1}(\mathcal{X}_j),\|.\|_\infty)$, we have
     $$
 \mathcal    H(\varepsilon;  \Gamma^{(h)}(\mathcal{X}_j),\|.\|_\infty) \lesssim \frac{1}{\varepsilon h} .
     $$
     Therefore, the $\varepsilon$-entropy with bracketing of the $J$-dimensional product space $\Gamma^{(h)}(\mathcal{X})=\Gamma^{(h)}(\mathcal{X}_1)\times \ldots \times \Gamma^{(h)}(\mathcal{X}_J)$ is
     $$
 \mathcal    H(\varepsilon;  \Gamma^{(h)}(\mathcal{X}),\|.\|_\infty) \lesssim \frac{1}{\varepsilon h} .
     $$
Let $\tau^{(h)}_{\bpi,\bpsi}=\ln f^{(h)}_{\bpi,\bpsi}$, considering the space
     $$
     T_{h}(\mathcal{X})=\{\tau^{(h)}_{\bpi,\bpsi},\, (\bpi, \bpsi)\in\Theta_K\},
     $$
     we have
$$
 \mathcal    H(\varepsilon;  T_h(\mathcal{X}),\|.\|_\infty) \lesssim \frac{1}{\varepsilon h} .
     $$
   Since, we have
            $$
     \int_0^\delta H^{1/2}(\varepsilon;  T_h(\mathcal{X}),\|.\|_\infty) d\varepsilon \lesssim h^{-1/2}\delta,
     $$
     then using \cite[Lemma 19.38]{van2000asymptotic}, we have
     $$
     \Eg \left[\sup_{(\bpi,\bpsi)\in\Theta_K(\mathcal{X})} \left|\frac{1}{n^{1/2}} \sum_{i=1}^n \ln f^{(h)}_{\bpi,\bpsi}(\bX_i) - \Eg\ln \ f^{(h)}_{\bpi,\bpsi}(\bX_i) \right|\right]=O(h^{-1/2}).
     $$
     The proof is concluded by noting that for any $(\bpi,\bpsi)$, we have $$ \lossnh(\bpi,\bpsi) -  \lossh(\bpi,\bpsi)=n^{-1/2}\left|\frac{1}{n^{1/2}} \sum_{i=1}^n \ln  f^{(h)}_{\bpi,\bpsi}(\bX_i) - \Eg\ln  f^{(h)}_{\bpi,\bpsi}(\bX_i) \right|,$$
     then by applying Markov's inequality.
 \end{proof}

\begin{proof}[Proof of Theorem~\ref{thm:consistence}]
Note that by triangular inequality, we have
$$
|\lossnh(\bpi,\bpsi) - \loss^{(0)}(\bpi,\bpsi)| \leq |\lossnh(\bpi,\bpsi) - \lossh(\bpi,\bpsi)| + |\lossh(\bpi,\bpsi) - \loss^{(0)}(\bpi,\bpsi)|.
$$
Combing Lemmas~\ref{lem:taylor} and~\ref{lem:deviation} provides
$$
\sup_{(\bpi,\bpsi)\in\Theta_K}|\lossnh(\bpi,\bpsi) - \loss^{(0)}(\bpi,\bpsi)| = O_\mathbb{P}(n^{-1/2}h^{-1/4} + h^2).
$$
Using the conditions on the bandwidth  establishes the uniform convergence in probability of $\lossnh$ to $\loss^{(0)}$ meaning that
$$
\sup_{(\bpi,\bpsi)\in\Theta_K} |\lossnh(\bpi,\bpsi) - \loss^{(0)}(\bpi,\bpsi)| = o_\mathbb{P}(1).
$$
We now establish the convergence in probability of $\loss^{(0)}(\hpi,\hpsi)$ to $\loss^{(0)}(\bpi^\star,\bpsi^\star)$. Due to the uniform convergence in probability of $\lossnh$ to $\loss^{(0)}$, we have $$\lossnh(\hpi,\hpsi)-\loss^{(0)}(\hpi,\hpsi)= o_{\mathbb{P}}(1),$$ leading that
\begin{equation*}
  \loss^{(0)}(\hpi,\hpsi)- \loss^{(0)}(\bpi^\star,\bpsi^\star)  =    \lossnh(\hpi,\hpsi)- \loss^{(0)}(\bpi^\star,\bpsi^\star) + o_\mathbb{P}(1).
\end{equation*}
 It remains to show that the difference $\lossnh(\hpi,\hpsi)- \loss^{(0)}(\bpi^\star,\bpsi^\star)$ is also $o_{\mathbb{P}}(1)$. Since $(\bpi^\star,\bpsi^\star)$ is the minimizer of $\loss^{(0)}$, we have $ \loss^{(0)}(\bpi^\star,\bpsi^\star) \leq \loss^{(0)}(\hpi,\hpsi) $ leading, using the uniform convergence in probability of $\lossnh$, that we have 
\begin{equation*}\label{eq:dessous}
   \loss^{(0)}(\bpi^\star,\bpsi^\star) \leq  \lossnh(\hpi,\hpsi) + o_{\mathbb{P}}(1).
\end{equation*}
Since $\hat\theta_{h,n}$ is a minimizer of $\lossnh$, we have $    \lossnh(\hpi,\hpsi) \leq \lossnh(\bpi^\star,\bpsi^\star)$ leading, using the uniform convergence in probability of $\lossnh$, that we have
\begin{equation*}\label{eq:dessus}
    \lossnh(\hpi,\hpsi) \leq  \loss^{(0)}(\bpi^\star,\bpsi^\star) + o_{\mathbb{P}}(1).
\end{equation*}
By combining the two last inequalities, we obtain that 
\begin{equation}\label{eq:cvprloss}
\lossnh(\hpi,\hpsi)- \loss^{(0)}(\bpi^\star,\bpsi^\star) = o_{\mathbb{P}}(1),
\end{equation}
which concludes the proof of the convergence in probability of $\loss^{(0)}(\hpi,\hpsi)$ to $\loss^{(0)}(\bpi^\star,\bpsi^\star)$ meaning
$$|\loss^{(0)}(\hpi,\hpsi) - \loss^{(0)}(\bpi^\star,\bpsi^\star)| = o_\mathbb{P}(1).$$

Now we conclude that $(\hpi,\hpsi)$ converges in probability to $(\bpi^\star,\bpsi^\star)$. Note that, for any $j= 1, \ldots, J$, $\mathcal{X}_j$ is a compact space. Hence, considering the supremum norm implies that $\Psi(\mathcal{X}_j)$ is equicontinuous because it is composed of Sobolev functions of order 1 defined on a compact space. In addition, the elements of $\Psi(\mathcal{X}_j)$ are uniformly bounded by $C_1$. Therefore, by the Arzelà–Ascoli theorem, $\Psi(\mathcal{X}_j)$ has the sequential compactness property, and so does $\Theta_K$. Since $\Theta_K$ is defined as a product of compact spaces, it is itself compact. Suppose, for the sake of contradiction, that $(\hpi,\hpsi)$ does not converge to $(\bpi^\star,\bpsi^\star)$. As the parameter space is sequentially compact, one can find a subsequence $(\hat\bpi_{h,n_k},\hat\bpsi_{h,n_k})_k$ which converges in probability to some $(\widetilde\bpi,\widetilde\bpsi) \neq (\bpi^\star,\bpsi^\star)$. By the continuity of $\loss^{(0)}$, $\loss^{(0)}(\hat\bpi_{h,n_k},\hat\bpsi_{h,n_k})$ converges in probability to $\loss^{(0)}(\widetilde\bpi,\widetilde\bpsi)$. On the other hand, by \eqref{eq:cvprloss}, $\loss^{(0)}(\hat\bpi_{h,n_k},\hat\bpsi_{h,n_k})$ converges in probability to $\loss^{(0)}(\bpi^\star,\bpsi^\star)$. Therefore, we have $\loss^{(0)}(\bpi^\star,\bpsi^\star) = \loss^{(0)}(\widetilde\bpi,\widetilde\bpsi)$. Recall that $(\hat\bpi_{h,n_k},\hat\bpsi_{h,n_k}) \neq (\bpi^\star,\bpsi^\star)$. This contradicts the parameter identifiability property ensured by Assumption~\ref{ass:controlvariancetrue}, which implies that $(\bpi^\star,\bpsi^\star)$ is the unique minimizer of $\loss^{(0)}$. Therefore, $(\hpi,\hpsi)$ converges in probability to $(\bpi^\star,\bpsi^\star)$.

\end{proof}

\section{Profiling the loss functions}

\begin{proof}[Proof of Lemma~\ref{lem:unicity}]
First, as in the Appendix of \citep{levine2011maximum} , let us define a set $B$ containing all possible functions $M_{h,\bpi}^{\{p\}}[\bpsi]$ except, possibly, the initial $\bpsi^{0}$. Under Assumption~\ref{ass:controlvariancekernel}.\ref{ass:boundkernel}, from  \cite[Lemma 4.1]{zhu2016theoretical}, we find that the functional  $\bpsi \mapsto \lossnh (\bpi,\bpsi)$ is well defined on the set $B$ since it is bounded from below by -$\ln b_2(h)$.   Lemma A3 of \citep{levine2011maximum}  guarantees lower semicontinuity and the strict convexity of any function belonging to the set $B$. Hence, for any $\bpsi$, the sequence $\bpsi,M^{(h,n,\bpi)}[\bpsi],M^{(h,n,\bpi)\{2\}}[\bpsi],\ldots$ converges to a global minimizer of the objective function $\lossnh$, where $f^{\{p\}}$ denotes $p$ iterations of function $f$ in the sense that $f^{\{p\}}=f(f^{\{p-1\}})$. 
Assumptions~\ref{ass:controlvariancetrue}.\ref{ass:indptlin} and the fact that $\pi_k>0$ as $\bpi$ is at the interior of $\mathcal{S}_K^r$ ensure the identifiability of the parameters of the density $g_{\bpi,\bpsi^\star}$. Indeed, since the proportions are known, fixed and different, there is no possibility that label swapping defines the same distribution and thus, $\lossnh(\bpi,\bpsi)$ has a single global minimizer $\hpsipi$ when $\bpi$ is fixed. Hence, any sequence $\bpsi,M^{(h,n,\bpi)}[\bpsi],M^{(h,n,\bpi)\{2\}}[\bpsi],\ldots $ converges to $\hpsipi$ meaning that
\begin{equation} \label{eq:iteration}
    \forall \bpsi\in\Psi_{K}(\mathcal{X}),\, \lim_{p\to\infty} M^{(h,n,\bpi)\{p\}}[\bpsi]=\hpsipi.
\end{equation}
As a direct consequence of \cite[Corollary 3.2]{zhu2016theoretical}, if $\hpsipi$ is a minimizer of $\lossnh(\bpi,\bpsi)$ with respect to $\bpsi$  then $M^{(h,n,\bpi)}[\hpsipi]=\hpsipi$. Now, suppose that there exists $\bar\bpsi$ such that  $M^{(h,n,\bpi)}[\bar\bpsi]=\bar\bpsi$ and $\bar\bpsi\neq \hpsipi$. Obviously, we have $\lim_{p\to\infty} M^{(h,n,\bpi)\{p\}}[\bar\bpsi]=\bar\bpsi\neq \hpsipi$ which contradict \eqref{eq:iteration}. Hence, $\hpsipi$ is the unique fixed point of $M^{(h,n,\bpi)}[\bpsi]$. The result on $\lossh(\bpi,\bpsi)$ follows by the same argument as for $\lossnh(\bpi,\bpsi)$ above.
\end{proof}

\section{Control of the estimators of the finite dimensional parameters}

\begin{proof}[Proof of Lemma~\ref{lem:prel}]
Let
\begin{equation}\label{eq:mu}
\mu^{(h,n,\pi)}(\bpsi)= \lossnh(\bpi,\bpsi) - \lossnh(\bpi,M^{(h,n,\bpi)}[\bpsi]) .    
\end{equation}
By definition, we have
$$
 \mu^{(h,n,\pi)}(\bpsi) = \frac{1}{n} \sum_{i=1}^n \ln \frac{\sum_{k=1}^K \pi_k \prod_{j=1}^J\smooth_j M_{k,j}^{(h,n,\bpi)}[\bpsi](X_{i,j})}{\sum_{k=1}^K \pi_k \prod_{j=1}^J\smooth_j\psi_{k,j}(X_{i,j})}.
$$
Using the definition of $\omega_{k}^{(h)}$ given by \eqref{eq:weights}, we have
$$\mu^{(h,n,\pi)}(\bpsi)= \frac{1}{n} \sum_{i=1}^n \ln  \sum_{k=1}^K \omega_{\bpi,\bpsi,k}^{(h)}(\bx) \frac{\prod_{j=1}^J\smooth_j M_{k,j}^{(h,n,\bpi)}[\bpsi](X_{i,j})}{\prod_{j=1}^J\smooth_j\psi_{k, j}(X_{i,j})}.
$$
Therefore, using Jensen's inequality, we have
$$
\mu^{(h,n,\pi)}(\bpsi)  \geq   \frac{1}{n} \sum_{i=1}^n \sum_{k=1}^K \omega_{k}^{(h)}(\bX_i;\bpi, \bpsi )   \ln\frac{\prod_{j=1}^J\smooth_j M_{k,j}^{(h,n,\bpi)}[\bpsi](X_{i,j})}{\prod_{j=1}^J\smooth_j\psi_{k, j}(X_{i,j})}.    
$$
This implies that we have
$$
\mu^{(h,n,\pi)}(\bpsi) \geq \sum_{k=1}^K \sum_{j=1}^J  \eta_{k,j}^{(h,n,\bpi)}(\bpsi).
$$
with 
$$
\eta_{k,j}^{(h,n,\bpi)}(\bpsi)=\frac{1}{n}\sum_{i=1}^n \omega_{\bpi,\bpsi,k}^{(h)}(\bX_i) \ln  \frac{ \smooth_j M_{k,j}^{(h,n,\bpi)}[\bpsi](X_{i,j})}{ \smooth_j\psi_{k, j}(X_{i,j})}
$$
Using the definition of the smoothing $\smooth_j$, we have
\begin{align*}
\eta_{k,j}^{(h,n,\bpi)}(\bpsi) &= \frac{1}{n  } \sum_{i=1}^n \omega_{\bpi,\bpsi,k}^{(h)}(\bX_i) \int_{\mathcal{X}_j} \frac{1}{h}\kernel\left(\frac{u-X_{i,j}}{h} \right)\ln  \frac{ M_{k,j}^{(h,n,\bpi)}[\bpsi] (u)}{ \psi_{k, j}(u)} du\\
    &=\int_{\mathcal{X}_j} \frac{1}{nh  } \sum_{i=1}^n \omega_{\bpi,\bpsi,k}^{(h)}(\bX_i) \kernel\left(\frac{u-X_{i,j}}{h} \right)\ln  \frac{  M_{h,n,k j}[\bpsi](u)}{ \psi_{k, j}(u)} du\\
    &=\pi_k \int_{\mathcal{X}_j}M_{k,j}^{(h,n,\bpi)}[\bpsi](u)\ln  \frac{ M_{k,j}^{(h,n,\bpi)}[\bpsi](u)}{\psi_{k, j}(u)} du\\
    &=\pi_k \KL(M_{k,j}^{(h,n,\bpi)}[\bpsi],\psi_{k,j}),
\end{align*}
where the last line is obtained by noting that $\int_{\mathcal{X}_j}\psi_{k,j}(u)du=\int_{\mathcal{X}_j} M_{k,j}^{(h,n,\bpi)}[\bpsi](u)du=1$. Hence, we have
$$
\mu^{(h,n,\pi)}(\bpsi) \geq  \sum_{k=1}^K \pi_k \sum_{j=1}^J   \KL(M_{k,j}^{(h,n,\bpi)}[\bpsi],\psi_{k,j}).
$$
The Kullback Leibler divergence can be lower-bounded by the $L_1$-norm as follows \citep[(3.21), p.16]{eggermont2001maximum}, using for instance the Pinsker's inequality,
$$
\KL(g_1,g_2) \geq \frac{1}{4} \|g_1 - g_2\|^2_{1},
$$
with $\|g_1 - g_2\|_1^2 = \int |g_1 - g_2|$. Therefore, we have
$$
\mu^{(h,n,\pi)}(\bpsi)\geq  \frac{1}{4  }    \sum_{k=1}^K \pi_k \sum_{j=1}^J  \|M_{k,j}^{(h,n,\bpi)}[\bpsi] - \psi_{k,j}\|_1^2.
$$
\end{proof}

\begin{proof}[Proof of Lemma~\ref{lem:inequality}]
For any positive integer $q$, using the definition of $\mu^{(h,n,\bpi)}$ stated by \eqref{eq:mu}, we have 
$$
  \lossnh(\bpi,\bpsi^\star) - \lossnh(\bpi,M^{(h,n,\bpi)\{q\}}[\bpsi^\star]) = \sum_{r=0}^{q-1} \mu^{(h,n,\bpi)}(M^{(h,n,\bpi)\{r\}}[\bpsi^\star]),
$$
with the convention $M^{(h,n,\bpi)\{0\}}[\bpsi] =\bpsi$. Hence, by applying Lemma~\ref{lem:prel} to have a lower bound of each term that appears in the sum of the right hand side of the previous equation, we have 
 \begin{multline*}
  \lossnh(\bpi,\bpsi^\star) - \lossnh(\bpi,M^{(h,n,\bpi)\{q\}}[\bpsi^\star])  \geq \\\frac{1}{4} \sum_{k=1}^K \pi_k \sum_{j=1}^J \sum_{r=0}^{q-1} \| M^{(h,n,\bpi)\{r\}}_{k,j}[\bpsi^\star]- M^{(h,n,\bpi)\{r+1\}}_{k,j}[\bpsi^\star]\|^2_1.     
 \end{multline*}
Using the triangular inequality to have an lower bound of the right-hand side of the previous inequality gives us 
   $$
   \lossnh(\bpi,\bpsi^\star) - \lossnh(\bpi,M^{(h,n,\bpi)\{q\}}[\bpsi^\star])  \geq  \frac{1}{4} \sum_{k=1}^K \pi_k \sum_{j=1}^J \| \psi^\star_{k,j} - M^{(h,n,\bpi)\{q\}}_{k,j}[\bpsi^\star]\|^2_1.
 $$
 We now aim to takes the limit as $q$ tends to infinity for both sides of the previous inequality. 
Considering Lemma~\ref{lem:unicity} with initial value of the MM algorithm equal to $\bpsi^\star$ implies that the sequence $M^{(h,n,\bpi)\{q\}}[\bpsi^\star]$ converges to $\hpsipi$ as $q$ tends to infinity, leading that
$$
\lim_{q\to\infty} \lossnh(\bpi,M^{(h,n,\bpi)\{q\}}[\bpsi^\star]) =\lossnh(\bpi,\hpsipi).
$$
In addition, noting that the weights $\omega_{\bpi,\bpsi,K}^{(h)}(\bX_i)$ are positive and upper-bounded by one,  we have
 $$
M_{k,j}^{(h,n,\bpi)}[\bpsi](u) \leq \frac{1}{ n\pi_k} \sum_{i=1}^n\frac{1}{h}\kernel\left(\frac{X_{i,j} - u}{h} \right).
$$
Using the law of the large numbers and the following control of the variance (\cite{hansen2008uniform}) $$\sup_{u\in\mathbb{R}}\left|\frac{1}{ n} \sum_{i=1}^n\frac{1}{h}\kernel\left(\frac{X_{i,j} - u}{h} \right) - E_{g^\star}\left[\frac{1}{h}\kernel\left(\frac{X_{i,j} - u}{h} \right)\right]\right|=O_\mathbb{P}\left(\frac{\ln^{1/2} n}{(nh)^{1/2}}\right),$$ we have
$$
\sup_{u\in\mathbb{R}} \left|M_{k,j}^{(h,n,\bpi)}[\bpsi](u) - E_{g^\star}\left[\frac{1}{h}\kernel\left(\frac{X_{i,j} - u}{h} \right)\right] \right| \leq  O_\mathbb{P}\left(\frac{\ln^{1/2} n}{(nh)^{1/2}}\right).
$$
since the proportions are not zero. Therefore,  there exists an integrable function that is greater than $M^{(h,n,\bpi)\{q\}} [\bpsi^\star]$ for all integer $q$.  
Hence, the dominated convergence theorem implies that
\begin{align*}
  \| \psi^\star_{k,j} - M^{(h,n,\bpi)\{q\}}_{k,j}[\bpsi^\star]\|^2_1&= \| \psi^\star_{k,j} - \lim_{q\to\infty}M^{(h,n,\bpi)\{q\}}_{k,j}[\bpsi^\star]\|^2_1\\
&=    \| \psi^\star_{k,j} - \widehat{\psi}^{(h,n,\bpi)}_{k,j}\|_1^2.
\end{align*}
Therefore, we have
$$
 \lossnh(\bpi,\bpsi^\star) - \lossnh(\bpi,\hpsipi) \geq  \frac{1}{4} \sum_{k=1}^K \pi_k \sum_{j=1}^J \| \psi^\star_{k,j} -\widehat{\psi}^{(h,n,\bpi)}_{k,j}\|^2_1.
 $$
\end{proof}
\begin{proof}[Proof of Lemma \ref{lem:invert}]
First, let us construct the following discretized analogue of the original model \eqref{eq:model}-\eqref{eq:model2}. For simplicity, let us assume that all of the univariate densities are defined on $[0,1]$. It is assumed that there is a collection of partitions ${\cal I}_{M}$, $M\in {\cal M}$, ${\cal M}\subset N$ so that for each $M\in {\cal M}$, ${\cal I}_{M}=(I_{m})_{m=1}^{M}$ is a partition of $[0,1]$ by Borel sets.  Let us also denote $P^{*}$ the probability measure corresponding to the true distribution of \eqref{eq:model} - \eqref{eq:model2}. Then, a discretized version of \eqref{eq:model} - \eqref{eq:model2} is
\begin{equation}\label{model:disc}
g_{\pi,\omega;M}(\bx)=\sum_{k=1}^{K}\pi_{k}\prod_{j=1}^{J}\left(\sum_{m=1}^{M}\frac{\omega_{k,j,m}}{|I_{m}|}\mathds{1}_{I_{m}}(x_j)\right)
\end{equation}
 where $\omega_{k,j,m}\ge 0$, $\sum_{m=1}^{M}\omega_{k,j,m}=1$, $|I_{m}|$ is the Lebesgue measure of the set $I_{m}$ and $\bomega=\{\omega_{k,j,m}\}$ where $1\le m\le M$, $1\le j\le J$ and $1\le k\le K$. Note that \eqref{model:disc} implies, essentially, that the original target density functions is modeled as a convex combination of products of mixtures of step functions defined as  
 \[
 f_{\omega_{k,j,m}}(x)=\sum_{m=1}^{M}\frac{\omega_{k,j,m}}{|I_{m}|}\mathds{1}_{I_{m}}(x).
 \]
Using this notation, we call $\bomega^{\star}_{M}$ the collection of values obtained by discretizing true univariate densities. Similarly, $\bpi^{\star}$  is the vector of true probability weights.  
Let $S_{M}^{\star}=(S^{\star}_{\bpi,M},S^{\star}_{\bomega,M})$ be the score function of the parameter $(\bpi,\bomega)$ at the point $(\bpi^{\star},\bomega_{M}^{\star})$ in the model \eqref{model:disc}. Explicit expressions for these score functions are given in formulas (5)-(6) of \cite{gassiat2018efficient}. The Fisher information of the discretized model $J_{M}$ is then defined as 
\[
J_{M}=\Eg[S_{M}^{\star}(X)S_{M}^{\star}(X)^\top]
\]
Now, let us partition this matrix according to the parameters $\bpi$ and $\bomega$, denoting corresponding blocks $[J_{M}]_{\bpi,\bpi},$ $[J_{M}]_{\bomega,\bomega}$ and $[J_{M}]_{\bpi,\bomega},$ respectively. Let us denote $\tilde{\nu}_{M}$ the efficient score function for the estimation of $\bpi$ 
\[
\tilde{\nu}_{M}=S_{\bpi,M}^{\star}-[J_{M}]_{\bpi,\bomega}([J_{M}]_{\bomega,\bomega})^{-1}S_{\bomega,M}^{\star}
\]
and the efficient Fisher information $\tilde{J}_{M}$ (a $(k-1)\times (k-1)$ matrix) 
\[
\tilde{J}_{M}=[J_{M}]_{\bpi,\bomega}([J_{M}]_{\bomega,\bomega})^{-1}[J_{M}]^\top_{\bpi,\bomega}.
\]
The first step of our argument is provided by Proposition $1$ of \cite{gassiat2018efficient} that proves non-singularity of $\tilde{J}_{M}$ for a sufficiently large $M$. Note that the assumptions of Proposition $1$ of \cite{gassiat2018efficient} are satisfied due to Assumptions~\ref{ass:controlvariancetrue} and the fact that each $\psi_{k,j}^\star\in \Psi({\mathcal X}_{j})$ and ${\mathcal X}_{j}$ is compact belongs to the compact space $j=1,\ldots,J$. Indeed, by definition of $\Psi(\mathcal{X}_j)$ each $C_1>\psi_{k,j}^\star>0$, while the compactness of ${\mathcal X}_{j}$ ensures that $\psi_{k,j}^\star$ is bounded away from zero. Therefore, the ratio $\psi_{k,j}^\star/\psi_{k',j}^\star$ is bounded away from zero and infinity for any $(k,k')$ and any $j$.

With this in mind, the desired result is due to the existence of spectral estimators of components of a discretized model \eqref{model:disc} first obtained in \cite{anandkumar2014tensor}. \cite{anandkumar2014tensor} also established the differentiability of the multinomial model \eqref{model:disc} in quadratic mean; this, together with the use of van Trees inequality, results in non-singularity of $\tilde{J}_{M}$ for a sufficiently large $M.$ The next step relies on Lemma $1$ of \cite{gassiat2018efficient} that proves convergence of the sequence $\tilde{J}_{M}$ to the limiting matrix $J$ which is necessarily non-singular. More can be obtained from careful reading of the proof of Lemma 1 in \cite{gassiat2018efficient}. There, the efficient score function $\tilde{\nu}_{M}$ is defined; next, it is shown that this function converges almost sure to the limiting efficient score function equivalent to our $\tilde{\ell}_{\bpi^{*},\bpsi^{*},k}$. This convergence is established in the proof of Lemma 1 of \cite{gassiat2018efficient} using only consistency of spectral estimators of $\bpi$ proposed in \cite{anandkumar2014tensor}. Indeed, the crucial argument is to construct a consistent estimator. In their paper, \cite{gassiat2018efficient} use a bin approximation with an increasing number of bins, but the argument still holds if a kernel-based estimator is built with a bandwidth tending to zero. 
It is shown next this convergence implies $L_{2}(g_{\bpi^{\star},\bpsi^{\star}})$ convergence. This, in its own turn, implies that the limit of the sequence of $\tilde{J}_{M}$, that we denoted $J$ earlier, is equal to $\Sigma_{\bpi^{\star},\bpsi^{\star}}.$

The above argument assumes that the sample size is equal to $1.$ To extend this argument to an arbitrary sample size $n$, let us first denote $\hat\pi_{M}$ the maximum likelihood estimator of the weight parameters of the discretized model \eqref{model:disc}. Let $\sigma_{n,M}$ also be a sequence of permutations of the set $\{1,2,\ldots,k\}$ for a given $M$. For an arbitrary sample size $n$, using Theorem 5.39 in \cite{van2000asymptotic}, we find that for each $M$, the MLE $\hat\pi_{M}$ is regular and asymptotically efficient:
\[
\sqrt{n}\left(\hat\pi_{M}^{\sigma_{n,M}}-\pi^{*}\right)=
\frac{\tilde J_{M}^{-1}}{\sqrt{n}}\sum_{i=1}^{n}\tilde\nu_{M}(X_i)+R_{n}(M)
\]
with $R_{n}(M)$ being a sequence of random vectors converging to zero in $g^{\star}$-probability as $n\rightarrow \infty$. Thus, we can say that there exists a sequence $M_{n}$ that tends to infinity sufficiently slowly so that, as $n\rightarrow \infty$, $R_{n}(M)$ tends to zero in $g^{\star}$-probability. The detailed discussion can be found in the proof of Theorem $1$ of \cite{gassiat2018efficient}. Now we can say that the corresponding sequence of matrices $\tilde J_{M_{n}}^{-1}$ converges to $\tilde J^{-1}$ which is non-singular due to Lemma $1$ of \cite{gassiat2018efficient}.
\end{proof}
\begin{proof}[Proof of Theorem~\ref{thm:rateL1}]
Lemma~\ref{lem:inequality} provides a bound on the sum of the squared $L_1$ norms of the differences between the true functions $\psi^\star_{k,j}$ and their corresponding estimators obtained for fixed proportions. This bound is expressed in terms of the difference between the empirical loss function evaluated at these two parameter values, as follows
$$
\forall \bpi\in\mathcal{S}_K^r,\,\sum_{k=1}^K   \sum_{j=1}^J \| \psi^\star_{k,j} -\widehat{\psi}^{(h,n,\bpi)}_{k,j}\|^2_1 \leq 4 \frac{1}{\min \pi_k}( \lossnh(\bpi,\bpsi^\star) - \lossnh(\bpi,\hpsipi) ).
$$
Let $\mathcal{B}(\bpi^\star)$ be the ball centered in $\bpi^\star$ with radius equal to $\min \pi^\star_k/2$. Since any $\pi^\star_k$ is strictly positive, then for any $\bpi\in\mathcal{B}(\bpi^\star)$, there exists a positive constant that is greater or equal to $\frac{1}{\min \pi_k}$. Thus, there exists a positive constant $A$, such that
$$
\forall \bpi\in\mathcal{B}(\bpi^\star),\, \sum_{k=1}^K  \sum_{j=1}^J \| \psi^\star_{k,j} -\widehat{\psi}^{(h,n,\bpi)}_{k,j}\|^2_1 \leq A ( \lossnh(\bpi,\bpsi^\star) - \lossnh(\bpi,\hpsipi) ).
$$
 From Lemma~\ref{lem:deviation}, replacing the empirical version of the loss function by its theoretical version, without changing the bandwidth, leads to a term of stochastic order $n^{-1/2}h^{-1/4}$ uniformly on $(\bpi,\bpsi)$, leading that  
\begin{multline}\label{eq:upper}
 \forall \bpi\in\mathcal{B}(\bpi^\star),\, \sum_{k=1}^K \sum_{j=1}^J \| \psi^\star_{k,j} - \widehat{\psi}^{(h,n,\bpi)}_{k,j}\|^2_1 \leq O_\mathbb{P}( n^{-1/2}h^{-1/4} )+ \\A( \lossh(\bpi,\bpsi^\star) - \lossh(\bpi,\hpsipi) ).
\end{multline}
For any $\bpsi \in \Psi_K(\mathcal{X})$ and $\bpi\in\mathcal{B}(\bpi^\star)$, we have
\begin{align*}
 |\lossh(\bpi^\star,\bpsi)- \lossh(\bpi,\bpsi) |&= \left\vert\int_{\mathcal{X}} g^\star(\bx) \ln \frac{\sum_{k=1}^K \pi_k \smooth\psi_k(\bx)}{\sum_{\ell=1}^K \pi_\ell^\star \smooth\psi_\ell(\bx)} d\bx \right\vert   \\
 &= \left\vert\int_{\mathcal{X}} g^\star(\bx) \ln \left[1 +\frac{\sum_{k=1}^K (\pi_k-\pi_k^\star) \smooth \psi_k(\bx)}{\sum_{\ell=1}^K \pi_\ell^\star \smooth \psi_\ell(\bx)} \right]d\bx\right\vert    \\
 &\leq \int_{\mathcal{X}} g^\star(\bx) \left| \ln \left[1 +\frac{\sum_{k=1}^K (\pi_k-\pi_k^\star) \smooth\psi_k(\bx)}{\sum_{\ell=1}^K \pi_\ell^\star \smooth\psi_\ell(\bx)} \right]\right|d\bx.    
\end{align*}
Note that
\begin{align*}
\left| \frac{\sum_{k=1}^K (\pi_k^\star-\pi_k) \smooth\psi_k(\bx)}{\sum_{\ell=1}^K \pi_\ell^\star \smooth\psi_\ell(\bx)}\right| &=\left|\sum_{k=1}^K (\pi_k^\star-\pi_k) \frac{ \smooth\psi_k(\bx)}{\sum_{\ell=1}^K \pi_\ell^{\star} \smooth\psi_\ell(\bx)}\right|  \\
&\leq    \sum_{k=1}^K  \left|(\pi_k^\star-\pi_k) \frac{\smooth \psi_k(\bx)}{\sum_{\ell=1}^K \pi_\ell^{\star} \smooth\psi_\ell(\bx)}\right| .
\end{align*}
Since $\pi_{k}^{\star}$ and $\smooth\psi_{k}$ non negative, then 
$\sum_{\ell=1}^K \pi_\ell^{\star} \smooth\psi_\ell(\bx) \geq  \pi_k^{\star} \smooth\psi_k(\bx)$, for any particular $k=1,\ldots,K$.
This, in its own turn, implies that
$$
\frac{\smooth \psi_k(\bx)}{\sum_{\ell=1}^K \pi_\ell^\star \smooth\psi_\ell(\bx)}\leq \frac{1}{\pi_k^{\star}}.
$$
and
$$
\left| \frac{\sum_{k=1}^K (\pi_k^\star-\pi_k) \smooth\psi_k(\bx)}{\sum_{k=1}^K \pi_k^\star \smooth\psi_k(\bx)}\right|\leq \sum_{k=1}^K\left| \frac{\pi_k-\pi_k^\star}{\pi_k^{*}}\right|.
$$
Hence, if \begin{equation} \label{eq:radius} \|\bpi - \bpi^\star\|_\infty\leq \frac{\min_k \pi_k^\star}{2K}, \end{equation} 
then we have
$$
\left| \frac{\sum_{k=1}^K (\pi_k^\star-\pi_k) \smooth\psi_k(\bx)}{\sum_{k=1}^K \pi_k^{\star} \smooth\psi_k(\bx)}\right|\leq 1/2.
$$
Since $|\ln(1+u)|\leq 2|u|$ for any $|u|\leq 1/2$, we have
$$ \sup_{\bpsi \in \Psi_K(\mathcal{X})} |\lossh(\bpi^\star,\bpsi)- \lossh(\bpi,\bpsi) |\leq 2\sum_{k=1}^K  \left|\frac{\pi_k-\pi_k^\star}{\pi_k^\star}\right|. $$ 
Since $\pi_k^\star>0$ there exists a positive constant $C$ such that 
$$  \sup_{\bpsi \in \Psi_K(\mathcal{X})} |\lossh(\bpi^\star,\bpsi)- \lossh(\bpi,\bpsi) |\leq C \|\bpi-\bpi^\star\|_1. $$ 
By definition of $\tpsipis$ given by \eqref{eq:tpsipi}, we have $$\lossh(\bpi^\star, \widehat{\bpsi}^{(h,n,\bpi^\star)})\geq  \lossh(\bpi^\star,\tpsipis).$$  Therefore,  from \eqref{eq:upper}, we have
\begin{multline*}
 \forall \bpi\in\mathcal{B}(\bpi^\star),\,\sum_{k=1}^K \sum_{j=1}^J \| \psi^\star_{k,j} - \widehat{\psi}^{(h,n,\bpi)}_{k,j}\|^2_1 \leq O_\mathbb{P}( n^{-1/2}h^{-1/4}+\|\bpi - \bpi^\star\|_1  )\\+ A( \lossh(\bpi^\star,\bpsi^\star) - \lossh(\bpi^\star,\tpsipis) ).
\end{multline*}
Noting that $\loss^{(0)}(\bpi^\star,\bpsi^\star)=0$, then from Lemma~\ref{lem:taylor}, we have 
$$
\lossh(\bpi^\star,\bpsi^\star)=O(h^2).
$$
Using the definition of the Kullback-Leibler divergence, we have
$$
\lossh(\bpi^\star,\tpsipis) = \KL(g^\star,f^{(h)}_{\bpi^\star,\tpsipis}) + \sum_{k=1}^K \pi^\star_k \int_{\mathcal{X}} \smooth {\bpsi}_{k}^{(h,\bpi^\star)}(\bx) d\bx - \int_{\mathcal{X}} g^\star(\bx)d\bx.
$$
The Kullback-Leibler divergence is positive, $\int_{\mathcal{X}} g^\star(\bx)d\bx=1$ and 
$\int_{\mathcal{X}} \smooth {\bpsi}_{k}^{(h,\bpi^\star)}(\bx) d\bx= 1 +O(h^2)$ by Lemma~\ref{lem:taylor}. 
Hence, since by definition of $\tpsipis$ we have $$\lossh(\bpi^\star,\bpsi^\star) - \lossh(\bpi^\star,\tpsipis)\geq 0,$$ then, we have
$$
\lossh(\bpi^\star,\bpsi^\star) - \lossh(\bpi^\star,\tpsipis)= O(h^2).
$$
Hence, we have
$$
\sum_{k=1}^K \sum_{j=1}^J \| \psi^\star_{k,j} - \widehat{\psi}^{(h,n,\bpi)}_{k,j}\|^2_1 = O_\mathbb{P}( n^{-1/2}h^{-1/4}+h^2+\|\bpi - \bpi^\star\|_1).
$$
\end{proof}

\section{Control of the estimators of the finite dimensional parameters}
 
\begin{proof}[Proof of Proposition~\ref{thm:convparam}]
Recall that $\bt\in\mathcal{S}_K^r$, so its last element $t_K=1-\sum_{q=1}^{K-1} t_q$. Since  $\dot{\bmap}^{(h)}(\bt,\bpi,\bpsi)=(\dot{\bmap}^{(h)}_{1}(\bt,\bpi,\bpsi),\ldots,\dot{\bmap}^{(h)}_{K-1}(\bt,\bpi,\bpsi))$ is the gradient of $\bt \mapsto \bmap^{(h)} (\bt,\bpi,\bpsi)$   where    $\dot{\bmap}^{(h)}_k(\bt,\bpi,\bpsi)$ denotes the partial derivative of $ \bmap^{(h)}(\bt,\bpi,\bpsi)$ with respect to $t_k$, we have
$$
\dot{\bmap}^{(h)}_k(\bt,\bpi,\bpsi)=\frac{\frac{\partial}{\partial t_{k}} f^{(h)}_{\bt,\bpsi_{\bt}(\bpi,\bpsi)}}{f^{(h)}_{\bt,\bpsi_{\bt}(\bpi,\bpsi)} },   
$$
with
\begin{equation}\label{eq:df}
\frac{\partial}{\partial t_{k}} f^{(h)}_{\bt,\bpsi_{\bt}(\bpi,\bpsi)}=  \smooth \psi_{\bt, k}(\bpi,\bpsi)- \smooth \psi_{\bt, K}(\bpi,\bpsi) +\sum_{\ell=1}^K t_\ell    \frac{\partial}{\partial t_k} \smooth \psi_{\bt,\ell}(\bpi,\bpsi) ,
\end{equation}
and
   \begin{equation*}
   \frac{\partial}{\partial t_k} \smooth \psi_{\bt,\ell}(\bpi,\bpsi) =  \smooth \psi_{\bt,\ell}(\bpi,\bpsi)\phi^{(h)}_{\bt,\bpi,\bpsi,\ell,k}
    \end{equation*}
    where $\phi^{(h)}_{\bt,\bpi,\bpsi,\ell}$ is defined by 
    $$
    \phi^{(h)}_{\bt,\bpi,\bpsi,\ell,k}= \sum_{j=1}^J \left( \kernel_h \star \frac{\frac{\partial}{\partial t_k}   \psi_{\bt,\ell,j}(\bpi,\bpsi)}{ \psi_{\bt,\ell,j}(\bpi,\bpsi)}\right) .
    $$
    Hence, using the definition of the naive score function with smoothing, we have for $k=1,\ldots,K-1$
    \begin{equation}\label{eq:ldot} 
        \dot{\bmap}^{(h)}_k(\bt,\bpi,\bpsi) = \score_{\bt,\bpsi_{\bt}(\bpi,\bpsi),k}^{(h)}  - \sum_{\ell=1}^K t_{\ell} \score_{\bt,\bpsi_{\bt}(\bpi,\bpsi),\ell}^{(h)} \phi^{(h)}_{\bt,\bpi,\bpsi,\ell,k} .
    \end{equation}
    The mapping $\bt\mapsto \bmap^{(h)}(\bt,\bpi,\bpsi)$ admits second-order derivatives defined by
\begin{align*}
  \ddot{\bmap}_{k,\ell}^{(h)}(\bt,\bpi,\bpsi)(\bx) 
  = \frac{\frac{\partial^2}{\partial t_{\ell} \partial t_{k}} f^{(h)}_{\bt,\bpsi_{\bt}(\bpi,\bpsi)}}{f^{(h)}_{\bt,\bpsi_{\bt}(\bpi,\bpsi)} }  -   \dot{\bmap}^{(h)}_k(\bt,\bpi,\bpsi) \dot{\bmap}^{(h)}_\ell(\bt,\bpi,\bpsi).
\end{align*}
We have
\begin{multline*}
\frac{\partial^2}{\partial t_{k'} \partial t_{k}} f^{(h)}_{\bt,\bpsi_{\bt}(\bpi,\bpsi)}=  \frac{\partial}{\partial t_{k'}}\left[\smooth \psi_{\bt, k}(\bpi,\bpsi) - \smooth \psi_{\bt, K}(\bpi,\bpsi) \right] \\+ \frac{\partial}{\partial t_{k}}\left[\smooth \psi_{\bt, k'}(\bpi,\bpsi) - \smooth \psi_{\bt, K}(\bpi,\bpsi) \right] +\sum_{\ell=1}^K t_\ell    \frac{\partial^2}{\partial t_{k'}\partial t_k} \smooth \psi_{\bt,\ell}(\bpi,\bpsi) ,    
\end{multline*}
and
$$
  \frac{\partial^2}{\partial t_{k'}\partial t_k} \smooth \psi_{\bt,\ell}(\bpi,\bpsi)=\smooth \psi_{\bt,\ell}(\bpi,\bpsi) \left(\phi^{(h)}_{\bt,\bpi,\bpsi,\ell,k}\phi^{(h)}_{\bt,\bpi,\bpsi,\ell,k'} - \lambda^{(h)}_{\bt,\bpi,\bpsi,\ell,k,k'}\right),
$$
with
$$
\lambda^{(h)}_{\bt,\bpi,\bpsi,\ell,k,k'}=  \sum_{j=1}^J \left( \kernel_h \star \left[\frac{\frac{\partial^2}{\partial t_{k'}\partial t_k}   \psi_{\bt,\ell,j}(\bpi,\bpsi)}{ \psi_{\bt,\ell,j}(\bpi,\bpsi)} - \frac{\frac{\partial}{\partial t_{k'}} \psi_{\bt,\ell,j}(\bpi,\bpsi)\frac{\partial}{\partial t_k}   \psi_{\bt,\ell,j}(\bpi,\bpsi)}{ \psi^2_{\bt,\ell,j}(\bpi,\bpsi)} \right]\right).
$$
Hence, we have  for $k=1,\ldots,K-1$ and  for $k'=1,\ldots,K-1$
\begin{multline} \label{eq:lddot}
\ddot{\bmap}_{k,\ell}^{(h)}(\bt,\bpi,\bpsi)(\bx) = \frac{ \smooth \psi_{\bt,k}(\bpi,\bpsi)\phi^{(h)}_{\bt,\bpi,\bpsi,k,k'} + \smooth \psi_{\bt,k'}(\bpi,\bpsi)\phi^{(h)}_{\bt,\bpi,\bpsi,k',k}}{ f^{(h)}_{\bt,\bpsi_{\bt}(\bpi,\bpsi)}} \\
-\frac{  \smooth \psi_{\bt,K}(\bpi,\bpsi)\left(\phi^{(h)}_{\bt,\bpi,\bpsi,K,k'}+\phi^{(h)}_{\bt,\bpi,\bpsi,K,k}\right)}{ f^{(h)}_{\bt,\bpsi_{\bt}(\bpi,\bpsi)}} \\+ 
 \sum_{\ell=1}^K    t_\ell \score_{\bt,\bpsi_{\bt}(\bpi,\bpsi),\ell}^{(h)} \left(\phi^{(h)}_{\bt,\bpi,\bpsi,\ell,k}\phi^{(h)}_{\bt,\bpi,\bpsi,\ell,k'} - \lambda^{(h)}_{\bt,\bpi,\bpsi,\ell,k,k'}\right)-   \dot{\bmap}^{(h)}_k(\bt,\bpi,\bpsi) \dot{\bmap}^{(h)}_\ell(\bt,\bpi,\bpsi).
\end{multline}

By continuity of $(\bt,\bpi,\bpsi)\mapsto\psi_{\bt,\ell,j}(\bpi,\bpsi)$ in a neighborhood of $V$ and by the continuity of $h\mapsto \kernel_h$, the mapping $(\bt,\bpi,\bpsi,h)\mapsto \smooth \psi_{\bt,\ell}(\bpi,\bpsi)$ is a continuous function of $(\bt,\bpsi,\bpsi)$ in $\tilde V=\{(\bt,\bpi,\bpsi,h): (\bt,\bpi,\bpsi)\in V, h>0\}$. In addition, since first and second order partial derivatives of $\bt\mapsto\psi_{\bt,\ell,j}(\bpi,\bpsi)$ are continuous functions of $(\bt,\bpi,\bpsi)$ in $ V$ due to our assumptions, we have that  $(\bt,\bpi,\bpsi,h)\mapsto \frac{\partial}{\partial t_k} \smooth \psi_{\bt,\ell}(\bpi,\bpsi) $ and $(\bt,\bpi,\bpsi,h)\mapsto   \frac{\partial^2}{\partial t_{k'}\partial t_k} \smooth \psi_{\bt,\ell}(\bpi,\bpsi)$ are continuous functions in $\tilde V$. This implies that  $(\bt,\bpi,\bpsi,h)\mapsto \dot\bmap^{(h)}(\bt,\bpi,\bpsi)$ and $(\bt,\bpi,\bpsi,h)\mapsto \ddot\bmap^{(h)}(\bt,\bpi,\bpsi)$ are continuous functions of $(\bt,\bpi,\bpsi,h)$  in $\tilde V$, for any $\bx$. 
Therefore,  as $(\bt,\bpi,\bpsi)$ tends to $(\bpi^\star,\bpi^\star,\bpsi^\star)$ and $h$ tends to zero from the right, we have $\dot\bmap^{(h)}(\bt,\bpi,\bpsi)$ converges point-wise  to $\beffscores$ by Condition \ref{C3}.  Since $\sup_{\bt,\bpi,\bpsi} |\bpsi_{\bt}(\bpi,\bpsi)|$ and $\sup_{\bt,\bpi,\bpsi} \left| \frac{\frac{\partial}{\partial t_k}   \psi_{\bt,\ell,j}(\bpi,\bpsi)}{ \psi_{\bt,\ell,j}(\bpi,\bpsi)}\right|$ are bounded by square integrable functions uniformly on $h$, the function $\dot\bmap^{(h)}_k(\bt,\bpi,\bpsi)$ is dominated by a square integrable function uniformly on $h$. 
Therefore, the dominated convergence theorem implies that for every $(\tilde\bt^{(n)},\tilde\bpi^{(n)},\tilde\bpsi^{(n)})$ that converges in probability to $(\bpi^\star,\bpi^\star,\bpsi^\star)$  and $h$ that tends to zero from the right as $n$ tends to infinity, we have 
$$
\Eg\left[\left\|\dot\bmap^{(h)} (\tilde\bt^{(n)},\tilde\bpi^{(n)},\tilde\bpsi^{(n)})(\bX_1) - \beffscores(\bX_1) \right\|_2^2\right] = o_\mathbb{P}(1).
$$
Combining this result with  the Donsker property of the class of functions $$\mathcal{D}_{n,r}=\{n^{r-1/2}\dot\bmap^{(h)}(\bt,\bpi,\bpsi): (\bt,\bpi,\bpsi)\in V\},$$ with $1/4<r\leq 1/2$ implies that
$$
\left\|     \Gn \dot\bmap^{(h)}(\tilde\bt^{(n)},\tilde\bpi^{(n)},\tilde\bpsi^{(n)}) -        \Gn\beffscores\right\|_2 = o_\mathbb{P}(  n^{1/2-r} )
$$
and thus, we have
\begin{equation} \label{eq:controldot}
\left\|    \frac{1}{\sqrt{n}} \Gn \dot\bmap^{(h)}(\tilde\bt^{(n)},\tilde\bpi^{(n)},\tilde\bpsi^{(n)}) -       \frac{1}{\sqrt{n}}  \Gn\beffscores\right\|_2 = o_\mathbb{P}( n^{-r}).
\end{equation}
Since $\sup_{\bt,\bpi,\bpsi} \left| \frac{\frac{\partial^2}{\partial t_{k'}\partial t_k}   \psi_{\bt,\ell,j}(\bpi,\bpsi)}{ \psi_{\bt,\ell,j}(\bpi,\bpsi)}\right|$ is bounded by an integrable function uniformly on $h$, the function $\ddot\bmap^{(h)}_k(\bt,\bpi,\bpsi)$ is dominated by an integrable function leading that
$$
\Eg \left[ \ddot\bmap^{(h)}_{k\ell}(\bpi^\star,\bpi^\star,\bpsi^\star)(\bX_1)\right] = - \Eg \left[  \dot\bmap^{(h)}_{k}(\bpi^\star,\bpi^\star,\bpsi^\star)(\bX_1)\dot\bmap^{(h)}_{\ell}(\bpi^\star,\bpi^\star,\bpsi^\star)(\bX_1)\right].
$$
Hence, using the definition of $\befffishers$ given by \eqref{eq:fisher},   the  dominated convergence theorem  states that  for every $(\tilde\bt^{(n)},\tilde\bpi^{(n)},\tilde\bpsi^{(n)})$ that converges in probability to $(\bpi^\star,\bpi^\star,\bpsi^\star)$  and $h$ that tends to 0 as $n$ tends to infinity, we have 
$$
\left\|\Eg\left[ \ddot\bmap^{(h)}(\tilde\bt^{(n)},\tilde\bpi^{(n)},\tilde\bpsi^{(n)})(\bX_1)\right] + \befffishers\right\|=o_\mathbb{P}(1).
$$
Combining this result with the fact that the class of functions $\{\ddot\bmap^{(h)}(\bt,\bpi,\bpsi): (\bt,\bpi,\bpsi)\in V\}$ is $g^\star$-Glivenko-Cantelli and is bounded in $L_1(g^\star)$, implies that  for every $(\tilde\bt^{(n)},\tilde\bpi^{(n)},\tilde\bpsi^{(n)})$ that converges in probability to $(\bpi^\star,\bpi^\star,\bpsi^\star)$  and $h$ that tends to 0 as $n$ tends to infinity, we have  
\begin{equation} \label{eq:controlddot}
\left\|\Pn\ddot\bmap^{(h)}(\tilde\bt^{(n)},\tilde\bpi^{(n)},\tilde\bpsi^{(n)}) + \befffishers\right\| = o_\mathbb{P}(1).
\end{equation}
The profiling of the loss function  implies that
$$
 \profilenh(\bpi^\star) -  \profilenh(\tilde\bpi^{(n)})  =  \lossnh(\bpi^\star,\widehat{\bpsi}^{(h,n,\bpi^\star)}) -  \lossnh(\tilde\bpi^{(n)},\widehat{\bpsi}^{(h,n,\tilde\bpi^{(n)})}) .
$$
Hence, since by Condition \ref{C2}, we have $$ \widehat{\bpsi}^{(h,n,\bpi^\star)} =\bpsi_{\bpi^\star}(\bpi^\star, \widehat{\bpsi}^{(h,n,\bpi^\star)})$$ and  $$ \widehat{\bpsi}^{(h,n,\tilde\bpi^{(n)})} =\bpsi_{\tilde\bpi^{(n)}}(\tilde\bpi^{(n)}, \widehat{\bpsi}^{(h,n,\tilde\bpi^{(n)})}),$$ then using the fact that $\hpsipi$ is a global minimizer of $\lossnh(\bpi,\bpsi)$ with respect to $\bpsi$ (see \eqref{eq:hpsipi}), we have
\begin{equation}\label{eq:encadrement}
\Pn \ln \frac{f^{(h)}_{\tilde\bpi^{(n)}, \bpsi_{\tilde\bpi^{(n)}}(\bpi^\star, \widehat{\bpsi}^{(h,n,\bpi^\star)}) } }{f^{(h)}_{\bpi^\star, \bpsi_{\bpi^\star}(\bpi^\star, \widehat{\bpsi}^{(h,n,\bpi^\star)}) } }\leq  \profilenh(\bpi^\star) -  \profilenh(\tilde\bpi^{(n)})  \leq 
\Pn  \ln \frac{f^{(h)}_{\tilde\bpi^{(n)}, \bpsi_{\tilde\bpi^{(n)}}(\tilde\bpi^{(n)}, \widehat{\bpsi}^{(h,n,\tilde\bpi^{(n)})}) } }{f^{(h)}_{\bpi^\star, \bpsi_{\bpi^\star}(\tilde\bpi^{(n)}, \widehat{\bpsi}^{(h,n,\tilde\bpi^{(n)})}) } }.
\end{equation}
To control the lower and upper bound, we use a Taylor expansion of order two of $\bmap^{(h)}(\bt,\bpi,\bpsi)(\bx)$ with respect to its first argument. Hence, for any sequence $(\bar\bpi^{(n)},\bar\bpsi^{(n)})$ that converges in probability to $(\bpi^\star,\bpsi^\star)$, there exists $\tilde\bt^{(n)}=(\tilde t_1^{(n)},\ldots,\tilde t_K^{(n)})$ with $|t_k^{(n)} - \pi_k^\star|\leq |\tilde\pi_k-\pi_k^\star|$ where
\begin{multline} \label{eq:diffmaping}
\Pn \bmap^{(h)}(\tilde\bpi^{(n)},\bar\bpi^{(n)},\bar\bpsi^{(n)}) 
  - \Pn\bmap^{(h)}(\bpi^\star,\bar\bpi^{(n)},\bar\bpsi^{(n)})  = 
 (\tilde\bpi^{(n)}-\bpi^\star)^\top \Pn \dot\bmap^{(h)}(\bpi^\star,\bar\bpi^{(n)},\bar\bpsi^{(n)})  \\ +
 \frac{1}{2}  (\tilde\bpi^{(n)}-\bpi^\star)^\top \left[ \Pn \ddot\bmap^{(h)}(\tilde\bt^{(n)},\bar\bpi^{(n)},\bar\bpsi^{(n)})  \right]  (\tilde\bpi^{(n)}-\bpi^\star).
\end{multline}
To control the first term on the right-hand side of the previous equation, we note that
\begin{align*}
  \Pn \dot\bmap^{(h)}(\bpi^\star,\bar\bpi^{(n)},\bar\bpsi^{(n)})   =
      \frac{1}{\sqrt{n}} \Gn \dot\bmap^{(h)}(\bpi^\star,\bar\bpi^{(n)},\bar\bpsi^{(n)}) +  \Eg[\dot\bmap^{(h)}(\bpi^\star,\bar\bpi^{(n)},\bar\bpsi^{(n)})(\bX)].
\end{align*}
Using \eqref{eq:controldot}, we have

\begin{multline*}
    \frac{1}{\sqrt{n}} (\tilde\bpi^{(n)}-\bpi^\star)^\top \Gn \dot\bmap^{(h)}(\bpi^\star,\bar\bpi^{(n)},\bar\bpsi^{(n)})=   \frac{1}{\sqrt{n}} (\tilde\bpi^{(n)}-\bpi^\star)^\top \Gn\beffscores \\  +   (\tilde\bpi^{(n)}-\bpi^\star) o_\mathbb{P}(n^{-r}).
\end{multline*}
The small-bias condition of Condition \ref{C4} implies that
$$
(\tilde\bpi^{(n)}-\bpi^\star)^\top\Eg[\dot\bmap^{(h)}(\bpi^\star,\bar\bpi^{(n)},\bar\bpsi^{(n)})(\bX)] =  (\tilde\bpi^{(n)}-\bpi^\star)^\top o_\mathbb{P}(\|\tilde\bpi^{(n)} - \bpi^\star\| +  n^{-r}  ).
$$
Hence, the first term on the right-hand side of the \eqref{eq:diffmaping} can be controlled by
\begin{multline*}
 (\tilde\bpi^{(n)}-\bpi^\star)^\top \Pn \dot\bmap^{(h)}(\bpi^\star,\bar\bpi^{(n)},\bar\bpsi^{(n)}) = \frac{1}{\sqrt{n}} (\tilde\bpi^{(n)}-\bpi^\star)^\top \Gn\beffscores\\ +   o_\mathbb{P}(\|\tilde\bpi^{(n)} - \bpi^\star\|^2 +\|\tilde\bpi^{(n)}-\bpi^\star\|  n^{-r}  ).   
\end{multline*}
 To control the second term in the right-hand side of \eqref{eq:diffmaping}, we use the fact that since $\tilde\bpi^{(n)}$ converges in probability to $\bpi^\star$ we have that $\tilde\bt^{(n)}$ converges in probability to $\bpi^\star$. Hence, using \eqref{eq:controlddot}, we have
\begin{multline*}
(\tilde\bpi^{(n)}-\bpi^\star)^\top \left[ \Pn \ddot\bmap^{(h)}(\tilde\bt^{(n)},\bar\bpi^{(n)},\bar\bpsi^{(n)})  \right]  (\tilde\bpi^{(n)}-\bpi^\star)=-(\tilde\bpi^{(n)}-\bpi^\star)^\top \befffishers  (\tilde\bpi^{(n)}-\bpi^\star) \\+ o_\mathbb{P}(\|\tilde\bpi^{(n)}-\bpi^\star\|^2).
\end{multline*}
Noting that $\Eg [\beffscores (\bX_1)]=\boldsymbol{0}_K$, we have $n^{-1/2}\Gn\beffscores=\Pn\beffscores$.   Therefore,   for any sequence $(\bar\bpi^{(n)},\bar\bpsi^{(n)})$ that converges in probability to $(\bpi^\star,\bpsi^\star)$, we have
\begin{multline*}
\Pn \bmap^{(h)}(\tilde\bpi^{(n)},\bar\bpi^{(n)},\bar\bpsi^{(n)}) 
  - \Pn\bmap^{(h)}(\bpi^\star,\bar\bpi^{(n)},\bar\bpsi^{(n)}) = \\
(\tilde\bpi^{(n)}-\bpi^\star)^\top  \Pn \beffscores   - \frac{1}{2} (\tilde\bpi^{(n)}-\bpi^\star)^\top \befffishers  (\tilde\bpi^{(n)}-\bpi^\star) +  o_\mathbb{P}([\|\tilde\bpi^{(n)} - \bpi^\star\| +  n^{-r}  ]^2) .
\end{multline*}
The bounds of \eqref{eq:encadrement} can be defined as $ \Pn \bmap^{(h)}(\tilde\bpi^{(n)},\bar\bpi^{(n)},\bar\bpsi^{(n)}) 
  -  \Pn\bmap^{(h)}(\bpi^\star,\bar\bpi^{(n)},\bar\bpsi^{(n)})  $, with $(\bar\bpi^{(n)},\bar\bpsi^{(n)})=(\bpi^\star,\widehat{\bpsi}^{(h,n,\bpi^\star)})$ for the lower bound and $(\bar\bpi^{(n)},\bar\bpsi^{(n)})=(\tilde\bpi^{(n)},\widehat{\bpsi}^{(h,n,\tilde\bpi^{(n)})})$ for the upper bound. Therefore, we have
  \begin{multline*}
  \profilenh(\bpi^\star) -  \profilenh(\tilde\bpi^{(n)}) =   (\tilde\bpi^{(n)}-\bpi^\star)^\top \Pn  \beffscores  - \frac{1}{2}  (\tilde\bpi^{(n)}-\bpi^\star)^\top \befffishers  (\tilde\bpi^{(n)}-\bpi^\star) \\+  o_\mathbb{P}([\|\tilde\bpi^{(n)} - \bpi^\star\| +  n^{-r}  ]^2) .      
  \end{multline*}
\end{proof}

\begin{lemma}\label{lem:control_sup_normdev}
There exists a constant $C>0$ such that if $h$ is small enough we have
$$
\max_{k,j} \left\| \frac{\check{\psi}^{(h)}_{k,j} - \psi_{k,j}^\star}{\psi_{k,j}^\star} \right\|_{\infty} \leq C.
$$
\end{lemma}

\begin{proof}[Proof of Lemma~\ref{lem:control_sup_normdev}]
 Using the definition of $\beffscore_{\bpi^\star,\bpsi^\star}^{(h)}$ given by \eqref{eq:effscore}, and since the projection is $g^\star$-orthogonal, we have
$$
\| \score_{\bpi^\star,\bpsi^\star}^{(h)}    \|_{L^2(g^\star)}^2=\|\beffscore_{\bpi^\star,\bpsi^\star}^{(h)}  \|_{L^2(g^\star)}^2 + \| \boldsymbol{1}_{K-1}\noisy_{\bpi^\star,\bpsi^\star}^{(h)}[\check\bpsi^{(h)} - \bpsi^\star]    \|_{L^2(g^\star)}^2 ,
$$
where $\check\bpsi^{(h)}$ is defined by
$$
\| \score_{\bpi^\star,\bpsi^\star}^{(h)} - \boldsymbol{1}_{K-1}\noisy_{\bpi^\star,\bpsi^\star}^{(h)}[\check\bpsi^{(h)} - \bpsi^\star]    \|_{L^2(g^\star)}^2 = \min_{\bpsi\in\Psi_K(\mathcal{X})} \|  \score_{\bpi^\star,\bpsi^\star}^{(h)} - \boldsymbol{1}_{K-1}\noisy_{\bpi^\star,\bpsi^\star}^{(h)}[\bpsi - \bpsi^\star] \|_{L^2(g^\star)}^2,
$$
 $\boldsymbol{1}_{K-1}$ begin the vector composed of $K-1$ ones.
Since $0\leq \score_{\bpi^\star,\bpsi^\star,k}^{(h)}(\bx)\leq 1/\pi_k^\star$ there exits a positive constant $C$ such that
$$
\sup_{h>0} \| \score_{\bpi^\star,\bpsi^\star}^{(h)}    \|_{L^2(g^\star)}^2 \leq C.
$$
Hence, the first equation implies that
\begin{equation} \label{eq:boundA}
\| \noisy_{\bpi^\star,\bpsi^\star}^{(h)}[\check\bpsi^{(h)} - \bpsi^\star]    \|_{L^2(g^\star)}^2 \leq  C.    
\end{equation}
For any $\bpsi\in\Psi_K(\mathcal{X})$, defined $\bnu_{\bpsi,j}^{(h)}=(\nu_{\bpsi,j,1}^{(h)},\ldots,\nu_{\bpsi,j,K}^{(h)})$ the $K$ dimensional vector with
$$
\nu_{\bpsi,j,k}^{(h)}(u)= \left[\kernel_h \star \frac{\psi_{k,j} - \psi_{k,j}^\star}{\psi_{k,j}^\star}\right](u).
$$
For any $\bpsi\in\Psi(\mathcal{X})$, we have
$$
  \|\noisy_{\bpi^\star,\bpsi^\star}^{(h)}[\bpsi - \bpsi^\star]    \|_{L^2(g^\star)}^2=\sum_{k,k',j,j'}\int_\mathcal{X}\omega_{\bpi^\star,\bpsi^\star,k}^{(h)}(\bx)\omega_{\bpi^\star,\bpsi^\star,k'}^{(h)}(\bx)\nu_{\bpsi,j,k}^{(h)}(x_j)\nu_{\bpsi,j',k'}^{(h)}(x_{j'}) g^\star(\bx)d\bx.
$$
Let $ g^\star_{j,j'}(x_j,x_{j'})$ denote the marginal density of $(X_j,X_{j'})$ defined as the integral of $g^\star$ over all the components of $\bX$ but components $j$ and $j'$ and $\Lambda_{j,j',k,k'}^{(h)}$ be defined by
with
$$
\Lambda_{j,j',k,k'}^{(h)}(x_j,x_{j'})=  \Eg\left[\omega_{\bpi^\star,\bpsi^\star,k}^{(h)}(\bX)\omega_{\bpi^\star,\bpsi^\star,k'}^{(h)}(\bX)  \mid X_j=x_j, X_{j'}=x_{j'}\right]
$$
Hence, we have
 for any $\bpsi\in\Psi_K(\mathcal{X})$
$$
\|\noisy_{\bpi^\star,\bpsi^\star}^{(h)}[\bpsi - \bpsi^\star]    \|_{L^2(g^\star)}^2= \sum_{j=1}^J \sum_{j'=1}^J \sum_{k=1}^K \sum_{k'=1}^K \bG_{\bpsi,j,j,k,k'}
$$
where
$$
\bG_{\bpsi,j,j,k,k'}=\int_{\mathcal{X}_j \times\mathcal{X}_{j'}} \nu_{\bpsi,j,k}^{(h)}(x_j) \bLambda_{j,j',k,k'}^{(h)}(x_j,x_{j'})\nu_{\bpsi,j',k}^{(h)}(x_{j'}) g^\star_{j,j'}(x_j,x_{j'}) dx_j dx_{j'}.
$$
Note that, for any $(j,j')$ the elements of $\bLambda_{j,j'}^{(h)}$ are positive and bounded from above by $1$ since the weights $\omega_{\bpi^\star,\bpsi^\star,k}^{(h)}(\bX)$ are between 0 and 1. In addition,  we have
$$
\inf_{\bpsi\in\Psi(\mathcal{X}_j)}\min_{k,j} \nu_{\bpsi,k,j}^{(h)}(u) \geq -1,
$$
leading that there exists a positive constant $C$ such that for any 
$$
\inf_{\bpsi\in\Psi(\mathcal{X}_j)}\min_{k,k',j,j'} \bG_{\bpsi,j,j,k,k'} \geq -C.
$$
Now suppose that, there exists at least one couple $(k,j)$ such that $(\check{\psi}^{(h)}_{k,j} - \psi_{k,j}^\star)/\psi_{k,j}^\star$ is not bounded when $h$ is small enough. This means that $(\check{\psi}^{(h)}_{k,j} - \psi_{k,j}^\star) / \psi_{k,j}^\star$ is not upper-bounded since this function is always lower-bounded by -1. Since $\check\psi_{k,j}^{(h)} \in \Psi(\mathcal{X}_j)$ this also implies that $\nu_{\bpsi,j,k}^{(h)}$ is not bounded when $h$ is small enough leading that $\limsup_{h\to 0} \bG_{\bpsi,j,j,k,k'} = \infty$ and hence 
$\limsup_{h\to0} \|\noisy_{\bpi^\star,\bpsi^\star}^{(h)}[\check\bpsi^{(h)} - \bpsi^\star]    \|_{L^2(g^\star)}^2=\infty$ which is in contradiction with \eqref{eq:boundA}.
\end{proof}

\begin{proof}[Proof of Theorem~\ref{thm:norma}]
   Let $\bt$ be the vector defined on the restricted simplex $\mathcal{S}_K^r$. This means that its last element $t_K=1-\sum_{q=1}^{K-1}t_q$. For any $(\bpi,\bpsi)$, consider the map $\bt\mapsto \bpsi_{\bt}(\bpi,\bpsi)$ with  $\bpsi_{\bt}(\bpi,\bpsi)=(\psi_{\bt,1}(\bpi,\bpsi),\ldots,\bpsi_{\bt,K}(\bpi,\bpsi))$ and  $\bpsi_{\bt,k}(\bpi,\bpsi)=(\psi_{\bt,k,1}(\bpi,\bpsi),\ldots,\psi_{\bt,k,K}(\bpi,\bpsi))$ where each $\psi_{\bt,k,j}(\bpi,\bpsi)$ is a univariate density defined as
    \begin{equation}\label{eq:mapingpsi}
    \psi_{\bt,k,j}(\bpi,\bpsi) =\frac{1}{Z_{\bt,\bpi,\bpsi,k,j}}\psi_{k,j} \exp\left( (t_K - \pi_K) \frac{\check\psi_{k,j}^{(h)} -\psi_{k,j}}{\psi_{k,j}}\right),
    \end{equation}
    where $\psi_{k,j} \in \Psi(\mathcal{X}_j)$, $\check{\psi}_{k,j}^{(h)}\in \Psi(\mathcal{X}_j)$ and is defined by \eqref{eq:condproj}, and $Z_{\bt,\bpi,\bpsi,k,j}$ is the normalization constant ensuring that $\psi_{\bt,k,j}(\bpi,\psi)$ integrates to one where
    $$
    Z_{\bt,\bpi,\bpsi,k,j} = \int_{\mathcal{X}_j}\psi_{k,j}(u) \exp\left( (t_K - \pi_K) \frac{\check\psi_{k,j}^{(h)}(u) -\psi_{k,j}(u)}{\psi_{k,j}(u)}\right)du.
    $$
    Noting that if $\psi_{k,j}$ is in a neighborhood of $\psi_{k,j}^\star$ in the sense of Lemma ~\ref{lem:control_sup_normdev}, 
    $
    \left\|\frac{\check\psi_{k,j}^{(h)} -\psi_{k,j}}{\psi_{k,j}}\right\|_\infty     $ is finite  and thus, noting that by construction $   \psi_{\bt,k,j}(\bpi,\bpsi)\geq 0$, $   \psi_{\bt,k,j}(\bpi,\bpsi)$ is a density function. 
    Hence, it can be checked that if $(\bt,\bpi,\bpsi)$ is in a neighborhood of $(\bpi^\star,\bpi^\star,\bpsi^\star)$ denoted by $V$ then $\psi_{\bt,k,j}\in \Psi(\mathcal{X}_j)$.
    For any $k=1,\ldots, K-1$, we have
$$
\frac{\partial}{\partial t_\ell}   \psi_{\bt,k,j}(\bpi,\bpsi) = \psi_{\bt,k,j}(\bpi,\bpsi) \left[ - \frac{\check\psi_{k,j}^{(h)} -\psi_{k,j}}{\psi_{k,j}}-     \frac{    Z_{\bt,\bpi,\bpsi,k,j}'}{    Z_{\bt,\bpi,\bpsi,k,j}}\right],
$$
with 
$$
   Z_{\bt,\bpi,\bpsi,k,j}'=-\int_{\mathcal{X}_j}(\check\psi_{k,j}^{(h)}(u) -\psi_{k,j}(u))\exp\left( (t_K - \pi_K) \frac{\check\psi_{k,j}^{(h)}(u) -\psi_{k,j}(u)}{\psi_{k,j}(u)}\right)du.
$$
Since the derivative of $\frac{\partial}{\partial t_\ell}   \psi_{\bt,k,j}(\bpi,\bpsi)$ is now known, we can define
$$
    \phi^{(h)}_{\bt,\bpi,\bpsi,\ell,k}= - \sum_{j=1}^J \left( \kernel_h \star \frac{\check\psi_{\ell,j}^{(h)}-\psi_{\ell,j}}{ \psi_{\ell,j}}\right) - \sum_{j=1}^J \frac{    Z_{\bt,\bpi,\bpsi,k,j}'}{    Z_{\bt,\bpi,\bpsi,k,j}}.
    $$
This, in its own turn, let us write down an explicit expression for the score function $\dot{\bmap}^{(h)}_k(\bt,\bpi,\bpsi)$ using \eqref{eq:ldot}. In addition, the second order partial derivatives of $\bt\mapsto \psi_{\bt,k,j}(\bpi,\bpsi)$ can also be written down explicitly as 
\begin{multline*}
\frac{\partial^2}{\partial t_{\ell'} \partial t_\ell}   \psi_{\bt,k,j}(\bpi,\bpsi) = \psi_{\bt,k,j}(\bpi,\bpsi) \left(\left[ - \frac{\check\psi_{k,j}^{(h)} -\psi_{k,j}}{\psi_{k,j}}-     \frac{    Z_{\bt,\bpi,\bpsi,k,j}'}{    Z_{\bt,\bpi,\bpsi,k,j}}\right]^2 \right.\\\left. -  \left[ \frac{    Z_{\bt,\bpi,\bpsi,k,j}''}{    Z_{\bt,\bpi,\bpsi,k,j}}- \left(\frac{    Z_{\bt,\bpi,\bpsi,k,j}'}{    Z_{\bt,\bpi,\bpsi,k,j}}\right)^2 \right]\right),
\end{multline*}
    with
    \[
   Z_{\bt,\bpi,\bpsi,k,j}''=\int_{\mathcal{X}_j}(\check\psi_{k,j}^{(h)}(u) -\psi_{k,j}(u))^2\exp\left( (t_K - \pi_K) \frac{\check\psi_{k,j}^{(h)}(u) -\psi_{k,j}(u)}{\psi_{k,j}(u)}\right)du.
\]
This also let us write a closed-form expression for the Hessian of the log-likelihood $\ddot{\bmap}^{(h)}_k(\bt,\bpi,\bpsi)$ using \eqref{eq:lddot}.
We now show that the four conditions of Proposition~\ref{thm:convparam} are satisfied.

\color{black}
\begin{enumerate}
\item  Using \eqref{eq:ldot}, we have
        $$\|\dot{\bmap}^{(h)}_k(\bt,\bpi,\bpsi)\|_{\infty} \leq \|\score_{\bt,\bpsi_{\bt}(\bpi,\bpsi),k}^{(h)}\|_{\infty}  + \|\score_{\bt,\bpsi_{\bt}(\bpi,\bpsi),\ell}^{(h)}\|_\infty \sum_{\ell=1}^K  \left\| \phi^{(h)}_{\bt,\bpi,\bpsi,\ell,k}\right\|_{\infty}.$$
        Note that $\| \score_{\bt,\bpsi_{\bt}(\bpi,\bpsi),k}^{(h)}\|_\infty \leq 1/t_k$. 
        In addition, for any $\bpsi$ and any $\bpi$, we have $Z_{\bpi,\bpi,\bpsi,k,j}=1$ and $Z_{\bpi,\bpi,\bpsi,k,j}'=0$ for any $(k,j)$. Therefore, using  Lemma~\ref{lem:control_sup_normdev}, we have that  $\phi^{(h)}_{\bpi^\star,\bpi^\star,\bpsi^\star,\ell,k}$ is bounded leading, by continuity, that $\phi^{(h)}_{\bpi^\star,\bpi^\star,\bpsi,\ell,k}$ it is bounded in $V$. In addition, for any $(\bt,\bpi,\bpsi)\in V$, we have $t_k>\min_k \pi_k^\star/2$ and thus $t_k$ is bounded away from zero since by assumption any $\pi_k^\star>0$, leading that
$$\sup_{(\bt,\bpi,\bpsi)\in V} \| \score_{\bt,\bpsi_{\bt}(\bpi,\bpsi),k}^{(h)}\|_\infty=O(1).$$
In addition, $\|(\check\psi_{\ell, j}^{(h)}-\psi_{\ell , j} )/ \psi_{\bt,\ell,j}(\bpi,\bpsi)\|_\infty$ is bounded since, as elements of $\Psi(\mathcal{X}_j)$, $\check\psi_{\ell, j}^{(h)}$ and $\psi_{\ell , j}$ are upperbounded and bounded away from zero. Therefore, $\dot{\bmap}^{(h)}_k(\bt,\bpi,\bpsi)$ is upperbounded by a constant and hence, there exists a square integrable function that upper-bounds $\dot{\bmap}^{(h)}_k(\bt,\bpi,\bpsi)$  for any $(\bt,\bpi,\bpsi)\in V$. With the same reasoning, we can show that $\ddot{\bmap}^{(h)}_k(\bt,\bpi,\bpsi)$ is upperbounded by a constant and hence, there exists an integrable function that upper-bounds $\ddot{\bmap}^{(h)}_k(\bt,\bpi,\bpsi)$  for any $(\bt,\bpi,\bpsi)\in V$. 
     \item The previous result implies that $\dot{\bmap}^{(h)}_k(\bt,\bpi,\bpsi) $ belongs to a Sobolev space $\mathcal{W}^{1,2,r}(\mathcal{X})$ defined by \eqref{eq:sobolev} where the  radius $r$ has an order $h^{-1}$. Hence, considering the space
     $$\mathcal{E}_{1,h}=\{ \dot\bmap^{(h)}(\bt,\bpi,\bpsi): (\bt,\bpi,\bpsi)\in V\},$$
     we have that $\mathcal{E}_{1,h}$ is included to a Sobolev space $\mathcal{W}^{1,2,r}(\mathcal{X})$ where the radius $r$ has an order $h^{-1}$ (see proof of Lemma~\ref{lem:deviation}), leading that
     $$
     \Eg \left[\sup_{e^{(h)} \in \mathcal{E}_{1,h}}\left|\Gn e^{(h)} \right|  \right]=o(h^{-1/2}).
     $$
     Let $\mathcal{D}_{1,n,r}$ be the class of functions defined by
     $$\mathcal{D}_{1,n,r}=\{n^{r-1/2}\dot\bmap^{(h)}(\bt,\bpi,\bpsi): (\bt,\bpi,\bpsi)\in V\},$$
     then we have
          $$
     \Eg \left[\sup_{d^{(n,h)} \in \mathcal{D}_{1,n,r}}\left|\Gn d^{(n,h)}  \right|  \right]=o(h^{-1/2}n^{r-1/2}).
     $$
     Since by Assumptions~\ref{ass:band1}, we have $h^{-1/2}=o(n^{1/2 - r})$, we have
               $$
     \Eg \left[\sup_{d^{(n,h)} \in \mathcal{D}_{1,n,r}}\left|\Gn d^{(n,h)}  \right|  \right]=o(1),
     $$
     leading that  $\mathcal{D}_{n,r}$ is $g^\star$-Donsker. With the same reasoning, we can show that $\ddot{\bmap}^{(h)}_k(\bt,\bpi,\bpsi) $ belongs to a Sobolev space  $\mathcal{W}^{1,2,r}(\mathcal{X})$ where the radius $r$ has an order $h^{-1}$. Hence, considering the space
     $$\mathcal{E}_{2,h}=\{ \ddot\bmap^{(h)}(\bt,\bpi,\bpsi): (\bt,\bpi,\bpsi)\in V\},$$
     we have that $\mathcal{E}_{2,h}$ is a subset of a Sobolev space $\mathcal{W}^{1,2,r}(\mathcal{X})$ where the radius $r$ has an order $h^{-1}$, leading that
     $$
     \Eg \left[\sup_{e^{(h)} \in \mathcal{E}_{2,h}}\left|\Gn e^{(h)} \right|  \right]=o(h^{-1/2}).
     $$
     Let $\mathcal{D}_{2,n,r}$ be the class of functions defined by
     $$\mathcal{D}_{2,n,r}=\{n^{r-1/2}\ddot\bmap^{(h)}(\bt,\bpi,\bpsi): (\bt,\bpi,\bpsi)\in V\},$$
     then we have
          $$
     \Eg \left[\sup_{d^{(n,h)} \in \mathcal{D}_{2,n,r}}\left|\Gn d^{(n,h)}  \right|  \right]=o(h^{-1/2}n^{r-1/2}).
     $$
     Since by Assumptions~\ref{ass:band1}, we have $h^{-1/2}=o(n^{1/2 - r})$, we have
               $$
     \Eg \left[\sup_{d^{(n,h)} \in \mathcal{D}_{2,n,r}}\left|\Pn d^{(n,h)} - \Eg d^{(n,h)}  \right|  \right]=o(1),
     $$
     implying that $\{\ddot\bmap(\bt,\bpi,\bpsi): (t,\bpi,\bpsi)\in V\}$ is $g^\star$-Glivenko-Cantelli and thus that Condition~\ref{C1} of Proposition~\ref{thm:convparam} holds true.
  \item   Note that for any $(\bpi,\bpsi)$, we have using \eqref{eq:mapingpsi} that $\psi_{\bpi,k,j}(\bpi,\bpsi) = \psi_{k,j}$ and hence $\bpsi_{\bpi}(\bpi,\bpsi)=\bpsi$ leading that Condition \ref{C2} of Proposition~\ref{thm:convparam} holds true. 
  \item In addition, since $\bpsi_{\bpi^\star}(\bpi^\star,\bpsi^\star)=\bpsi^\star$ and $\phi^{(h)}_{\bpi^\star,\bpi^\star,\bpsi^\star,\ell}=\zeta_{\bpsi^\star,\check\bpsi^{(h)},h,k}$ and $\zeta_{\bpsi^\star,\check\bpsi^{(h)},h,k}$ having been defined in \eqref{eq:zeta}. Then,
    $$\sum_{k=1}^K \pi^\star_k \score_{\bpi^\star,\bpsi_{\bt}(\bpi^\star,\bpsi^\star),k}^{(h)} \phi^{(h)}_{\bpi^\star,\bpi^\star,\bpsi^\star,k}= \noisy_{\bpi^\star,\bpsi^\star}^{(h)}[\check\bpsi^{(h)}- \bpsi^\star].$$ 
    Hence,  we have 
\begin{equation}\label{eq:equality}
\dot{\bmap}^{(h)}(\bpi^\star,\bpi^\star,\bpsi^\star)  = \beffscore_{\bpi^\star,\bpsi^\star}^{(h)}.
\end{equation}
In addition, we have $\beffscore_{\bpi^\star,\bpsi^\star,k}^{(h)}$ is a continuous function of $h$ and since $\beffscore_{\bpi^\star,\bpsi^\star,k}=\lim_{h\to 0} \beffscore_{\bpi^\star,\bpsi^\star,k}^{(h)}$, leading that Condition \ref{C3}  of Proposition~\ref{thm:convparam} holds true. 
\item For any random sequence $\tilde\bpi^{(n)}$ that converges in probability to $\bpi^\star$, we have that $\tilde\bpi^{(n)}$ belongs to $\mathcal{B}(\bpi^\star)$ with high-probability. Hence, since Assumptions~\ref{ass:controlvariancetrue} and \ref{ass:controlvariancekernel} are supposed to hold true, we have  by Theorem~\ref{thm:rateL1}, 
$$
 \sum_{k=1}^K \sum_{j=1}^J \| \psi^\star_{k,j} - \widehat{\psi}^{(h,n,\tilde\bpi^{(n)})}_{k,j}\|^2_1 = O_\mathbb{P}( n^{-1/2}h^{-1/2}+h^2+\|\tilde\bpi^{(n)} - \bpi^\star\|_1),
$$ and thus  and that Assumptions~\ref{ass:band1}, then 
$$
 \sum_{k=1}^K \sum_{j=1}^J \| \psi^\star_{k,j} - \widehat{\psi}^{(h,n,\tilde\bpi^{(n)})}_{k,j}\|^2_1 = O_\mathbb{P}(  \|\tilde\bpi^{(n)} - \bpi^\star\|_1) + o_\mathbb{P}(n^{-r}).
 $$
 
Since $\tilde\bpi^{(n)}$ converges in probability to $\bpi^\star$, this implies that $ \sum_{k=1}^K \sum_{j=1}^J \| \psi^\star_{k,j} - \widehat{\psi}^{(h,n,\tilde\bpi^{(n)})}_{k,j}\|^2_1 = o_\mathbb{P}(1)$, leading that
 $$
         \widehat{\bpsi}^{(h,n,\tilde\bpi^{(n)})} \xrightarrow{p} \bpsi^\star.
 $$
 Because $\dot{\bmap}^{(h)}(\bpi,\bpi,\bpsi)$ is the score function at model $f^{(h)}_{\bpi,\bpsi}$, we have
 $$
 \forall (\bpi,\bpsi)\in\Theta,\, \int f^{(h)}_{\bpi,\bpsi}(\bx) \dot{\bmap}^{(h)}(\bpi,\bpi,\bpsi)(\bx)d\bx = \boldsymbol{0}_K,
 $$
leading that
\begin{equation}\label{eq:centeredscore}
     \forall (\bpi,\bpsi)\in\Theta,\, \mathbb{E}_{g_{\bpi,\bpsi}} \left[\frac{f^{(h)}_{\bpi,\bpsi}(\bX_1)}{g_{\bpi,\bpsi}(\bX_1)} \dot{\bmap}^{(h)}(\bpi,\bpi,\bpsi)(\bX_1)\right] = \boldsymbol{0}_K.
\end{equation}
In addition, recall that by definition
$$
\forall \bpsi \in \Psi_K(\mathcal{X}),\, \noisy_{\bpi^\star,\bpsi^\star}^{(h)}[\bpsi-\bpsi^\star] =\frac{\partial}{\partial t} \ln f^{(h)}_{\bpi^\star,\bpsi_t} \Big|_{t=0}.
$$
In addition, as stated by \eqref{eq:equality}, we have $\dot{\bmap}^{(h)}(\bpi^\star,\bpi^\star,\bpsi^\star)  = \beffscore_{\bpi^\star,\bpsi^\star}^{(h)}$,  then using \eqref{eq:condproj} at $(\bpi,\bpsi)=(\bpi^\star,\bpsi^\star)$ leads to
\begin{equation}\label{eq:proj}
\forall  \bpsi, \forall k\in\{1,\ldots,K\},\, \Eg\left[\effscore_{\bpi^\star,\bpsi^\star,k}^{(h)}\noisy_{\bpi^\star,\bpsi^\star}^{(h)}[\bpsi-\bpsi^\star](\bX_1) \right]=0.
\end{equation}
Using \eqref{eq:centeredscore} and \eqref{eq:proj} provides
$$
    \Eg \left[\dot{\bmap}^{(h)}_k(\bpi^\star,\bpi^\star,\bpsi)(\bX_1)\right] = \Delta_{1,\bpi^\star,\bpsi^\star,\bpsi,k} + \Delta_{2,\bpi^\star,\bpsi^\star,\bpsi,k}   ,
$$
with 
$$
\Delta_{1,\bpi^\star,\bpsi^\star,\bpsi,k}=
    \Eg \left[\effscore_{\bpi^\star,\bpsi^\star,k}^{(h)}(\bX_1)\left(\noisy_{\bpi^\star,\bpsi^\star}^{(h)}[\bpsi-\bpsi^\star](\bX_1)-\frac{f^{(h)}_{\bpi^\star,\bpsi}(\bX_1)-f^{(h)}_{\bpi^\star,\bpsi^\star}(\bX_1)}{f^{(h)}_{\bpi^\star,\bpsi^\star}(\bX_1)} \right)\right]
$$
and
$$
\Delta_{2,\bpi^\star,\bpsi^\star,\bpsi,k}=\Eg \left[\frac{f^{(h)}_{\bpi^\star,\bpsi^\star}(\bX_1)-f^{(h)}_{\bpi^\star,\bpsi}(\bX_1)}{f^{(h)}_{\bpi^\star,\bpsi^\star}(\bX_1)} \left(\dot{\bmap}^{(h)}_k(\bpi^\star,\bpi^\star,\bpsi)(\bX_1)-\effscore_{\bpi^\star,\bpsi^\star,k}^{(h)}(\bX_1)\right)\right] .
$$
From Lemma~\ref{lem:Delta}, we have 
$$
\Delta_{1,\bpi^\star,\bpsi^\star,\bpsi,k} =  \sum_{k=1}^K \sum_{j=1}^J O\left(\left\| \psi_{k,j} - \psi_{k,j}^\star\right\|_{L_1}^2\right) + O(h^2).
$$
and
$$
\Delta_{2,\bpi^\star,\bpsi^\star,\bpsi,k} =  \sum_{k=1}^K \sum_{j=1}^J O\left(\left\| \psi_{k,j} - \psi_{k,j}^\star\right\|_{L_1}^2\right) + O(h^2).
$$
Since Assumption~\ref{ass:band1} ensures that $n^{-1/2}h^{-1/2}=o(n^{-r})$ and $h^2=o(n^{-r})$, using Theorem~\ref{thm:rateL1}, we have
$$
\Delta_{1,\bpi^\star,\bpsi^\star,\widehat{\bpsi}^{(h,n,\tilde\bpi^{(n)})}}^{(h)} = O_{\mathbb{P}}( \|\tilde\bpi^{(n)} - \bpi^\star\|) + o_\mathbb{P}(n^{-r})
$$
and
$$
\Delta_{2,\bpi^\star,\bpsi^\star,\widehat{\bpsi}^{(h,n,\tilde\bpi^{(n)})}}^{(h)} = O_{\mathbb{P}}( \|\tilde\bpi^{(n)} - \bpi^\star\|) + o_\mathbb{P}(n^{-r}).
$$
Hence, condition \ref{C4} is satisfied. 
\end{enumerate}
\color{black}

Since all the conditions of Proposition~\ref{thm:convparam} are satisfied, then  for any random sequence $\tilde\bpi^{(n)} \xrightarrow{p}  \bpi^\star$,
        \begin{multline}\label{eq:condVV}
\profilenh(\bpi^\star)=\profilenh(\tilde\bpi^{(n)}) + (\tilde\bpi^{(n)} -  \bpi^\star)^\top \Pn \beffscores \\ - \frac{1}{2}(\tilde\bpi^{(n)} -  \bpi^\star)^\top \befffishers(\tilde\bpi^{(n)} -  \bpi^\star) + o_\mathbb{P}([ \|\tilde\bpi^{(n)} -  \bpi^\star\| + n^{-r}  ]^2).
\end{multline}
    Let $\bUp= n^{-r}  \Pn\beffscores $ and $\bup= n^{-r}  (\hpi -  \bpi^{\star})$. Applying \eqref{eq:condVV} with $\tilde \bpi^{(n)}=\hpi$ implies that
    $$
   n^{2r}  \profilenh(\bpi^\star)= n^{2r}  \profilenh(\hpi) + \buptop\bUp - \frac{1}{2} \buptop \befffishers \bup + o_\mathbb{P}([\|\bup\|_2+1]^2).
    $$
Applying \eqref{eq:condVV} with $\tilde \bpi^{(n)}=\bpi^\star +  n^{-r}  \befffishersinv \bUp$ implies that
    $$
  n^{2r}   \profilenh(\bpi^\star)= n^{2r}  \profilenh\left(\bpi^\star +  n^{-r}  \befffishersinv \bUp\right) + \frac{1}{2} \bUptop \befffishersinv \bUp  + o_\mathbb{P}(1).
    $$
Taking the difference of the previous two equations, we have
\begin{multline*}
\buptop\bUp - \frac{1}{2} \buptop \befffishers \bup-\frac{1}{2} \bUptop \befffishersinv \bUp  + o_\mathbb{P}([\|\bup\|_2+1]^2) \\= n^{2r}  \profilenh\left(\bpi^\star +  n^{-r}  \befffishersinv \bUp\right) - n^{2r}  \profilenh(\hpi)     .
\end{multline*}
Note that we have 
\begin{multline*}
    -\frac{1}{2}\left(\bup - \befffishersinv \bUp  \right)^\top\befffishers \left(\bup - \befffishersinv \bUp  \right)\\=\buptop\bUp - \frac{1}{2} \buptop \befffishers \bup-\frac{1}{2} \bUptop \befffishersinv \bUp . 
\end{multline*}
Combining this results with the fact that by definition $\hpi$ is a global minimizer of $\profilenh$ leads
$$
-\frac{1}{2}\left(\bup - \befffishersinv \bUp  \right)^\top\befffishers \left(\bup - \befffishersinv \bUp  \right)+ o_\mathbb{P}([\|\bup\|_2+1]^2)\geq 0.
$$
Since $\befffishers$ is invertible by Lemma \eqref{lem:invert}, there exists a strictly positive constant $c$ such that 
$$
\frac{1}{2}\left(\bup - \befffishersinv \bUp  \right)^\top\befffishers \left(\bup - \befffishersinv \bUp  \right) \geq c \|\bup - \befffishersinv \bUp  \|_2^2.$$
Hence,
$$
o_\mathbb{P}([\|\bup\|_2+1]^2)\geq c \|\bup - \befffishersinv \bUp  \|_2^2
$$
leading that
\begin{equation}\label{eq:ratenorm}
\|\bup - \befffishersinv \bUp  \|_2=o_\mathbb{P}(\|\bup\|_2+1).
\end{equation}
Central limit theorem combined with invertibility of $\bSigma$ implies that $$\|\befffishersinv \bUp\|_2 =O_\mathbb{P}(n^{r-1/2}).$$
To establish the stochastic order of $\|\bup\|_2$,  suppose that $\|\bup\|_2$ diverges in probability. This leads that $\|\befffishersinv \bUp\|$ is stochastically negligible with respect to $\|\bup\|_2$, \emph{i.e.,} $\|\befffishersinv \bUp\|=o_\mathbb{P}(\|\bup\|_2)$. Then, using reverse triangular inequality, we have
$\|\bup\|_2 |1-o_\mathbb{P}(1)|\leq \|\bup\|_2 o_\mathbb{P}(1+1/\|\bup\|_2)$. This implies that $|1-o_\mathbb{P}(1)|\leq o_\mathbb{P}(1)$ which is impossible. Therefore, 
$$
\|\bup\|_2=O_\mathbb{P}(1),
$$
leading that the 
$$\|\hpi - \bpi^\star\|_2=O_\mathbb{P}(n^{-r}).$$ 

\end{proof}

\begin{lemma}\label{lem:Delta}
Under the assumptions of Theorem~\ref{thm:norma}, we have
$$
\Delta_{1,\bpi^\star,\bpsi^\star,\bpsi,k} =  \sum_{k=1}^K \sum_{j=1}^J O\left(\left\| \psi_{k,j} - \psi_{k,j}^\star\right\|_{L_1}^2\right) + O(h^2)
$$
and
$$
\Delta_{2,\bpi^\star,\bpsi^\star,\bpsi,k} =  \sum_{k=1}^K \sum_{j=1}^J O\left(\left\| \psi_{k,j} - \psi_{k,j}^\star\right\|_{L_1}^2\right) + O(h^2),
$$
with 
$$ 
\Delta_{1,\bpi^\star,\bpsi^\star,\bpsi,k}=-\Eg \left[\effscore_{\bpi^\star,\bpsi^\star,k}^{(h)}(\bX_1)\kappa_{\bpsi,1}^{(h)}(\bX_1)\right]
$$
and
$$ 
\Delta_{2,\bpi^\star,\bpsi^\star,\bpsi,k}=- \Eg \left[\kappa_{\bpsi,2}^{(h)}(\bX_1)\kappa_{\bpsi,3,k}^{(h)}(\bX_1)\right],
$$
where $
\kappa_{\bpsi,1}^{(h)}=\kappa_{\bpsi,2}^{(h)} - \noisy_{\bpi^\star,\bpsi^\star}^{(h)}[\bpsi - \bpsi^\star]
$, 
$
\kappa_{\bpsi,2}^{(h)}=[f^{(h)}_{\bpi^\star,\bpsi}-f^{(h)}_{\bpi^\star,\bpsi^\star}]/f^{(h)}_{\bpi^\star,\bpsi^\star}
$ and
$
\kappa_{\bpsi,3,k}^{(h)}=\dot{\bmap}^{(h)}_k(\bpi^\star,\bpi^\star,\bpsi)-\dot{\bmap}^{(h)}_k(\bpi^\star,\bpi^\star,\bpsi^\star)
$.
\end{lemma}

\begin{proof}[Proof of Lemma~\ref{lem:Delta}]
We have
$$
 | \Delta_{1,\bpi^\star,\bpsi^\star,\bpsi,k}| \leq   \left\|\effscore_{\bpi^\star,\bpsi^\star,k}^{(h)}\right\|_\infty   \left\|\kappa_{\bpsi,1}^{(h)}  \right\|_{L_1(g^\star)}.$$
Using the definition of the efficient score function with smoothing given by \eqref{eq:effscore} as well as the definition of the nuisance score function with smoothing, we have
$$
\effscore_{\bpi^\star,\bpsi^\star,k}^{(h)} = \score_{\bpi^\star,\bpsi^\star,k}^{(h)}  - \sum_{k'=1}^K \pi^\star_{k'} \score_{\bpi^\star,\bpsi^\star,k'}^{(h)} \zeta^{(h)}_{\bpsi^\star,\check\bpsi^{(h)},k'} .
$$
where $\check\bpsi^{(h)}\in\Psi(\mathcal{X})$ satisfies \eqref{eq:condproj}.
For any $\bpsi$, we have from \eqref{eq:naive} that $-1/\pi_K^\star \leq \score_{\bpi^\star,\bpsi,k}^{(h)}\leq 1/\pi^\star_k$. 
In addition, using the definition of $\Psi(\mathcal{X})$, $\zeta^{(h)}_{\bpsi^\star,\check\bpsi^{(h)},k}$ is bounded uniformly in $h$ due to Lemma ~\ref{lem:control_sup_normdev}.
Therefore,  there exists a positive constant $C$ such that  $\| \effscore_{\bpi^\star,\bpsi^\star,k}^{(h)} \|_\infty \leq C$,
leading that
\begin{equation}\label{eq:maj1}
| \Delta_{1,\bpi^\star,\bpsi^\star,\bpsi,k} | \leq C \left\|\kappa_{\bpsi,1}^{(h)}  \right\|_{L_1(g^\star)} .
\end{equation}
By Cauchy-Schwarz inequality, we have
\begin{equation}\label{eq:maj2}
 | \Delta_{2,\bpi^\star,\bpsi^\star,\bpsi,k}| \leq   \left\|\kappa_{\bpsi,2}^{(h)}  \right\|_{L_2(g^\star)}   \left\|\kappa_{\bpsi,3,k}^{(h)}  \right\|_{L_2(g^\star)}.
\end{equation}
    To control $\Delta_{1,\bpi^\star,\bpsi^\star,\bpsi,k}$ and $\Delta_{2,\bpi^\star,\bpsi^\star,\bpsi,k}$, it suffices to control $L_p(g^\star)$-norms of $\kappa_{\bpsi,1}^{(h)}$, $\kappa_{\bpsi,2}^{(h)}$ and $\kappa_{\bpsi,3,k}^{(h)}$. These controls can be done by noting that these three terms can be defined as the remainder with integral form of Taylor expansions using Gateaux derivatives. 
     To give the expressions of these Taylor expansions, we denote by $f_{\bpi^\star,\bpsi}^{(h)'}[\delta]$ and $f_{\bpi^\star,\bpsi}^{(h)''}[\delta][\delta]$, the first and second order derivatives of $\bpsi\mapsto f_{\bpi,\bpsi}^{(h)}$ in direction $\delta$ at $\bpsi$. 
Hence, we have
$$
 f_{\bpi^\star,\bpsi}^{(h)'} [\delta] = \sum_{k=1}^K \pi_k^\star  \left(\prod_{j=1}^J \smooth_j \psi_{k,j} \right)\chi_{1,\bpsi,\delta,k}^{(h)},
$$
and
$$
 f_{\bpi^\star,\bpsi}^{(h)''} [\delta][\delta] = \sum_{k=1}^K  \pi_k^\star \left( \prod_{j=1}^J \smooth_j \psi_{k,j}\right)\left[(\chi_{1,\bpsi,\delta,k}^{(h)})^2 - \chi_{2,\bpsi,\delta,k}^{(h)}\right] ,
$$
with
$$
\chi_{u,\bpsi,\delta,k}^{(h)}=\sum_{\ell=1}^J \kernel_h \star \left(\frac{\delta_{k,\ell} }{\psi_{k,\ell}  }\right)^u,
$$
Noting that any function $q$, we have $(\kernel_h \star q)(x_j)=\mathbb{E}_\kernel\left[q(x_j + Vh)\right]$, we have
\begin{multline*}
(\chi_{1,\bpsi,\delta,k}^{(h)})^2(\bx) - \chi_{2,\bpsi,\delta,k}^{(h)}(\bx)  =\\ \sum_{\ell=1}^J \sum_{\ell'\neq \ell} \mathbb{E}_\kernel\left[  \frac{\delta_{k,\ell}(x_\ell + Vh) }{\psi_{k,\ell}(x_\ell + Vh)  }\right]\mathbb{E}_\kernel\left[\frac{\delta_{k,\ell'}(x_{\ell'} + Vh) }{\psi_{k,\ell'}(x_{\ell'} + Vh)  }\right]  - \sum_{\ell=1}^J \text{Var}_\kernel\left[ \frac{\delta_{k,\ell}(x_\ell + Vh) }{\psi_{k,\ell}(x_\ell + Vh)  }\right].
\end{multline*}
Hence, we have
\begin{multline*}
f_{\bpi^\star,\bpsi}^{(h)''} [\delta][\delta](\bx) = \sum_{k=1}^K  \sum_{\ell=1}^J \sum_{\ell'\neq \ell} 
\pi_k^\star \epsilon_{a,\bpsi,\delta,k,\ell}^{(h)}(x_\ell)\epsilon_{a,\bpsi,\delta,k,\ell'}^{(h)}(x_{\ell'})  \left( \prod_{j\notin\{\ell,\ell'\}} \smooth_j \psi_{k,j}(x_j)\right) \\
- \sum_{k=1}^K  \sum_{\ell=1}^J \pi_k^\star \epsilon_{b,\bpsi,\delta,k,\ell}^{(h)}(x_\ell)\left( \prod_{j \neq \ell } \smooth_j \psi_{k,j}(x_j)\right)
\end{multline*}
with
\begin{equation}\label{eq:epsilona}
\epsilon_{a,\bpsi,\delta,k,\ell}^{(h)}(u) = \smooth_\ell \psi_{k,\ell} (u)\mathbb{E}_\kernel\left[  \frac{\delta_{k,\ell}(u + Vh) }{\psi_{k,\ell}(u + Vh)  }\right]
\end{equation}
and
\begin{equation}\label{eq:epsilonb}
\epsilon_{b,\bpsi,\delta,k,\ell}^{(h)}(u) = \smooth_\ell \psi_{k,\ell} (u)\text{Var}_\kernel\left[  \frac{\delta_{k,\ell}(u + Vh) }{\psi_{k,\ell}(u + Vh)  }\right].
\end{equation}
We consider the direction $\delta_{\bpsi}:=\bpsi - \bpsi^\star$ defined such that for each element $(k,j)$ we have $\delta_{\bpsi,k,j}=\psi_{k,j} - \psi_{k,j}^\star$ and the parameter  defined for $t\in[0,1]$ by $\bpsi_t:=\bpsi^\star + t\delta_{\bpsi}$,  leading that each element is defined by $\psi_{t,k,j}=\psi_{k,j}^\star + t\delta_{\bpsi,k,j}$. Note that using the definition of the naive score function with smoothing given by \eqref{eq:naive} as well as the definition of the nuisance score function with smoothing, we have $$\frac{f_{\bpi^\star,\bpsi^\star}^{(h)'}[\delta_{\bpsi}]}{f_{\bpi^\star,\bpsi^\star}^{(h)}}= \noisy_{\bpi^\star,\bpsi^\star}^{(h)}[\delta_{\bpsi}].$$
A Taylor expansion at order 2 of $\bpsi \mapsto f^{(h)}_{\bpi^\star,\bpsi} $ around $\bpsi = \bpsi^\star$ in direction $\delta_{\bpsi}$ implies that
$$
f^{(h)}_{\bpi^\star,\bpsi}=f^{(h)}_{\bpi^\star,\bpsi^\star} + f^{(h)'}_{\bpi^\star,\bpsi^\star} [\delta_{\bpsi}] +f^{(h)''}_{\bpi^\star,\bpsi^\star} [\delta_{\bpsi}] [\delta_{\bpsi}] + f_{\bpi^\star,\bpsi^\star}^{(h)}r_{1,\bpsi^\star,\bpsi},
$$
where $r_{1,\bpsi^\star,\bpsi}=\int_0^1 (1-t)  f_{\bpi^\star,\bpsi_t}^{(h)''}[\delta_{\bpsi}][\delta_{\bpsi}] /f_{\bpi^\star,\bpsi^\star}^{(h)} dt - f^{(h)''}_{\bpi^\star,\bpsi^\star} [\delta_{\bpsi}] [\delta_{\bpsi}]/f_{\bpi^\star,\bpsi^\star}^{(h)}$. 
 Dividing both sides of the previous equation by $f^{(h)}_{\bpi^\star,\bpsi^\star}$ implies that
$$
\kappa_{\bpsi,2}^{(h)} = \noisy_{\bpi^\star,\bpsi^\star}^{(h)}[\delta_{\bpsi}] +   \frac{f_{\bpi^\star,\bpsi^\star}^{(h)''}}{f_{\bpi^\star,\bpsi^\star}^{(h)}} + r_{1,\bpsi^\star,\bpsi}.
$$
Hence, using the definition of $\kappa_{\bpsi,1}^{(h)}$, the previous equation implies $\kappa_{\bpsi,1}^{(h)}$ is the remainder with integral form of a Taylor expansion at order 2 of $\bpsi \mapsto f^{(h)}_{\bpi^\star,\bpsi}/f^{(h)}_{\bpi^\star,\bpsi^\star}$ around $\bpsi = \bpsi^\star$ in direction $\delta_{\bpsi}$,  such that
$$\kappa_{\bpsi,1}^{(h)} =  \frac{f_{\bpi^\star,\bpsi^\star}^{(h)''}}{f_{\bpi^\star,\bpsi^\star}^{(h)}} + r_{1,\bpsi^\star,\bpsi}.$$

\paragraph*{Controlling the $L_1(g^\star)$-norm of $\kappa_{\bpsi,1}^{(h)}$}
Since     $f_{\bpi^\star,\bpsi^\star}^{(h)''}$ is a continuous function of $\bpsi$, we have
$
\|r_{1,\bpsi^\star,\bpsi}\|_{L^1(g^\star)} = o\left( \|  f_{\bpi^\star,\bpsi^\star}^{(h)''}/f_{\bpi^\star,\bpsi^\star}^{(h)}  \|_{L^1(g^\star)} \right)
$, hence we have
$$
\|\kappa_{\bpsi,1}^{(h)}\|_{L_1(g^\star)}  = (1+o(1)) \int_{\mathcal{X}} \left|f_{\bpi^\star,\bpsi^\star}^{(h)''}[\delta_{\bpsi}][\delta_{\bpsi}](\bx) \right|\frac{g^\star(\bx)}{f_{\bpi^\star,\bpsi^\star}^{(h)}(\bx)}d\bx.
$$
We have
\begin{align*}
    \frac{g^\star(\bx)}{f_{\bpi^\star,\bpsi^\star}^{(h)}(\bx)} & = \sum_{k=1}^K \frac{\pi_k^\star \prod_{j=1}^J \psi_{k,j}^\star}{ \sum_{\ell=1}^K \pi_\ell^\star \prod_{j=1}^J \smooth_j \psi_{\ell,j}^\star} \\
    & \leq  \sum_{k=1}^K\frac{\pi_k^\star \prod_{j=1}^J \psi_{k,j}^\star}{  \pi_k^\star \prod_{j=1}^J \smooth_j \psi_{k,j}^\star} \\
    & = \sum_{k=1}^K \exp\left(-\frac{h^2 \nu_{\kernel,2}}{2} \sum_{j=1}^J [\ln \psi_{k,j}^\star ]'' + o(h^2)\right).
\end{align*}
Hence, there exits a positive constant $C^\star$ such that
  $   \frac{g^\star(\bx)}{f_{\bpi^\star,\bpsi^\star}^{(h)}(\bx)}  \leq 1+C^\star h^2$,
leading, since $h=o(1)$, that
$
\|\kappa_{\bpsi,1}^{(h)}\|_{L_1(g^\star)}  =(1+o(1)) \int_{\mathcal{X}} \left|f_{\bpi^\star,\bpsi^\star}^{(h)''}[\delta_{\bpsi}][\delta_{\bpsi}](\bx) \right|d\bx.
$
Using the definition of $f_{\bpi^\star,\bpsi}^{(h)''} [\delta][\delta]$, we have
$$
|f_{\bpi^\star,\bpsi^\star}^{(h)''} [\delta_{\bpsi}][\delta_{\bpsi}](\bx) |\leq \sum_{k=1}^K  \sum_{\ell=1}^J \sum_{\ell'\neq \ell} 
|m_{a,\delta_{\bpsi},k,\ell,\ell'}(\bx)| + \sum_{k=1}^K  \sum_{\ell=1}^J |m_{b,\delta_{\bpsi},k,\ell}(\bx)|,
$$
where
$
m_{a,\delta_{\bpsi},k,\ell,\ell'}(\bx)= \epsilon_{a,\bpsi^\star,\delta_{\bpsi},k,\ell}^{(h)}(x_\ell)\epsilon_{a,\bpsi^\star,\delta_{\bpsi},k,\ell'}^{(h)}(x_{\ell'})  \left( \prod_{j\notin\{\ell,\ell'\}} \smooth_j \psi_{k,j}^\star(x_j)\right)
$ 
and
$
m_{b,\delta_{\bpsi},k,\ell,}(\bx)= \epsilon_{b,\bpsi^\star,\delta_{\bpsi},k,\ell}^{(h)}(x_\ell)  \left( \prod_{j\neq \ell} \smooth_j \psi_{k,j}^\star(x_j)\right).
$
Hence, we have
$$
\|\kappa_{\bpsi,1}^{(h)}\|_{L_1(g^\star)}  \lesssim \sum_{k=1}^K  \sum_{\ell=1}^J \sum_{\ell'\neq \ell} 
\|m_{a,\delta_{\bpsi},k,\ell,\ell'}\|_{L^1} + \sum_{k=1}^K  \sum_{\ell=1}^J \|m_{b,\delta_{\bpsi},k,\ell}\|_{L^1}.
$$

We now need to control the integrals of the absolute values of $m_{a,\delta_{\bpsi},k,\ell,\ell'}(\bx)$ and $m_{b,\delta_{\bpsi},k,\ell}(\bx)$. To do so, we need to investigate $\epsilon_{a,\bpsi^\star,\delta_{\bpsi},k,\ell}^{(h)}$ and $\epsilon_{b,\bpsi^\star,\delta_{\bpsi},k,\ell}^{(h)}$. Using the definition of $\epsilon_{a,\bpsi^\star,\delta_{\bpsi},k,\ell}^{(h)}$, we have
$
\|\epsilon_{a,\bpsi^\star,\delta_{\bpsi},k,\ell}^{(h)}\|_{L^1} \leq \int_{\mathcal{X}_j^2} \kernel(v) \left| \frac{\smooth_\ell \psi^\star_{k,\ell}(u)}{\psi^\star_{k,\ell}(u+vh)}\right| |\psi_{k,\ell}(u)-\psi^\star_{k,\ell}(u)|dudv
$. Hence, using the variable change $t=u+vh$, we have
$$
\|\epsilon_{a,\bpsi^\star,\delta_{\bpsi},k,\ell}^{(h)}\|_{L^1} \leq \int_{\mathcal{X}_j^2} \kernel(v) \left| \frac{\smooth_\ell \psi^\star_{k,\ell}(t-vh)}{\psi^\star_{k,\ell}(t)}\right| |\psi_{k,\ell}(t)-\psi^\star_{k,\ell}(t)|dtdv
$$
Using a Taylor expansion of $\smooth_\ell \psi^\star_{k,\ell}$ around $s$ and noting that the second order derivative de $\ln \psi^\star$ is bounded by $C_3$, we have
\begin{equation}\label{eq:Taylorratio}
\frac{\smooth_\ell \psi^\star_{k,\ell}(t-vh)}{\psi^\star_{k,\ell}(t)}  = \exp\left( -vh [\ln \psi^\star_{k,\ell}]'(t) + \frac{h^2}{2} \rho(v,t)\right),
\end{equation}
where $\|\rho(v,\cdot)\|_\infty\leq C_3 M_\kernel(v)$ where $M_\kernel(v) = \int_{\mathcal{X}_j} \kernel(w) (v+w)^2dw$. Hence, we have
$$
\|\epsilon_{a,\bpsi^\star,\delta_{\bpsi},k,\ell}^{(h)}\|_{L^1} \leq \int_{\mathcal{X}_j^2} \kernel(v) \exp\left( -vh [\ln \psi^\star_{k,\ell}]'(t) + \frac{h^2}{2}C_3 M_\kernel(v)\right) |\psi_{k,\ell}(t)-\psi^\star_{k,\ell}(t)|dtdv
$$
Due to the assumptions made on the kernel, we have $\int_{\mathcal{X}_j} \kernel(v) \exp\left(h^2C_3 M_\kernel(v)\right)dv=O(1)$ and 
that if $s$ are small enough, then $\mathbb{E}_\kernel[\exp(Vs)]\leq 1+O(s^2)$ leading that
$$
\int_{\mathcal{X}_j} \kernel(v) \exp\left(-2vh [\ln \psi^\star_{k,\ell}]'(t)\right)  \leq 1 + O((h[\ln \psi^\star_{k,\ell}]'(t))^2).
$$
Hence, Cauchy-Schwarz inequality implies that 
$$
\|\epsilon_{a,\bpsi^\star,\delta_{\bpsi},k,\ell}^{(h)}\|_{L^1} \lesssim  \int_{\mathcal{X}_j} (1+h|[\ln \psi^\star_{k,\ell}]'(t))|) |\psi_{k,\ell}(t)-\psi^\star_{k,\ell}(t)|dt.
$$
Using the definition of $\Psi(\mathcal{X}_j)$, the integral in the previous equation is upperbounded by $2C_2  \int_{\mathcal{X}_j} |[\ln \psi^\star_{k,\ell}]'(t))|dt$ and that $\int_{\mathcal{X}_j} |[\ln \psi^\star_{k,\ell}]'(t))|dt$ is finite, leading that
$$
\|\epsilon_{a,\bpsi^\star,\delta_{\bpsi},k,\ell}^{(h)}\|_{L^1} = O(\|\delta_{\bpsi,k,\ell}\|_{L^1}).
$$
Noting that $ \| \delta_{\bpsi,k,\ell}\|_{L_1}\| \delta_{\bpsi,k,\ell'}\|_{L_1}\leq  \| \delta_{\bpsi,k,\ell}\|_{L_1}^2+  \| \delta_{\bpsi,k,\ell'}\|_{L_1}^2 $, we have
\begin{equation}\label{eq:ma}
  \int_{\mathcal{X}} |m_{a,\delta_{\bpsi},k,\ell,\ell'}(\bx)| d\bx=     \sum_{k=1}^K  \sum_{\ell=1}^JO(\| \delta_{\bpsi,k,\ell}\|_{L_1}^2).
\end{equation}
Using the definition of $\epsilon_{b,\bpsi,\delta,k,\ell}^{(h)}$ given in \eqref{eq:epsilonb}, we have
$$
\epsilon_{b,\bpsi^\star,\delta_{\bpsi},k,\ell}^{(h)}(u) = [\smooth_\ell \psi_{k,\ell}^\star (u)]^{1/2}\text{Var}_\kernel\left[  \frac{\smooth_\ell \psi_{k,\ell}^\star (u)}{\psi_{k,\ell}^\star(u + Vh) } \delta_{\bpsi,k,\ell}(u + Vh) \right].
$$
Since the elements of $\Psi(\mathcal{X}_j)$ are bounded, then 
$$
\epsilon_{b,\bpsi^\star,\delta_{\bpsi},k,\ell}^{(h)}(u) \lesssim
\text{Var}_\kernel\left[ \exp\left( -vh [\ln \psi^\star_{k,\ell}]'(u)\right)\right].$$
In addition, since the kernel is Gaussian or sub-Gaussian, we have that if   $s$ is small enough, then $\text{Var}_\kernel[\exp(Vs)]\leq O(s^2)$ leading that
$$
\epsilon_{b,\bpsi^\star,\delta_{\bpsi},k,\ell}^{(h)}(u) \lesssim
\left(h [\ln \psi^\star_{k,\ell}]'(u)\right)^2.$$
Since the integral of $\left([\ln \psi^\star_{k,\ell}]'(u)\right)^2$ is finite by definition of  $\Psi(\mathcal{X}_j)$, then we have $\|\epsilon_{b,\bpsi^\star,\delta_{\bpsi},k,\ell}^{(h)}\|_{L^1}=O(h^2)$, 
leading that
\begin{equation}\label{eq:mb}
  \int_{\mathcal{X}} |m_{b,\delta_{\bpsi},k,\ell}(\bx)| d\bx=   O(h^2).
\end{equation}
Combining \eqref{eq:ma} and \eqref{eq:mb} gives
$$
\|\kappa_{\bpsi,1}^{(h)}\|_{L_1(g^\star)}  =   \sum_{k=1}^K  \sum_{\ell=1}^JO(\| \delta_{\bpsi,k,\ell}\|_{L_1}^2) +  O(h^2).
$$

\paragraph*{Controlling the $L_2(g^\star)$-norm of $\kappa_{\bpsi,2}^{(h)}$}
A Taylor expansion considering a remainder with integral form implies that 
$
 f_{\bpi^\star,\bpsi }^{(h)}  -  f_{\bpi^\star,\bpsi^\star}^{(h)} =  f_{\bpi^\star,\bpsi^\star}^{(h)'} [\delta_{\bpsi}]  +  f_{\bpi^\star,\bpsi }^{(h)} r_{2,\bpsi^\star,\bpsi}  
$
where 
$
 r_{2,\bpsi^\star,\bpsi} = \frac{1}{f_{\bpi^\star,\bpsi }^{(h)}} \left[ \int_{0}^1   f_{\bpi^\star,\bpsi_t}^{(h)'} [\delta_{\bpsi}] dt -  f_{\bpi^\star,\bpsi^\star}^{(h)'} [\delta_{\bpsi}] \right].
$
Divinding both sides of the previous equation  by $f_{\bpi^\star,\bpsi^\star}^{(h)}$ and using the definition of  $\kappa_{\bpsi,2}^{(h)}$ imply
$$
\kappa_{\bpsi,2}^{(h)} = \frac{ f_{\bpi^\star,\bpsi^\star}^{(h)'} [\delta_{\bpsi}]}{f_{\bpi^\star,\bpsi^\star}^{(h)} } +  r_{2,\bpsi^\star,\bpsi} .
$$
Since     $f_{\bpi^\star,\bpsi^\star}^{(h)'}$ is a continuous function of $\bpsi$ and that $ \|\frac{g^\star}{f_{\bpi^\star,\bpsi^\star}^{(h)}}\|_\infty=O(1)$, then we have
$$
\|\kappa_{\bpsi,2}^{(h)}\|_{L_2(g^\star)}^2   \lesssim  \int_{\mathcal{X}}  \left|f_{\bpi^\star,\bpsi^\star}^{(h)'} [\delta_{\bpsi}](\bx)\right|\frac{|f_{\bpi^\star,\bpsi^\star}^{(h)'} [\delta_{\bpsi}](\bx)|}{f_{\bpi^\star,\bpsi^\star}^{(h)}(\bx)} d\bx.
$$
Note that 
$$
\frac{|f_{\bpi^\star,\bpsi_t}^{(h)'} [\delta_{\bpsi}](\bx)|}{f_{\bpi^\star,\bpsi^\star}^{(h)}(\bx)} \leq \sum_{k=1}^K |\chi_{1,\bpsi^\star,\delta_{\bpsi},k}^{(h)}(\bx)|.
$$
Hence, there exists a positive constant $C$ such that
$$
\|\kappa_{\bpsi,2}^{(h)}\|_{L_2(g^\star)}^2   \lesssim   \int_{\mathcal{X}} \sum_{k=1}^K  \sum_{k'=1}^K \pi_k^\star  \left(\prod_{j=1}^J \smooth_j \psi_{k,j}^\star (x_j)\right)[\chi_{1,\bpsi^\star,\delta_{\bpsi},k'}^{(h)}(\bx)]^2d \bx
$$
Hence, we have $
\|\kappa_{\bpsi,2}^{(h)}\|_{L_2(g^\star)}^2   \lesssim   \sum_{k=1}^K  \sum_{\ell=1}^J \sum_{\ell'\neq \ell} 
\int_{\mathcal{X}} |m_{a,\delta_{\bpsi},k,\ell,\ell'}(\bx)|  d\bx
$, then using \eqref{eq:ma}, we have
$$
\|\kappa_{\bpsi,2}^{(h)}\|_{L_2(g^\star)}^2    = \sum_{k=1}^K  \sum_{\ell=1}^JO(\| \delta_{\bpsi,k,\ell}\|_{L_1}^2).
$$

\paragraph*{Controlling the $L_2(g^\star)$-norm of $\kappa_{\bpsi,3}^{(h)}$}
Using the definition of $ \dot{\bmap}^{(h)}_k$, we have for any $k$
$$
\dot{\bmap}^{(h)}_k(\bpi^\star,\bpi^\star,\bpsi)  = \score_{\bpi^\star,\bpsi,k}^{(h)}  - \sum_{k'=1}^K \pi_{k'}^\star \score_{\bpi^\star,\bpsi,k'}^{(h)} \phi^{(h)}_{\bpi^\star,\bpi^\star,\bpsi,k'} .
$$
Let  $\dot{\bmap}^{(h)'}_k(\bpi^\star,\bpi^\star,\bpsi) [\delta] $ be the derivative of $\bpsi\mapsto\dot{\bmap}^{(h)}_k(\bpi^\star,\bpi^\star,\bpsi)$ in direction $\delta$ at $\bpsi$. We have
$$
\dot{\bmap}^{(h)'}_k(\bpi^\star,\bpi^\star,\bpsi ) [\delta]=    \score_{\bpi^\star,\bpsi,k}^{(h)'}  [\delta] -  \sum_{k'=1}^K \pi_{k'}^\star\left(  \score_{\bpi^\star,\bpsi,k'}^{(h)'} [\delta] \phi^{(h)}_{\bpi^\star,\bpi^\star,\bpsi,k'} +\score_{\bpi^\star,\bpsi,k'}^{(h)}    \phi^{(h)'}_{\bpi^\star,\bpi^\star,\bpsi,k'}[\delta]\right),
$$
where $ \score_{\bpi^\star,\bpsi, k}^{(h)'}  [\delta]$ and $   \phi^{(h)'}_{\bpi^\star,\bpi^\star,\bpsi,k}[\delta]$ are the partial derivatives of $\bpsi\mapsto \score_{\bpi^\star,\bpsi,k}^{(h)} $ and $\bpsi\mapsto \phi^{(h)}_{\bpi^\star,\bpi^\star,\bpsi  ,k}$ in direction $\delta$ evaluated at $\bpsi$. Hence, we have
$$
   \score_{\bpi^\star,\bpsi,k}^{(h)'}  [\delta] =
\score_{\bpi^\star,\bpsi,k}^{(h)} \left(\chi_{1,\bpsi,\delta,k}^{(h)} -  \sum_{k'=1}^K \pi_{k'}^\star\score_{\bpi^\star,\bpsi,k'}^{(h)} \chi_{1,\bpsi,\delta,k'}^{(h)} \right)
$$
and
$$  \phi^{(h)'}_{\bpi^\star,\bpi^\star,\bpsi,k}[\delta] = 
\chi_{1,\bpsi,\delta,k}^{(h)} - \sum_{j=1}^J \kernel_h \star  
\frac{(\psi_{k,j}  - \check\psi_{k,j})\delta_{k,j}}{\psi_{k,j}^2}.
$$
Hence, using Lemma~\ref{lem:control_sup_normdev}, we have
\begin{equation}\label{eq:phiprime}
| \phi^{(h)'}_{\bpi^\star,\bpi^\star,\bpsi^\star,k}[\delta_{\bpsi}]|\lesssim   |\chi_{1,\bpsi^\star,\delta_{\bpsi},k}^{(h)}|.
\end{equation}
From the definition of $\kappa_{\bpsi,3,k}^{(h)}$, a Taylor expansion considering a remainder with integral form implies that 
$$
\kappa_{\bpsi,3,k}^{(h)} =  \dot{\bmap}^{(h)'}_k(\bpi^\star,\bpi^\star,\bpsi^\star ) [\delta_{\bpsi}] + r_{3,\bpsi^\star,\bpsi},$$
where the reminder term is equal to
$$
r_{3,\bpsi^\star,\bpsi}=\int_{0}^1  \left( \dot{\bmap}^{(h)'}_k(\bpi^\star,\bpi^\star,\bpsi_t ) [\delta_{\bpsi}] -  \dot{\bmap}^{(h)'}_k(\bpi^\star,\bpi^\star,\bpsi^\star ) [\delta_{\bpsi}]\right) dt.
$$
Since $\score_{\bpi^\star,\bpsi,k}^{(h)}  [\delta_{\bpsi}]$, $   \score_{\bpi^\star,\bpsi,k}^{(h)'}  [\delta_{\bpsi}]$, $\phi^{(h)}_{\bpi^\star,\bpi^\star,\bpsi,k}$ and $\phi^{(h)'}_{\bpi^\star,\bpi^\star,\bpsi,k}$ are continuous function of $\bpsi$ for any $k$, we have that  $\dot{\bmap}^{(h)'}_k(\bpi^\star,\bpi^\star,\bpsi)$ is a continuous function of $\bpsi$, leading that
$$
\|r_{3,\bpsi^\star,\bpsi}\|_{L^2(g^\star)} = o\left( \|  \dot{\bmap}^{(h)'}_k(\bpi^\star,\bpi^\star,\bpsi^\star ) [\delta_{\bpsi}]  \|_{L^2(g^\star)} \right).
$$
We have
\begin{multline*}
     \left\|   \dot{\bmap}^{(h)'}_k(\bpi^\star,\bpi^\star,\bpsi^\star) ) [ \delta_{\bpsi}] \right\|_{L_2(g^\star)} \leq   \left\|   \score_{\bpi^\star,\bpsi^\star,k}^{(h)'}  [ \delta_{\bpsi}] \right\|_{L_2(g^\star)} +\\ \sum_{k'=1}^K \pi_{k'}^\star\left(    \left\| \score_{\bpi^\star,\bpsi^\star,k'}^{(h)'} [\delta_{\bpsi}] \phi^{(h)}_{\bpi^\star,\bpi^\star,\bpsi^\star,k'} \right\|_{L_2(g^\star)} +   \right. \left. \left\| \score_{\bpi^\star,\bpsi^\star,k'}^{(h)}     \phi^{(h)'}_{\bpi^\star,\bpi^\star,\bpsi^\star,k'}[\delta_{\bpsi}] \right\|_{L_2(g^\star)}\right).
\end{multline*}
We have
\begin{align*}
     \left\|   \score_{\bpi^\star,\bpsi^\star,k}^{(h)'}  [ \delta_{\bpsi}] \right\|_{L_2(g^\star)}  &\leq \sum_{\ell=1}^K  \left\|   \score_{\bpi^\star,\bpsi^\star,k}^{(h)}  [ \delta_{\bpsi}] \chi_{1,\bpsi^\star,\delta_{\bpsi},\ell}^{(h)}\right\|_{L_2(g^\star)} \\
     &\leq \sum_{\ell=1}^K \left[\int_{\mathcal{X}}   \left(\prod_{j=1}^J \smooth_j \psi_{\ell,j}^\star (x_j)\right)^2 \left(\chi_{1,\bpsi^\star,\delta_{\bpsi},\ell}^{(h)}(\bx)\right)^2 d\bx\right]^{1/2}.
\end{align*}
Since $\smooth_j \psi_{\ell,j}^\star $ is  bounded, we have
$
     \left\|   \score_{\bpi^\star,\bpsi^\star,k}^{(h)'}  [ \delta_{\bpsi}] \right\|_{L_2(g^\star)}  \leq
\sum_{k=1}^K  \sum_{\ell=1}^J \sum_{\ell'\neq \ell} 
\int_{\mathcal{X}} |m_{a,\delta_{\bpsi},k,\ell,\ell'}(\bx)|  d\bx$,
leading using \eqref{eq:ma}, we have
$$
   \left\|   \score_{\bpi^\star,\bpsi^\star,k}^{(h)'}  [ \delta_{\bpsi}] \right\|_{L_2(g^\star)}^2    = \sum_{k=1}^K  \sum_{\ell=1}^JO(\| \delta_{\bpsi,k,\ell}\|_{L_1}^2).
$$
Noting that by definition of $\Psi(\mathcal{X})$, $ \phi^{(h)}_{\bpi^\star,\bpi^\star,\bpsi^\star,k'}$ is upperbounded, then 
$$
  \left\| \score_{\bpi^\star,\bpsi^\star,k'}^{(h)'} [\delta_{\bpsi}] \phi^{(h)}_{\bpi^\star,\bpi^\star,\bpsi^\star,k'} \right\|_{L_2(g^\star)}  = \sum_{k=1}^K  \sum_{\ell=1}^JO(\| \delta_{\bpsi,k,\ell}\|_{L_1}^2).
$$
From \eqref{eq:phiprime}, we have
$$
 \left\| \score_{\bpi^\star,\bpsi^\star,k'}^{(h)}     \phi^{(h)'}_{\bpi^\star,\bpi^\star,\bpsi^\star,k'}[\delta_{\bpsi}] \right\|_{L_2(g^\star)} \lesssim  \left\| \score_{\bpi^\star,\bpsi^\star,k'}^{(h)}     \chi_{1,\bpsi^\star,\delta_{\bpsi},\ell}^{(h)} \right\|_{L_2(g^\star)}.
$$
Therefore, with the same argument that thoses used to control $\left\|   \score_{\bpi^\star,\bpsi^\star,k}^{(h)'}  [ \delta_{\bpsi}] \right\|_{L_2(g^\star)} $, we obtain that
$$
 \left\| \score_{\bpi^\star,\bpsi^\star,k'}^{(h)}     \phi^{(h)'}_{\bpi^\star,\bpi^\star,\bpsi^\star,k'}[\delta_{\bpsi}] \right\|_{L_2(g^\star)} = \sum_{k=1}^K  \sum_{\ell=1}^JO(\| \delta_{\bpsi,k,\ell}\|_{L_1}^2).
$$
Hence, we have
$$
\|\kappa_{\bpsi,3}^{(h)}\|^2_{L^2(g^\star)} = \sum_{k=1}^K  \sum_{\ell=1}^JO(\| \delta_{\bpsi,k,\ell}\|_{L_1}^2).
$$
\end{proof}

\section{Extension to the variables defined on the real line}

\begin{proof}[Proof of Theorem~\ref{thm:normabis}]
    With a careful reading of the proof of Lemma~\ref{lem:taylor}, we can see that the compactness of $\mathcal{X}_j$ is not used. Therefore, Lemma~\ref{lem:taylor} still hold true when $\mathcal{X}_j=\mathbb{R}$.
    Lemma~\ref{lem:deviation} uses the argument of compactness of $\mathcal{X}_j$ and thus cannot be used anymore. It is replaced by Lemma~\ref{lem:deviationbis}.
    Hence, we are able to state the consistency of the estimator. Indeed, following the same steps that the proof of Theorem~\ref{thm:consistence}, we have 
$$|\loss^{(0)}(\hpi,\hpsi) - \loss^{(0)}(\bpi^\star,\bpsi^\star)| = o_\mathbb{P}(1).$$
by noting that combing Lemmas~\ref{lem:taylor} and~\ref{lem:deviationbis} provides
$$
\sup_{(\bpi,\bpsi)\in\Theta_K}|\lossnh(\bpi,\bpsi) - \loss^{(0)}(\bpi,\bpsi)| = O_\mathbb{P}(n^{-1/2}h^{-1/4} + h^2).
$$
By the Fréchet–Kolmogorov theorem, the equicontinuity and uniform boundedness of  $\Psi(\mathbb{R})$ in the $L_2(g^\star)$-norm ensure that every sequence in 
 $\Psi(\mathbb{R})$ has a convergent subsequence; hence,  $\Psi(\mathbb{R})$ is sequentially compact in $L_2(g^\star)$. Therefore, the parameter space $\widetilde\Theta_K$ is sequentially compact in $L_2(g^\star)$.  Suppose, for the sake of contradiction, that $(\hpi,\hpsi)$ does not converge to $(\bpi^\star,\bpsi^\star)$  in probability for the $L_2(g^\star)$-norm. As the parameter space $\widetilde{\Theta}_K$ is sequentially compact, one can find a subsequence $(\hat\bpi_{h,n_k},\hat\bpsi_{h,n_k})_k$ which converges in probability for the $L_2(g^\star)$-norm to some $\tilde\btheta=(\widetilde\bpi,\widetilde\bpsi)$ such that $\|\btheta^\star-\tilde\btheta\|_{L_2(g^\star)}\neq 0$. By the continuity of $\loss^{(0)}$, $\loss^{(0)}(\hat\bpi_{h,n_k},\hat\bpsi_{h,n_k})$ converges in probability to $\loss^{(0)}(\widetilde\bpi,\widetilde\bpsi)$. On the other hand, by \eqref{eq:cvprloss}, $\loss^{(0)}(\hat\bpi_{h,n_k},\hat\bpsi_{h,n_k})$ converges in probability to $\loss^{(0)}(\bpi^\star,\bpsi^\star)$. Therefore, we have $\loss^{(0)}(\bpi^\star,\bpsi^\star) = \loss^{(0)}(\widetilde\bpi,\widetilde\bpsi)$. This contradicts the parameter identifiability property ensured by Assumption~\ref{ass:controlvariancetrue}, which implies that $(\bpi^\star,\bpsi^\star)$ is the unique minimizer of $\loss^{(0)}$. Therefore, $(\hpi,\hpsi)$ converges in probability to $(\bpi^\star,\bpsi^\star)$ for the $L_2(g^\star)$-norm.
Now note that Lemmas~\ref{lem:unicity}, \ref{lem:prel} and \ref{lem:inequality} do not use the argument of compactness of $\mathcal{X}_j$ and thus they still hold true when $\mathcal{X}_j=\mathbb{R}$. With a careful reading of the proof of Theorem~\ref{thm:rateL1}, and by replacing the callings of Lemma~\ref{lem:deviation} by the callings of Lemma~\ref{lem:deviationbis}. We have that under Assumptions~\ref{ass:controlvariancetrue} and \ref{ass:controlvariancekernel}, 
$$
\forall \bpi\in\mathcal{B}(\bpi^\star),\; \sum_{k=1}^K \sum_{j=1}^J \| \psi^\star_{k,j} - \widehat{\psi}^{(h,n,\bpi)}_{k,j}\|^2_1 = O_\mathbb{P}( n^{-1/2}h^{-1/2}+h^2+\|\bpi - \bpi^\star\|_1).
$$
Lemma~\ref{lem:invert} is only true for density functions $\psi_{k,j}^\star$ defined on compact sets. Using a quantile transformation as suggested in the Conclusion section of \cite{gassiat2018efficient}, a similar result can be established for some marginal densities whose support is defined on the real line. For example, the result will still be true for marginal densities with tails decaying at the same polynomial order in the same dimension as stated by Assumptions~\ref{ass:poly}. This result cannot be extended, however, to many other marginal densities.   
The other arguments used in the proof of Theorem~\ref{thm:norma} do not use the argument of compactness of $\mathcal{X}_j$. Therefore, under Assumptions~\ref{ass:controlvariancetrue}, \ref{ass:controlvariancekernel} and \ref{ass:band1}, the estimator of the proportions $\hpi$ converges at the rate $n^{-r}$ such that
    $$
    \|\hpi - \bpi^\star\|_1=O_\mathbb{P}(n^{-r}).
    $$
\end{proof}

\begin{lemma} \label{lem:deviationbis}
     Under Assumption~\ref{ass:controlvariancekernel}, the properties of $\widetilde{\Theta}_K$ ensures that
 $$
\sup_{(\bpi,\bpsi) \in \Theta_K} | \lossnh(\bpi,\bpsi) -  \lossh(\bpi,\bpsi)|= O_\mathbb{P}(n^{-1/2}h^{-1/2}).
 $$
\end{lemma}

\begin{proof}[Proof of Lemma~\ref{lem:deviationbis}]
With the same arguments that those used in the proof of Lemma~\ref{lem:deviation}, we have
     $\sup_{\gamma^{(h)} \in \Gamma^{(h)}(\mathbb{R})}     \|\gamma^{(h)}\|_\infty \leq  C_1$
      and
     $\sup_{\gamma^{(h)} \in \Gamma^{(h)}(\mathbb{R})}     \|\gamma^{(h)'}\|_\infty =\bar C_2 h^{-1}$.
     Hence, using \cite{van1994bracketing}, we have
     $
 \mathcal    H(\varepsilon;  \Gamma^{(h)}(\mathbb{R}),\|.\|_{L_2(g^\star)}) \lesssim 1/(\varepsilon h).
     $
     Therefore, the $\varepsilon$-entropy with bracketing of the $J$-dimensional product space $\Gamma^{(h)}(\mathbb{R}^J)=\Gamma^{(h)}(\mathcal{X}_1)\times \ldots \times \Gamma^{(h)}(\mathbb{R})$ is
     $$
 \mathcal    H(\varepsilon;  \Gamma^{(h)}(\mathbb{R}^J),\|.\|_{L_2(g^\star)}) \lesssim \frac{1}{\varepsilon h} .
     $$
Let $\tau^{(h)}_{\bpi,\bpsi}=\ln f^{(h)}_{\bpi,\bpsi}$, considering the space
     $
     \widetilde T_{h}(\mathcal{X})=\{\tau^{(h)}_{\bpi,\bpsi},\, (\bpi, \bpsi)\in\widetilde\Theta_K\},
     $
     we have
$
 \mathcal    H(\varepsilon;  \widetilde T_h(\mathbb{R}^d),\|.\|_{L_2(g^\star)} \lesssim \frac{1}{\varepsilon h}$. 
 Note that the class $\widetilde T_h(\mathbb{R}^d)$ admits an envelop having a finite $L_2(g^\star)$-norm since the elements of $\widetilde T_h(\mathbb{R}^d)$ are bounded and $g^\star$ is strictly positive and bounded. Hence, noting that 
            $     \int_0^\delta H^{1/2}(\varepsilon;  \widetilde T_h(\mathbb{R}^d),\|.\|_{L_2(g^\star)}) d\varepsilon \lesssim h^{-1/2}\delta$,
     then using \cite[Lemma 19.38]{van2000asymptotic}, we have
     $$
     \Eg \left[\sup_{(\bpi,\bpsi)\in\Theta_K(\mathcal{X})} \left|\frac{1}{n^{1/2}} \sum_{i=1}^n \ln f^{(h)}_{\bpi,\bpsi}(\bX_i) - \Eg\ln \ f^{(h)}_{\bpi,\bpsi}(\bX_i) \right|\right]=O(h^{-1/2}).
     $$
     The proof is concluded by noting that for any $(\bpi,\bpsi)$, we have $$ \lossnh(\bpi,\bpsi) -  \lossh(\bpi,\bpsi)=n^{-1/2}\left|\frac{1}{n^{1/2}} \sum_{i=1}^n \ln  f^{(h)}_{\bpi,\bpsi}(\bX_i) - \Eg\ln  f^{(h)}_{\bpi,\bpsi}(\bX_i) \right|,$$
     then by applying Markov's inequality.  
\end{proof}

 \end{appendix}
\end{document}